\documentclass[10pt,twoside,a4paper]{article}
\usepackage[utf8]{inputenc}
\usepackage{authblk}
\usepackage[T1]{fontenc}
\usepackage{lmodern}
\usepackage{amssymb}
\usepackage{amsmath}
\usepackage{enumerate}
\usepackage{subfigure}

\numberwithin{equation}{section}
\usepackage{microtype}
\usepackage{multirow}

\usepackage{graphicx}
\graphicspath{{images/}}

\usepackage[pdftex,colorlinks=true,linkcolor=blue,citecolor=green,urlcolor=blue,bookmarks=true,bookmarksnumbered=true]{hyperref}
\hypersetup
{
    pdfauthor={Aspri-et-al},
    pdfsubject={Num. Anal.},
    pdftitle={Data driven regularization},
    pdfkeywords={Keywords}
}

\usepackage[hyperref,amsmath,thmmarks]{ntheorem}
\usepackage{cleveref}

\date{March 6, 2020}

\usepackage{aliascnt}

\newtheorem{lemma}{Lemma}[section]

\newtheorem{remark}{Remark}[section]

\newaliascnt{proposition}{lemma}
\newtheorem{proposition}[proposition]{Proposition}
\aliascntresetthe{proposition}

\newaliascnt{corollary}{lemma}
\newtheorem{corollary}[corollary]{Corollary}
\aliascntresetthe{corollary}

\newaliascnt{theorem}{lemma}
\newtheorem{theorem}[theorem]{Theorem}
\aliascntresetthe{theorem}

\theorembodyfont{\normalfont}
\newaliascnt{definition}{lemma}

\aliascntresetthe{definition}

\newaliascnt{assumption}{lemma}
\newtheorem{assumption}[assumption]{Assumption}
\aliascntresetthe{assumption}

\theoremstyle{nonumberplain}
\theoremseparator{:}
\theoremheaderfont{\normalfont\itshape}

\theoremsymbol{\ensuremath{\square}}
\newtheorem{proof}{Proof}

\usepackage[a4paper,centering,bindingoffset=0cm,marginpar=2cm,margin=2.5cm]{geometry}

\usepackage[pagestyles]{titlesec}
\titleformat{\section}[block]{\large\sc\filcenter}{\thesection.}{0.5ex}{}[]
\titleformat{\subsection}[runin]{\bf}{\thesubsection.}{0.5ex}{}[.]

\newpagestyle{headers}%
{%
	\headrule
	\sethead%
		[\footnotesize\thepage]%
		[\footnotesize\sc Aspri et al.]%
		[]%
		{}%
		{\footnotesize\sc A data-driven iteratively regularized Landweber iteration}%
		{\footnotesize\thepage}
	\setfoot{}{}{}
}
\usepackage[font=footnotesize,format=plain,labelfont=sc,textfont=sl,width=0.75\textwidth,labelsep=period]{caption}

\usepackage{dsfont}
\newcommand{\N}{\mathbb{N}}

\let\RE\Re
\let\Re=\undefined
\DeclareMathOperator{\Re}{\RE e}
\let\IM\Im
\let\Im=\undefined
\DeclareMathOperator{\Im}{\IM m}

\newcommand{\norm}[1]{\left\|#1\right\|}

\newcommand{\inner}[2]{\left<#1,#2\right>}

\newcommand{\e}{\mathrm e}

\newcommand{\yd}{y^\delta}

\newcommand{\uk}{u_{k}}
\newcommand{\uks}{u_{k_{*}}}
\newcommand{\ukp}{u_{k+1}}
\newcommand{\udag}{u^\dagger}

\begin{document}

\title{A data-driven iteratively regularized Landweber iteration}
\author[2]{A. Aspri}
\author[3]{S. Banert}
\author[3]{O. Öktem}
\author[1,2]{O. Scherzer}
\affil[1]{{\small Computational Science Center, University of Vienna, \texttt{otmar.scherzer@univie.ac.at}}}	
\affil[2]{{\small Johann Radon Institute for Computational and Applied Mathematics (RICAM), Linz, \texttt{ andrea.aspri@ricam.oeaw.ac.at}}}
\affil[3]{{\small KTH Royal Institute of Technology, Stockholm, \texttt{\{banert,ozan\}@kth.se}}}
\maketitle

\begin{abstract}
We derive and analyse a new variant of the iteratively regularized Landweber iteration, for solving linear and nonlinear 
ill-posed inverse problems. The method takes into account training 
data, which are used to estimate the interior of a black box, which is used to define the iteration process.  
We prove convergence and stability for the scheme in infinite dimensional Hilbert spaces. 
These theoretical results are complemented by some numerical experiments for solving linear inverse problems for the Radon 
transform and a nonlinear inverse problem for Schlieren tomography.
\end{abstract}
\vskip.15cm

{\sf Keywords.} Iteratively regularized Landweber iteration, expert and data driven regularization, black box strategy. 
\vskip.15cm

{\sf 2010 AMS subject classifications.}
65J20  (65J10, 65J15, 65J22)

\section{Introduction}
This paper is concerned with a generalization of the \emph{iteratively regularized Landweber iteration}, as introduced in \cite{Sch98}, 
for solving (linear and nonlinear) ill--posed operator equations
\begin{equation} \label{eq:operator}
F(u) = y.
\end{equation}
In the course of this paper, for the sake of simplicity of presentation, we restrict our attention to an operator 
$F:\mathcal{D}(F) \subset X \to Y$ between real separable Hilbert spaces $X$ and $Y$ with inner products $\langle\cdot,\cdot\rangle$ 
and norms $\|\cdot\|$, respectively. 
We denote with $\yd$ noisy data and we assume that 
\begin{equation*}
\norm{\yd-y}\leq \delta.
\end{equation*} 
Generalizations to the Banach space setting (see \cite{SchuKalHofKaz12}) are formally similar, but technically more 
complicated, and thus omitted here.

In the original form, the iteratively regularized Landweber iteration, as considered in \cite{Sch98} 
(see also \cite{KalNeuSch08}), consists in computing the following iterative updates
\begin{equation}\label{eq:ir_landweber_iter}		
\ukp := \uk - F'(\uk)^*\bigl( F(\uk)-\yd \bigr) - \lambda_k(\uk - u^{(0)}),\qquad k\in\mathbb{N},
\end{equation} 
where $u^{(0)}$ is an initial guess which incorporates \emph{a-priori knowledge} on the solution to be recovered. 
In presence of noise in the data (that is the available data are $\yd$), to guarantee that 
the iterative scheme \eqref{eq:ir_landweber_iter} is a regularization procedure, it has to be complemented with 
a stopping rule. The discrepancy principle is often employed, i.e., the iteration is stopped 
after the first $k_*=k_*(\delta,\yd)$ steps for which:
\begin{equation*}
\norm{F(u_{k_*})-\yd}\leq \tau\delta < \norm{F(\uk)-\yd}, \qquad 0\leq k< k_*, 
\end{equation*}  
for some $\tau>1$. 
Method \eqref{eq:ir_landweber_iter} can be considered a modification of the Landweber iteration, i.e, 
\begin{equation}\label{eq: landweber_iter}		
\ukp := \uk - F'(\uk)^*\bigl( F(\uk)-\yd \bigr),\qquad k\in\mathbb{N},
\end{equation} 
when we put $\lambda_k=0$ for all $k \in \N$.
However, opposed to the Landweber iteration \eqref{eq: landweber_iter}, the convergence rates 
analysis of \eqref{eq:ir_landweber_iter} is indeed simpler, and requires less restrictive conditions on the 
operator $F$ (see \cite{KalNeuSch08}). 
The reason for that is the presence of the damping term $\lambda_k(\uk - u^{(0)})$ which stabilizes the algorithm: 
The modified Landweber iteration converges to a solution which is closest to $u^{(0)}$, which is not guaranteed 
for the Landweber iteration itself without posing additional assumptions. However, the practical convergence rates 
are even slower than for the Landweber iteration. 
The damping term $\lambda_k(\uk-u^{(0)})$  was originally introduced as an additional stabilizing term in the
Gauß-Newton's method in \cite{Bak92}. Later, it has been observed that it has a similar effect in all the iterative 
regularization methods when used as an add-on factor.
This observation is the motivation for this paper to introduce a data driven damping factor in the Landweber iteration. 
We first note that the iteration \eqref{eq:ir_landweber_iter} can be rewritten in the form
\begin{equation}\label{eq:ir_landweber_iter_opt} 
\ukp := \uk - \tfrac{1}{2}\partial_u \left( \norm{F(u)-\yd}^2 + \lambda_k \norm{u-u^{(0)}}^2 \right)(u_k),\qquad k\in\mathbb{N}.
\end{equation}
Our objective is to include a damping factor gained from \emph{expert data} $(u^{(i)},F(u^{(i)}))_{1 \leq i \leq n}$ 
in the iteration process. As a first attempt, one could generalize \eqref{eq:ir_landweber_iter_opt} by 
ignoring the image data $(F(u^{(i)})_{1 \leq i \leq n})$, which leads to 
\begin{equation} \label{eq:ir_landweber_iter_opt_II} 
\ukp := \uk - \tfrac{1}{2}\partial_u \left( \norm{F(u)-\yd}^2 + \lambda_k \sum_{i=1}^n \norm{u-u^{(i)}}^2 \right)(u_k),\qquad k\in\mathbb{N}.
\end{equation}
From an analytical point of view, this does not offer significant benefits because we expect that in general it will converge to 
the solution which is closest to the mean $\hat{u}=\sum_{i=1}^n u^{(i)}$. 
Therefore, in order to include the image data $(F(u^{(i)})_{1 \leq i \leq n})$ as prior information as well, we follow a \emph{black box strategy}. To be more specific, we identify an operator $\hat{A}$, which maps each $u^{(i)}$ to $F(u^{(i)})$, $i=1,\ldots,n$, and vice versa, to include as a damping term in the Landweber iteration \eqref{eq: landweber_iter}, i.e.,
\begin{equation}\label{SDschemeNew}
\ukp := \uk - \tfrac{1}{2}\partial_u \left( \norm{F(u)-\yd}^2 + {\lambda^{\delta}_k} \norm{\hat{A}(u)-\yd}^2\right)(u_k), \qquad k\in\N,
\end{equation}
which, in explicit form, is equal to  
\begin{equation}\label{SDschemeNew_II}
\ukp := \uk - F'(\uk)^*\bigl( F(\uk)-\yd \bigr) - {\lambda^{\delta}_k} \hat{A}'(\uk) ^*\bigl(\hat{A}(\uk) -\yd \bigr), \qquad k\in\N.
\end{equation}
The objective of the second term in \eqref{SDschemeNew_II} is to give some bias for the expert data.
We mention that \emph{system identification} and \emph{black box} theory was considered extensively in the sixties of the last century, see 
for instance \cite{Pap62}. 
Our main results concern the proof of strong convergence and stability for scheme \eqref{SDschemeNew_II}, essentially under the assumptions 
that $F$ satisfies the usual tangential cone condition, see \eqref{eq:tcc}, and both $F$ and $\hat{A}$ are Fr\'echet differentiable. 
These theoretical results are followed by some numerical experiments for linear and nonlinear operator equations \eqref{eq:operator}.
To be more precise, we take the Radon (see \cite{Kuc14}) and the Schlieren (see \cite{SchGraGroHalLen09}) operators  as models of linear 
and nonlinear problems, respectively. 
Concerning the operator $\hat{A}$, we do not tackle the problem in the full generality in the sense that 
all the numerical experiments are realized by restricting attention to bounded linear operators between separable Hilbert spaces $X$ and $Y$, that 
is, taking $\tilde{A}\in\mathcal{B}(X,Y)$, we consider 
\begin{equation}\label{SDschemeNew_III}
\ukp := \uk - F'(\uk)^*\bigl( F(\uk)-\yd \bigr) - {\lambda^{\delta}_k} \tilde{A}^*\bigl(\tilde{A}\uk-\yd \bigr), \qquad k\in\N.
\end{equation}
We reserve to point the general case of nonlinear operators for future work. 
Specifically, given $n$ pairs of input-output relation $(u^{(i)},F(u^{(i)}))$, and defining the functional 
\begin{equation*}
l(\tilde{T}):=\frac{1}{2} \sum_{i=1}^n \norm{\tilde{T} u^{(i)}-y^{(i)}}^2_{Y},
\end{equation*}
where $\tilde{T}\in B(X,Y)$ and {$y^{(i)}=F(u^{(i)})$}, we define $\tilde{A}:X \to Y$ as the bounded linear operator which satisfies
\begin{equation}\label{eq: min_A_intro}
l(\tilde{A})=\underset{\tilde{T}\in\mathcal{B}(X,Y)}{\min} l(\tilde{T})=\underset{\tilde{T}\in\mathcal{B}(X,Y)}{\min} \left[\frac{1}{2} \sum_{i=1}^n \norm{\tilde{T} u^{(i)}-y^{(i)}}^2_{Y}\right].
\end{equation}	
The operator $\tilde{A}$ can be {projected onto a finite dimensional subspace and here represented by} a matrix $A\in\mathbb{R}^{M\times N}$, {where $M, N\in\mathbb{N}$ will be defined in Section \ref{sec: num_exp}}. In fact, from classical results of functional analysis, every bounded linear operator has a matrix representation of infinite dimensions, through complete orthonormal bases in $X$ and $Y$, see for example \cite{Kat95}. 
We refer the reader to Section \ref{sec: num_exp} for a more in-depth discussion.
In all the numerical experiments we compare the outcomes of the {data driven} iteratively regularized Landweber scheme \eqref{SDschemeNew_III} with {the classic iteratively regularized Landweber scheme \eqref{eq:ir_landweber_iter} and} the 
Landweber iteration \eqref{eq: landweber_iter}.
This is a first attempt to include expert data in iterative reconstruction algorithms. Nonlinear learning strategies might behave 
significantly better. However, for the solution of linear ill--posed operator equations, the numerical reconstructions of 
\eqref{SDschemeNew_III} are significantly better than the results of Landweber iterations \eqref{eq:ir_landweber_iter} and \eqref{eq: landweber_iter}. 
In the case of nonlinear operator equations this approach is not as well suited, because we attempt to simulate 
a nonlinear input-output relation by a linear operator $\tilde{A}$. This is a challenging topic to investigate for future research.

The outline of the paper is as follows: In Section \ref{sec:LI} we analyze \eqref{SDschemeNew_II} in an infinite dimensional Hilbert 
space setting and prove strong convergence and stability. In Section \ref{sec: num_exp} we discuss the numerical implementation of \eqref{SDschemeNew_III} and 
study applications to the Radon inversion, for the linear case, and to the Schlieren tomography, for the nonlinear case, with full and limited data. In all the numerical examples we compare the outcomes of \eqref{eq:ir_landweber_iter}, \eqref{eq: landweber_iter} and \eqref{SDschemeNew_III}.

 \section{Iteratively Regularized Landweber Iteration}\label{sec:LI}
In this section we analyze the convergence of the iteratively regularized Landweber iteration, presented in the introduction, which we recall here for the reader's sake
\begin{equation}\label{eq:SDscheme}
\begin{aligned}
\ukp &:= \uk - F'(\uk)^*\bigl( F(\uk)-\yd \bigr) - {\lambda^{\delta}_k} \hat{A}'(\uk) ^*\bigl(\hat{A}(\uk) -\yd \bigr), \qquad k\in\N,\\
u_0 &- \textrm{initial guess},
\end{aligned}
\end{equation}
where $\yd$ are noisy data such that 
\begin{equation*}
\norm{\yd-y}\leq \delta,
\end{equation*}
and ${\lambda^{\delta}_k}$ is a suitable parameter.	
Moreover, we assume that \eqref{eq:operator} has a solution $\udag$, possibly not unique.

For the analysis, we basically follow the approach proposed in \cite[Chapters 1, 3]{KalNeuSch08} because such iterates can be considered a modified version of the Landweber scheme.\\
It is well-known that for iterative schemes and in case of noisy data, a stopping criterion has to be added to the scheme \eqref{eq:SDscheme} in order to obtain a regularization method. For that, we employ the discrepancy principle for which the iteration is stopped after $k_*=k_*(\delta,\yd)$ steps once 
\begin{equation}\label{eq: discrep_princ}
\norm{F(\uks)-\yd}\leq \tau\delta\leq \norm{F(\uk)-\yd},\qquad 0\leq k< k_*,
\end{equation}  
where $\tau$ is a positive number chosen suitably.	

First, we introduce some assumptions on the operators $F$ and $\hat{A}$ to guarantee local convergence of the iterative schemes as \eqref{eq:SDscheme}, and we define some constants which will appear in the convergence theorem. 
\begin{assumption}\label{as:fa}
	Let $\mathcal{B}_\rho(\udag) \subseteq \mathcal{B}_{2\rho}(u_0)$, where $\mathcal{B}_\rho(\udag)$ denotes a closed ball of radius $\rho$ and center $\udag$.
	We assume that
	\begin{enumerate}
		\item $F$ satisfies the \emph{tangential cone condition}, that is 
		\begin{equation}\label{eq:tcc}
		\norm{F(u) - F(v) - F'(u)(u-v)} \leq \nu \norm{F(u) - F(v)}, \qquad         \forall u,v \in \mathcal{B}_{\rho}(\udag),
		\end{equation}
		where $\nu>0$.		 	
		\item $F$ and $\hat{A}$ have continuous Fr\'echet derivative $F'$ and $\hat{A}'$, respectively, with Lipschitz constants $L_F$ and $L_{\hat{A}}$, i.e., for all $u \in \mathcal{B}_\rho(\udag)$ it holds
		\begin{subequations}\label{eq:lc}  	
			\begin{equation}\label{eq:lcF}
			\norm{F'(u)} \leq L_F,
			\end{equation}     
			\begin{equation}\label{eq:lcA}
			\norm{\hat{A}'(u)} \leq L_{\hat{A}}.  
			\end{equation}
		\end{subequations}
		\item	We assume that data-driven model $\hat{A}$ cannot fully explain the model for the true data, hence
		\begin{equation}\label{eq:non-attain}
		\norm{\hat{A}(\udag)-y} \geq C_N \text{ with } C_N > 0\;.
		\end{equation}
	\end{enumerate}
\end{assumption}
Here, we present an auxiliary estimate which will be needed to prove a monotonicity property and some convergence results of \eqref{eq:SDscheme}.
\begin{lemma} \label{le:aux} Let Assumption \ref{as:fa} be satisfied. 
	Then there exists a positive constant $C^{\delta}_{\hat{A}}$ such that
	\begin{equation} \label{eq:A_stab}
	\norm{\hat{A}(\uk)-\yd} \leq  C^{\delta}_{\hat{A}}\;
	\end{equation}
	holds true for every $u_k\in \mathcal{B}_\rho(\udag)$.		
\end{lemma}
\begin{proof} 
	Equation \eqref{eq:A_stab} follows from the triangle inequality and the Lipschitz-continuity of $\hat{A}'$, in fact
	\begin{equation*}
	\begin{aligned}
	\norm{\hat{A}(\uk)-\yd} \leq & \norm{\hat{A}(\uk)-\hat{A}(\udag)} + \norm{\hat{A}(\udag)} + \norm{\yd}\\
	= & \norm{\int_0^1 \hat{A}'(\udag + t(\uk-\udag))(\uk-\udag)\,dt} + \norm{\hat{A}(\udag)} + \norm{\yd}\\
	\leq & L_{\hat{A}}\rho + \norm{\hat{A}(\udag)}+\norm{\yd} =: C^{\delta}_{\hat{A}}\;.
	\end{aligned}
	\end{equation*}
\end{proof}

We first prove that, under some suitable assumptions on ${\lambda^{\delta}_k}$, a monotonicity property is verified for the scheme \eqref{eq:SDscheme}.
\begin{proposition} \label{le:mon} 
	Let Assumption \ref{as:fa} be satisfied. Assume that 
	$\uk \in \mathcal{B}_\rho(\udag)$ and
	\begin{equation}\label{eq:disc}
	\norm{F(\uk)-\yd}  \geq \tau \delta, 
	\end{equation}
	where $\tau >0$ and it is such that
	\begin{equation}\label{eq:disc_ca}
	C_\tau := 1 - L_F^2 -\nu - \frac{1+\nu}{\tau} \geq 0\;.
	\end{equation}
	Moreover, let us assume that there exists a positive constant $C^{\delta}_\lambda$ such that the following conditions hold:
	\begin{subequations}\label{eq:lambda_ca}
		\begin{equation}\label{eq:lambda_ca1}
		{\lambda^{\delta}_k}L_{\hat{A}} C^{\delta}_{\hat{A}}\leq \rho,
		\end{equation}
		\begin{equation}\label{eq:lambda_ca2}
		{\lambda^{\delta}_k} \leq   C^{\delta}_\lambda \norm{F(\uk)-\yd}^2,
		\end{equation}
	\end{subequations}
	\begin{equation}\label{eq:c_tau}
	C_\tau - 2L_{\hat{A}} C^{\delta}_{\hat{A}} C^{\delta}_\lambda \rho> 0.
	\end{equation}
	Then $\ukp \in \mathcal{B}_{\rho}(\udag)$ and $\norm{\ukp-\udag}\leq\norm{\uk-\udag}$.  
\end{proposition}
\begin{proof}
	Let us assume that $\uk \in \mathcal{B}_\rho(\udag)$, then we have 
	\begin{equation}	\label{eq:mono_h}
	\begin{aligned}
	\norm{\ukp-\udag}^2 &- \norm{\uk-\udag}^2 \\
	\leq & 2\Big(- \inner{F(\uk)-\yd}{F'(\uk)(\uk-\udag)} +  \norm{F'(\uk)^*(F(\uk)-\yd)}^2\\
	& - {\lambda^{\delta}_k} \inner{\hat{A}(\uk)-\yd}{\hat{A}'(\uk)(\uk-\udag)} +  {(\lambda^{\delta}_k)}^2 \norm{\hat{A}'(\uk)^*(\hat{A}(\uk)-\yd)}^2\Big)\\
	\leq & 2\Big(-\norm{F(\uk)-\yd}^2 + \inner{F(\uk)-\yd}{F(\uk)-y - F'(\uk)(\uk-\udag)} \\ 
	& + \delta \norm{F(\uk)-\yd} + \norm{F'(\uk)^*(F(\uk)-\yd)}^2\\
	& - {\lambda^{\delta}_k} \inner{\hat{A}(\uk)-\yd}{\hat{A}'(\uk)(\uk-\udag)} + {(\lambda^{\delta}_k)}^2 \norm{\hat{A}'(\uk)^*(\hat{A}(\uk)-\yd)}^2\Big)=:2(T_F+T_{\hat{A}}), \\
	\end{aligned}
	\end{equation}
	where $T_F$ and $T_{\hat{A}}$ represent the terms in $F$ and $\hat{A}$, respectively, in the right hand side of the second inequality.	
	Next, we study $T_F$ and $T_{\hat{A}}$ separately, providing for both of them a bound in terms of the square norm of the residual of $F$ at the $k$-iteration.  
	By using \eqref{eq:tcc}, \eqref{eq:lcF}, \eqref{eq:disc} and \eqref{eq:disc_ca} we get
	\begin{equation}\label{eq:mono_h1}
	\begin{aligned}
	T_F =&  - \norm{F(\uk)-\yd}^2 + \inner{F(\uk)-\yd}{F(\uk)-y - F'(\uk)(\uk-\udag)} \\ 
	& \qquad +  \delta \norm{F(\uk)-\yd} +  \norm{F'(\uk)^*(F(\uk)-\yd)}^2\\
	\leq &  - \norm{F(\uk)-\yd}^2 + \nu \inner{F(\uk)-\yd}{F(\uk)-y} \\
	&\qquad  + \delta \norm{F(\uk)-\yd} + L_F^2 \norm{F(\uk)-\yd}^2\\
	\leq & \norm{F(\uk)-\yd} \left( -(1 - L_F^2-\nu) \norm{F(\uk)-\yd} + \left(1 + \nu \right) \delta \right)\\
	\leq & - C_\tau \norm{F(\uk)-\yd}^2\,.
	\end{aligned}
	\end{equation}
	On the contrary, to estimate the term $T_{\hat{A}}$, we utilize \eqref{eq:lcA}, \eqref{eq:A_stab}, \eqref{eq:lambda_ca1} and \eqref{eq:lambda_ca2} from what it follows
	\begin{equation}
	\label{eq:mono_h2a}
	\begin{aligned}
	T_{\hat{A}}=& -  {\lambda^{\delta}_k} \inner{\hat{A}(\uk)-\yd}{\hat{A}'(\uk)(\uk-\udag)} +  {(\lambda^{\delta}_k)}^2 \norm{\hat{A}'(\uk)^*(\hat{A}(\uk)-\yd)}^2\\
	\leq & {\lambda^{\delta}_k}L_{\hat{A}} \norm{\hat{A}(\uk)-\yd}\left(\rho+ {\lambda^{\delta}_k} L_{\hat{A}} \norm{{\hat{A}}(\uk)-\yd}\right)   \\
	\leq&  2{\lambda^{\delta}_k} L_{\hat{A}} C^{\delta}_{\hat{A}} \rho\leq 2 L_{\hat{A}} C^{\delta}_{\hat{A}} C^{\delta}_{\lambda}\rho \norm{F(\uk)-\yd}^2. 
	\end{aligned}
	\end{equation}
	Finally, by inserting \eqref{eq:mono_h2a} and \eqref{eq:mono_h1} into \eqref{eq:mono_h}, we get
	\begin{equation}\label{eq: diff_ukp_uk}
	\norm{\ukp-\udag}^2 - \norm{\uk-\udag}^2 \leq -2 \left(C_\tau - 2L_{\hat{A}} C^{\delta}_{\hat{A}} C^{\delta}_\lambda \rho \right) \norm{F(\uk)-\yd}^2
	\end{equation}
	which gives the assertion of the theorem thanks to \eqref{eq:c_tau}.
\end{proof}
\begin{remark}
	We note that \eqref{eq:lambda_ca1} is asymptotically always satisfied if the iteration is convergent and the damping 
	parameters \eqref{eq:lambda_ca} are chosen sufficiently small. Equation \ref{eq:c_tau} requires that the operator $F$ is not too 
	non-linear and that the iterates are close to the true solution. 
\end{remark}

To prove a convergence result for the iteratively regularized Landweber iteration \eqref{eq:SDscheme}, we first need to show that there exists a finite stopping index $k_*$ for which the discrepancy principle \eqref{eq: discrep_princ} holds.
\begin{corollary}
	Under the assumptions of Proposition \ref{le:mon}, let $k_*$ be the first index such that the discrepancy principle \eqref{eq: discrep_princ} is verified.
	Then 
	\begin{equation}\label{eq: res_ks}
	k_*(\tau\delta)^2\leq \sum_{k=0}^{k_*-1}\norm{F(\uk)-\yd}^2\leq \frac{1}{2(C_\tau-2 L_{\hat{A}}  C^{\delta}_{\hat{A}} C^{\delta}_\lambda\rho)}\norm{u_0 - \udag}.
	\end{equation}
	In particular, when $\delta=0$, we have 
	\begin{equation}\label{eq: res_ks_no_noise}
	\sum_{k=0}^{\infty}\norm{F(\uk)-y}^2<\infty.
	\end{equation}	
\end{corollary}
\begin{proof}
	We sum the inequality \eqref{eq: diff_ukp_uk} from 0 to the step $k_*-1$, that is
	\begin{equation*}
	\sum_{k=0}^{k_*-1} \Big[ \norm{\ukp-\udag}^2 - \norm{\uk-\udag}^2 \Big]\leq -2(C_\tau - 2L_A C^{\delta}_A C^{\delta}_\lambda \rho) \sum_{k=0}^{k_*-1}\norm{F(\uk)-\yd}^2,
	\end{equation*}	
	hence
	\begin{equation}\label{eq: sum 0 ks}
	\norm{u_0-\udag}^2 - \norm{\uks-\udag}^2\geq 2(C_\tau - 2L_{\hat{A}} C^{\delta}_{\hat{A}} C^{\delta}_\lambda \rho) \sum_{k=0}^{k_*-1}\norm{F(\uk)-\yd}^2.
	\end{equation}
	Using the fact that $ \norm{u_0-\udag}^2 - \norm{\uks-\udag}^2\leq  \norm{u_0-\udag}^2$ and the discrepancy principle \eqref{eq: discrep_princ}, from  \eqref{eq: sum 0 ks}	it is straightforward to obtain  \eqref{eq: res_ks}. Moreover, when $\delta=0$, we can assume $\tau$ arbitrarily large such that $C_\tau - 2L_{\hat{A}} C^0_{\hat{A}} C^0_\lambda \rho\overset{\tau\to\infty}{\longrightarrow} 1-\nu-L^2_F - 2 L_{\hat{A}} C^0_{\hat{A}} C^0_\lambda \rho$ and $k_*\to \infty$. Therefore \eqref{eq: res_ks} reduces to 
	\begin{equation*}
	\sum_{k=0}^{\infty}\norm{F(\uk)-y}^2\leq \frac{1}{2(1-\nu-L^2_F - 2 L_{\hat{A}} C^0_{\hat{A}} C^0_\lambda \rho)}\norm{u_0 - \udag},
	\end{equation*}	
	which gives the assertion of the theorem.		 	
\end{proof}
Next, before proving the convergence result in case of noisy data, we show that, when $\delta=0$, the residual norm of the modified Landweber iteration \eqref{eq:SDscheme} goes to zero as $k$ tends to infinity. This means that if the iterations converge, then the limit is a solution of \eqref{eq:operator}.
\begin{theorem}\label{th: conv_no_noise}
	Let Assumption \ref{as:fa}, \eqref{eq:A_stab} and \eqref{eq:lambda_ca2} hold, with $\delta=0$. Then the iteratively regularized Landweber iteration \eqref{eq:SDscheme}, applied to the exact data $y$, converges to a solution $\udag$ of $F(u)=y$.
\end{theorem}
\begin{proof}
	Let us define 
	\begin{equation}\label{eq: def_ek}
	e_k=u_k-\udag.
	\end{equation}	
	Given $j\geq k$, we choose an integer $h$ such that $k\leq h\leq j$ and 
	\begin{equation}\label{eq: Fuh_y}
	\norm{F(u_h)-y}\leq \norm{F(u_i)-y},\qquad \forall\, k\leq i\leq j.
	\end{equation}
	Then, we consider 
	\begin{equation}\label{eq: norm_ej_ek}
	\norm{e_j-e_k}\leq \norm{e_j-e_k}+\norm{e_h-e_k}.
	\end{equation}
	We only analyse the first term in \eqref{eq: norm_ej_ek} since the other one can be treated in the same way.	So, we have
	\begin{equation}\label{eq: ejmek}
	\norm{e_j-e_h}^2=\langle e_j-e_h, e_j-e_h \rangle
	=2\langle e_h-e_j,e_h\rangle + \norm{e_j}^2-\norm{e_h}^2.   
	\end{equation}	 	
	We only need to study the first term in the right-hand side of the previous equality since $\norm{e_j}^2$ and $\norm{e_h}^2$ are monotonically decreasing to some $\varepsilon \geq 0$, thanks to Proposition \ref{le:mon} (even in the case $\delta=0$), hence $\norm{e_j}^2-\norm{e_h}^2$ converges to $0$ as $k\to\infty$. We show that $|\langle e_h-e_k, e_h \rangle |\to 0$ as $k\to +\infty$. 
	Using the definition of $\e_j, e_h$, see \eqref{eq: def_ek}, the iteration \eqref{eq:SDscheme} and the fact that $j\geq h$, we find
	\begin{equation}\label{eq: S_A+S_F}
	\begin{aligned}
	|\langle e_h-e_j, e_h \rangle | &= |\langle u_j-u_h, e_h \rangle | \\
	&\leq \Big| \Big\langle \sum_{i=h}^{j-1} F'(u_i)^* (y-F(u_i))+\lambda_i \hat{A}'(u_i)^*(y-\hat{A}(u_i)), e_h \Big\rangle \Big|\\
	&\leq \sum_{i=h}^{j-1}\Big| \Big\langle y-F(u_i), F'(u_i)(u_h-\udag) \Big \rangle \Big| + \sum_{i=h}^{j-1}\lambda_i\Big| \Big\langle y-\hat{A}(u_i), \hat{A}'(u_i)(u_h-\udag) \Big \rangle \Big|\\
	&=: S_F + S_{\hat{A}} 
	\end{aligned}
	\end{equation} 
	We study the two terms $S_F$ and $S_{\hat{A}}$ separately.\\
	\textit{Term $S_F$:} the analysis can be found, for example, in \cite{KalNeuSch08}. For the sake of completeness, here we summarize the principal steps to obtain an estimate of $S_F$. Adding and subtracting $u_i$ {and $y-F(u_i)$} in $S_F$, we find
	\begin{equation*}
	S_F\leq \sum_{i=h}^{j-1} \left[\norm{y-F(u_i)}\left(\norm{y-F(u_i)-F'(u_i)(\udag-u_i)} +\norm{y-F(u_h)}+\norm{F(u_i)-F(u_h)-F'(u_i)(u_i-u_h)}\right)\right],
	\end{equation*}
	hence, applying the tangential cone condition \eqref{eq:tcc}
	\begin{equation*}
	S_F\leq 2\nu \sum_{i=h}^{j-1} \norm{y-F(u_i)}^2 +(1+\nu) \sum_{i=h}^{j-1} \norm{y-F(u_i)} \norm{y-F(u_h)}.
	\end{equation*}	
	From \eqref{eq: Fuh_y}, we finally find
	\begin{equation}\label{eq: S_F}
	S_F\leq (1+3\nu) \sum_{i=h}^{j-1} \norm{y-F(u_i)}^2.
	\end{equation}
	\textit{Term $S_A$}: Utilizing the hypotheses \eqref{eq:A_stab} and \eqref{eq:lambda_ca2}, with $\delta=0$, we have
	\begin{equation}\label{eq: S_A}
	\begin{aligned}
	S_{\hat{A}}\leq \sum_{i=h}^{j-1}\lambda_i \Big| \Big\langle y-{\hat{A}}(u_i), \hat{A}'(u_i)(u_h-\udag)\Big\rangle \Big| &\leq \sum_{i=h}^{j-1}\lambda_i \norm{ y-\hat{A}(u_i)} \norm{\hat{A}'(u_i)(u_h-\udag)}\\
	&\leq L_{\hat{A}} C^0_{\hat{A}}\,\rho \sum_{i=h}^{j-1}\lambda_i\leq L_{\hat{A}} C^0_{\hat{A}} C^0_{\lambda}\, \rho \sum_{i=h}^{j-1} \norm{F(u_i)-y}^2   
	\end{aligned}  
	\end{equation}
	Inequalities \eqref{eq: S_A} and \eqref{eq: S_F}, inserted in Equation \eqref{eq: S_A+S_F}, give
	\begin{equation}\label{eq: est_ehmej}
	|\langle e_h-e_j, e_h \rangle |\leq (1+3\nu+ L_{\hat{A}} C^0_{\hat{A}} C^0_{\lambda}\, \rho)\sum_{i=h}^{j-1}\norm{y-F(u_i)}^2.
	\end{equation}
	Reasoning in the same way for $|\langle e_h-e_k, e_h \rangle |$ we find
	\begin{equation}\label{eq: est_ehmek}
	|\langle e_h-e_k, e_h \rangle |\leq (1+3\nu+ L_{\hat{A}} C^0_{\hat{A}} C^0_{\lambda}\, \rho)\sum_{i=k}^{h-1}\norm{y-F(u_i)}^2.
	\end{equation}
	From these estimates, it follows that both $|\langle e_h-e_j, e_h \rangle |$ and $|\langle e_h-e_k, e_h \rangle |$ go to zero as $k\to \infty$ thanks to \eqref{eq: res_ks_no_noise}.
	Therefore, from \eqref{eq: ejmek} and \eqref{eq: norm_ej_ek}, we find that $\{e_k\}$ is a Cauchy sequence. 
	Then, from \eqref{eq: def_ek}, we derive that $\{u_k\}$ is a Cauchy sequence as well hence the assertion of the theorem follows. 
\end{proof}
In the case of noisy data, we cannot expect that the iteratively regularized Landweber iteration converges, since $\yd$ might not belong to the range of the operator $F$. In fact, we can only obtain a stable approximation of a solution of $F(u)=y$ provided the discrepancy principle \eqref{eq: discrep_princ} is employed, i.e, the iteration is stopped after a finite number of steps. 
Here, assuming that the discrepancy principle \eqref{eq: discrep_princ} holds, we prove a stability result for the iteratively regularized Landweber iteration in the case of noisy data.
\begin{theorem}
	Under Assumption \ref{as:fa}, \eqref{eq:disc_ca} and \eqref{eq:lambda_ca}, let $k_*=k_*(\delta,\yd)$ be chosen such that \eqref{eq: discrep_princ} holds. {Moreover, we assume that for each fixed iteration $k$, ${\lambda^{\delta}_k}\to \lambda_k$, as $\delta\to 0$}. Then, the data driven iteratively regularized Landweber iteration \eqref{eq:SDscheme} converges to a solution of $F(u)=y$, as $\delta\to 0$.   	
\end{theorem} 
\begin{proof}
	We recall that $\udag$ represents the limit of the iteration \eqref{eq:SDscheme} when $\delta=0$, i.e, the case with data $y$. Let $\{\delta_n\}$ be a sequence such that $\{\delta_n\}\to 0$ as $n\to\infty$ and denote with $y_n:=y^{\delta_n}$ the sequence of the perturbed data. Let $k_n=k_*(\delta_n,y_n)$ be the stopping index for which the discrepancy principle \eqref{eq: discrep_princ} holds, i.e. 
	\begin{equation}\label{eq: discr_princ_yn}
	\norm{F(u_{k_n}^{\delta_n})-y_n}\leq \tau \delta_n.
	\end{equation} 
	We now distinguish two cases as $n\to\infty$:
	\begin{enumerate}
		\item $k_n\to k,\, k\in \mathbb{R}^+$;
		\item $k_n \to +\infty$.
	\end{enumerate}
	\textit{Case (i)}: to avoid technicalities, we can assume that $k_n=k$, for all $n\in\mathbb{N}$. Therefore, from \eqref{eq: discr_princ_yn} we get
	\begin{equation*}
	\norm{F(u_{k}^{\delta_n})-y_n}\leq \tau \delta_n.
	\end{equation*}	 
	From the continuity hypotheses on $F$, $F'$ and $\hat{A}$, $\hat{A}'$ and recalling \eqref{eq:lambda_ca2}, we find that
	\begin{equation*}
	u^{\delta_n}_k \to u_k,\quad F(u^{\delta_n}_k)\to F(u_k)=y,\quad \hat{A}(u^{\delta_n}_k)\to \hat{A}(u_k),\quad \lambda^{\delta_n}_k\to 0,\quad \text{as}\, n\to \infty.
	\end{equation*}
	This means that the $k-$th iterate of \eqref{eq:SDscheme} is a solution of $F(u)=y$ and thus the iteration terminates with $\udag=u_k$.\\
	\textit{Case (ii)}: We choose $k$ sufficiently large such that $k_n> k$ and by Proposition \ref{le:mon} we have
	\begin{equation}\label{eq: udeltakn_us}
	\norm{u^{\delta_n}_{k_n}-\udag}\leq \norm{u^{\delta_n}_{k}-\udag}\leq \norm{u^{\delta_n}_{k}-u_k}+\norm{u_k-\udag}
	\end{equation} 
	Given $\epsilon>0$, from Theorem \ref{th: conv_no_noise} { there exists  $\bar{k}=\bar{k}(\epsilon)$ such that 
		\begin{equation}\label{eq: uk_us_eps}
		\norm{u_k-\udag}<\epsilon/2,\qquad \forall\, k>\bar{k}. 
		\end{equation}
		Moreover, for $k$ fixed, there exists $\bar{n}=\bar{n}(\epsilon,k)$ such that for all $n>\bar{n}$  
		\begin{equation}\label{eq: udeltak_uk_eps}
		\norm{u^{\delta_n}_{k}-u_k}<\epsilon/2,\qquad \forall\, n>\bar{n}.
		\end{equation}
		In fact, writing explicitly the difference $u^{\delta_n}_{k}-u_k$, we get
		\begin{equation*}
		\begin{aligned}
		u^{\delta_n}_k-u_k=u^{\delta_n}_{k-1}-u_{k-1}&- F'(u^{\delta_n}_{k-1})^*(F(u^{\delta_n}_{k-1})-y_n)+ F'(u_{k-1})^*(F(u_{k-1})-y)\\
		&-\lambda^{\delta_n}_{k-1}\hat{A}'(u^{\delta_n}_{k-1})^*(\hat{A}(u^{\delta_n}_{k-1})-y_n)+\lambda _{k-1} \hat{A}'(u_{k-1})^*(\hat{A}(u_{k-1})-y).
		\end{aligned}
		\end{equation*}	
		The estimate \eqref{eq: udeltak_uk_eps} follows using again the continuity property of $F$ and $\hat{A}$, and also the hypothesis that $\lambda^{\delta_n}_{k-1}\to\lambda_{k-1}$, for every fixed iteration.  
	}
	Therefore, by \eqref{eq: udeltak_uk_eps} and \eqref{eq: uk_us_eps} in \eqref{eq: udeltakn_us}, we find	$\norm{u^{\delta_n}_{k_n}-\udag}< \epsilon$ for $n>\bar{n}$, which means that $u^{\delta_n}_{k_n}\to \udag$ as $n\to \infty$. 
\end{proof}

\newpage
\section{Numerical Experiments}\label{sec: num_exp}
In this section we present some numerical examples related to the iteration \eqref{eq:SDscheme} both for linear operator equations and nonlinear ones. We first consider the linear case for which the operator is the Radon transform. Then, we present some results for the Schlieren model which we take as the prototype for nonlinear operators.\\
{In the following numerical examples, we always consider noise in the data hence we do not need distinguish anymore between $\lambda_k$ and $\lambda^{\delta}_k$. For this reason, in order to simplify the notation, from this point on we merely use the notation $\lambda_k$ instead of $\lambda^{\delta}_k$}.\\ 
All the numerical results in the following sections are based on the assumption that the damping term in \eqref{eq:SDscheme} is a bounded linear map  $\tilde{A}:X \to Y$, hence the iterates become
\begin{equation}\label{eq: IRLI_Alinear}
\ukp := \uk - F'(\uk)^*\bigl( F(\uk)-\yd \bigr) - {\lambda_k}
\tilde{A}^*\bigl(\tilde{A}\uk -\yd \bigr), \qquad k\in\N. 
\end{equation}
To build the operator $\tilde{A}:X\to Y$, we assume to have some a-priori information about \eqref{eq:operator} in the form of a finite set of expert data that is 
\begin{equation}\label{eq: train_data}
(u^{(i)},y^{(i)})\in X \times Y,\,\, \textrm{for}\,\, i=1,\cdots,n,
\end{equation}
where $n>0$,  and we define, for $\tilde{T}\in B(X,Y)$,  the functional
\begin{equation}\label{eq: def_func}
l(\tilde{T})=\frac{1}{2} \sum_{i=1}^n \norm{\tilde{T} u^{(i)}-y^{(i)}}^2_{Y}.
\end{equation}
The bounded operator $\tilde{A}$ is defined as the operator which minimizes the functional \eqref{eq: def_func}, i.e.,  
\begin{equation}\label{eq: min_A}
l(\tilde{A})=\underset{\tilde{T}\in\mathcal{B}(X,Y)}{\min} l(\tilde{T})=\underset{\tilde{T}\in\mathcal{B}(X,Y)}{\min} \left[\frac{1}{2} \sum_{i=1}^n \norm{\tilde{T} u^{(i)}-y^{(i)}}^2_{Y}\right].
\end{equation} 
Let $H\in {B}(X,Y)$ and $t\in J$ where $J\subset \mathbb{R}$ is an interval containing the origin. Then
\begin{equation*}
\begin{aligned}	
\frac{l(\tilde{T}+tH)-l(\tilde{T})}{t}&=\frac{1}{2}\frac{1}{t}\Bigg\{\sum_{i=1}^{n}\left[\Big\langle(\tilde{T}+tH)u^{(i)}-y^{(i)},(\tilde{T}+tH)u^{(i)}-y^{(i)}\Big\rangle  - \Big\langle \tilde{T}u^{(i)}-y^{(i)},\tilde{T}u^{(i)}-y^{(i)}\Big\rangle  \right]\Bigg\}\\
&=\sum_{i=1}^{n}\left[\Big\langle \tilde{T}u^{(i)}-  y^{(i)}, Hu^{(i)}\Big\rangle +\frac{t}{2}\Big\langle Hu^{(i)}, Hu^{(i)}\Big\rangle\right].
\end{aligned}	
\end{equation*}
Therefore, as $t\to 0$, for each $H\in {B}(X,Y)$ and every $(u^{(i)},y^{(i)})\in X\times Y$, for $i=1,\cdots,n$, the operator $\tilde{A}$ is defined as the operator which satisfies the condition 
\begin{equation*}
\sum_{i=1}^{n}\Big\langle \tilde{A}u^{(i)}-  y^{(i)}, Hu^{(i)}\Big\rangle=0,\qquad {\forall H\in\mathcal{B}(X,Y)}
\end{equation*}
which implies that 
\begin{equation}\label{eq: min_cond_A}
\sum_{i=1}^{n}\tilde{A}u^{(i)}=\sum_{i=1}^{n}  y^{(i)}.
\end{equation}
Now, denoting by $\{e_k(x)\}_{k\in\N}$ and $\{\bar{e}_j(y)\}_{j\in\N}$  complete orthonormal families of $X$ and $Y$, respectively, we have
\begin{equation}\label{eq: repr u,y}
\begin{aligned}
u^{(i)}&=\sum_{k\in\N}\langle u^{(i)}, e_k \rangle e_k=\sum_{k\in\N} u^{(i)}_k e_k,\,\qquad \forall i=1,\cdots,n, \\
y^{(i)}&=\sum_{j\in\N}\langle y^{(i)}, \bar{e}_j \rangle \bar{e}_j=\sum_{j\in\N} y^{(i)}_j \bar{e}_j,\, \qquad \forall i=1,\cdots,n,
\end{aligned}
\end{equation}
and we recall that each bounded operator between Hilbert spaces has a matrix representation given by
\begin{equation}\label{eq: repr_A}
\tilde{A}e_k=\sum_{j\in\N}\tilde{a}_{jk} \bar{e}_j,\qquad\qquad \text{where}\,\, \tilde{a}_{jk}=\langle \tilde{A}e_k, \bar{e}_j\rangle, 
\end{equation}
see for example \cite{Kat95}. In this way, from \eqref{eq: repr u,y} and \eqref{eq: repr_A}, we find 
\begin{equation*}
\tilde{A}u^{(i)}=\sum_{k\in\N}u^{(i)}_k \tilde{A}e_k=\sum_{k,j\in\N}\tilde{a}_{jk} u^{(i)}_k \bar{e}_j,\qquad \forall\, i=1,\cdots, n.
\end{equation*} 
Therefore, Equation \ref{eq: min_cond_A} can be rewritten as 
\begin{equation*}
\sum_{i=1}^{n}\sum_{k,j\in\N}\tilde{a}_{jk}u^{(i)}_k \bar{e}_j=\sum_{i=1}^{n}\sum_{j\in\N}y^{(i)}_j \bar{e}_j
\end{equation*}
that is, for each $j\in\N$ and $i=1,\cdots, n$, we have 	
\begin{equation*}
\sum_{k\in\N}\tilde{a}_{jk}u^{(i)}_k= y^{(i)}_j. 
\end{equation*}
In the sequel we approximate the matrix $\tilde{a}_{jk}$, with a matrix $A\in\mathbb{R}^{M\times N}$, that is we take a finite number of elements of the orthonormal families of the Hilbert spaces $X$ and $Y$. To be more precise, we employ $\{e_k\}$, for $k=1,\cdots N$, and $\{\bar{e}_j\}$, for $j=1,\cdots M$.\\	  
\textbf{Construction of $\mathbf{A}$.} 
Following considerations above, we consider
\begin{equation}\label{eq: train_data_finite}
({u}^{(i)},{y}^{(i)})\in \mathbb{R}^N \times \mathbb{R}^M,\,\, \textrm{for}\,\, i=1,\cdots,n,
\end{equation}
where $n>0$. 
With this choice, the matrix $A\in\mathbb{R}^{M\times N}$ is defined as the linear map such that 
\begin{equation}\label{eq: def_A}
AU=Y, 
\end{equation}
where $U\in\mathbb{R}^{N\times n}$ and $Y\in\mathbb{R}^{M\times n}$ are the matrices which contain columnwise the data $u^{(i)}$ and $y^{(i)}$, respectively, for $i=1,\cdots,n$. The matrix $A$ can be obtained by utilizing a Singular Value Decomposition (SVD) on $U$, hence 
\begin{equation}\label{eq: def_A_calc}
A=YU^{\dagger}.
\end{equation}

\begin{remark}
	One of the main issue of this approach is the computational burden in building the matrix $A$ that is the amount of the storage needed for the data. In fact the dimensions $M$ and $N$ of the matrix $A$ could be very large even for small scale inverse problems. 
\end{remark}
The following numerical experiments have been accomplished using Matlab and some of its functions.

\subsection{Linear Operator - Radon Transform}\label{sec:Radon}
In this section, we present and discuss some numerical experiments assuming that the operator $F$ in \eqref{eq:operator} is the Radon transform, i.e., we consider the operator equation		
\begin{equation}\label{eq:operator_R}
Ru=y.
\end{equation}
We refer, for instance, the reader to \cite{Kuc14,OlafQuin06} for the definition and properties of the Radon transform. 

In the sequel, we implement the Data-Driven Iteratively Regularized Landweber Iteration (DDIRLI) \eqref{eq: IRLI_Alinear}, adapted to Equation \eqref{eq:operator_R}, i.e  
\begin{equation}\label{eq:ModLand_R}
\ukp := \uk - \omega_R R^*\bigl( R\uk-\yd \bigr) - {\lambda^{\delta}_k} \tilde{A}^*\bigl(\tilde{A}\uk -\yd \bigr), \qquad k\in\N,
\end{equation} 
where $R^*$ is the backprojection operator and $\omega_R$ is a positive constant {such that $\norm{R^*(Ru_k-\yd)}\leq1$}.  The operator $\tilde{A}$ can be realized by a matrix $A$ which is generated following the approach discussed at the beginning of Section \ref{sec: num_exp}. For this purpose, we use the following set of training data and parameters: 
\begin{enumerate}[(a)]
	\item\label{item: point1} $180$ different equally distributed angles $\theta$ within the interval $[0,\pi)$; specifically, we consider the set $\Theta=\Big\{0,\frac{\pi}{180},\frac{\pi}{90},\cdots,\frac{179\pi}{180}\Big\}$;
	\item \label{item: point2} A sample of 50 grayscale images of handwritten digits, from MNIST database, and their related sinograms, which are obtained using the 180 directions of \eqref{item: point1}. 
\end{enumerate}
Matrices $U$ and $Y$ in \eqref{eq: def_A} are built utilizing data $({u}^{(i)},{y}^{(i)})$, where ${u}^{(i)}$ and ${y}^{(i)}$ are, respectively, the images and their related sinograms described in \eqref{item: point2}. {The number of training pairs utilized for the creation of $A$ will be specified in each numerical test.} 

In the implementation, for the set of parameters we make these choices:
\begin{enumerate}
	\item We choose
	\begin{equation}\label{value lambdak}
	\omega_R=10^{-2} \qquad \textrm{and}\qquad \lambda_{k}=7.7\times10^{-5} \norm{R u_k-\yd}_2^2,
	\end{equation}
	where $\norm{\cdot}_2$ is the spectral norm. The value of $\lambda_k$ is inferred by conditions \eqref{eq:lambda_ca}.
	\item \label{stopping_criteria} As stopping rule we use the discrepancy principle \eqref{eq: discrep_princ}, where the choice of $\tau$ will be specified for each test (see tables below), or a maximum number of 100 iterates; 
	\item The synthetic data $\yd$ are generated by adding a gaussian distributed noise with zero mean and variance $\sigma^2$, specified at each test, to the matrix of the exact data $y$;
	\item The initial guess $u_0$ is always set to be $u_0=0$. 
\end{enumerate}
Results of \eqref{eq:ModLand_R} are compared with the outcomes of the Landweber iteration, that is
\begin{equation}\label{eq:Land_R}
\ukp := \uk - \omega_R R^*\bigl( R\uk-\yd \bigr), \qquad k\in\N,
\end{equation}
{and of the Iteratively Regularized Landweber Iteration (IRLI) given by
	\begin{equation}\label{eq:IRLI_classic}
	\ukp := \uk - \omega_R R^*\bigl( R\uk-\yd \bigr) - \beta_k (\uk -u^{(0)}), \qquad k\in\N,
	\end{equation} 
	where, according to the parameters rules in \cite[Section 3.2]{KalNeuSch08}, we choose $\beta_k=(\frac{1}{4})^{k+1}$, for $k\in\mathbb{N}$, and we use as $u^{(0)}$ a noise version of the image to be reconstructed, see Figure \ref{fig: choice_u0}. Finally, the numerical results provided by these iterative algorithms are compared with those of} the Filtered Back-Projection (FBP) algorithm, which performs the inverse Radon transform. For this last case, we run the Matlab function ``iradon'' making the choice of the Ram-Lak filter.
\begin{figure}[!h]
	\centering
	\subfigure{
		\includegraphics[scale=0.6]{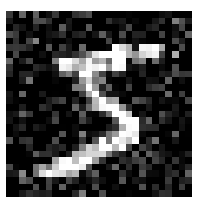}}\hspace{2cm}
	\subfigure{
		\includegraphics[scale=0.6]{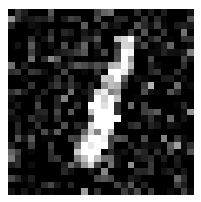}}\hspace{2cm}
	\subfigure{
		\includegraphics[scale=0.5]{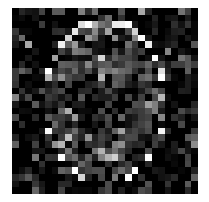}}
	\caption{Choices of $u^{(0)}$ in \eqref{eq:IRLI_classic}: image on the left is used in Test 1 and Test 3a. Image in the center is used in Test 2a/2b and Test 3b. Image on the right is used in Test 4.}\label{fig: choice_u0}
\end{figure}

We used the training and validation sets of the MNIST dataset of handwritten digits, {see \cite{LecunBottBengHaff}}. We essentially consider four different numerical experiments: In \textbf{Test 1} we reconstruct an image which has been utilized to create the matrix $A$, i.e, it is contained in the training set of MNIST. In \textbf{Test 2a/2b}, using the same matrix $A$ of test 1, we reconstruct a digit in the validation test of MNIST. {In \textbf{Test 3a} and \textbf{Test 3b}}, we reconstruct the same images of {Test 1} {and Test 2, respectively,} but using only partial data, i.e., we consider a cropped noisy version of the sinogram of the true image. {We stress that the true image of Test 3b has not been utilized to create the matrix $A$. Finally, in \textbf{Test 4}, we reconstruct the Shepp Logan phantom in order to test our method on an image which is completely different from the ones in the training and validation sets.}  

The choices and values of the main parameters, which are utilized for the generation of the data and to run the schemes \eqref{eq:ModLand_R}, \eqref{eq:Land_R} and \eqref{eq:IRLI_classic}, are contained in the tables below, i.e., Tables \ref{tab: table1}, \ref{tab: table2a}, \ref{tab: table2b}, \ref{tab: table3a}, \ref{tab: table3b} and \ref{tab: table4}: 
{In the left part of each table,} we specify the choice of the variance in the gaussian noise, the value of $\delta$ and $\tau$. 
{In the right part we state some of the results of the numerical test, i.e.,} the total number of iterates before one of the two stopping criteria in \ref{stopping_criteria} has been accomplished,
{the computational time and, finally, the relative error $\norm{u_{\textrm{true}}-u_{\textrm{rec}}}_2/\norm{u_{\textrm{true}}}_2$, where $u_{\textrm{true}}$ is the image to be reconstructed while $u_{\textrm{rec}}$ is the reconstructed image. 
	We stress that the execution time of DDIRLI is the result of the construction of the matrix $A$ and the iterates of \eqref{eq:ModLand_R}}. 
\ \\ 
\textbf{Test 1.} (Figure \ref{fig: test1}). \textit{Target: reconstruct one of the image from the training set which has been used to create the matrix $A$}. See Figure \ref{fig: test1} for the ``True Image''. 
The initial datum $\yd$ is given by the sinogram of the true image $y$, adding a gaussian distributed noise of zero mean ad variance $\sigma^2=0.5$. 
In the discrepancy principle \eqref{eq: discrep_princ} we choose $\tau=1.1$. See Table \ref{tab: table1}.
We observe, from Figure \ref{fig: test1} and Figure \ref{fig: test1_2}, how the presence of $A$ (and the fact that the information of the "True Image" are included in $A$) ``helps'' the iterates \eqref{eq:ModLand_R} in recognizing the true image, as expected. Indeed, despite the presence of a strong noise, DDIRLI provides better results compared to those of the Landweber iteration \eqref{eq:Land_R} 
{and IRLI \eqref{eq:IRLI_classic}, with a smaller number of iterates. Note that the initial guess $u^{(0)}$ in IRLI, see \eqref{eq:IRLI_classic}, is the \textquotedblleft True Image\textquotedblright\ corrupted with a gaussian noise of variance $0.05$ only.}
\ \\
\textbf{Test 2a/2b}. (Figures \ref{fig: test2a} and \ref{fig: test2b}). \textit{Target: reconstruct one of the image from the validation set, {not utilized to create $A$}}.
We reconstruct a digit from the validation set of MNIST. The difference between the two tests, see Figures \ref{fig: test2a} and \ref{fig: test2b}, is the number of the samples from training data used to create the matrix $A$. In fact, we used in Test 2a only 50 images while in Test 2b 150 images.
{From these numerical examples, we clearly observe an improvement of the results of DDIRLI when we increase the number of images utilized to create the matrix $A$}. In fact, Test 2b shows very good results, comparable with those of the Landweber scheme and {IRLI}.

\ \\
\textbf{Test 3a/3b} (Figures \ref{fig: test3a} and \ref{fig: test3b}). {\textit{Target: reconstruct images from the training or the validation set using partial data}.
	We consider the case of partial noisy measurements which corresponds to taking a cut version of the sinogram $\yd$ with a gaussian distributed noise of zero mean and variance $\sigma=0.03$. The cropped sinogram is obtained mantaining information only in some directions and setting to zero all the rest. See Figure \ref{fig: test3a} and Figure \ref{fig: test3b}. \\
	In Test 3a, we reconstruct an image belonging to the training set and which has been utilized to create the matrix $A$. On the contrary, in Test 3b, we reconstruct an image from the validation set and therefore was not utilized to generate the matrix $A$. The outcomes of our method are promising in both the cases.
}

\ \\
{\textbf{Test 4} (Figure \ref{fig: test4a} and \ref{fig: test4b}).
	\textit{Target: reconstruct Shepp Logan phantom, i.e., an image which is completely different from the elements contained in training data which are used to create the matrix $A$.}\\
	To use the MNIST training set to create $A$, the Shepp Logan image has been down-sampled to $28\times 28$. In numerical simulation, we observe that, in order to get a good reconstruction of the true image and comparable with the numerical results of the other methods, we use $600$ images from the training set. 
}

We stress that Figures \ref{fig: test1_2} C), \ref{fig: test2a_2} C), \ref{fig: test2b_2} C), \ref{fig: test3_2a} C) and \ref{fig: test3_2b} C) show the ratios between the residuals of the data driven model, $\norm{Au_k-\yd}$, and the model driven approach, $\norm{R u_k-\yd}$. {A value greater than $1$ of the ratio implies that the data driven approach has a significant influence on the iteration process. In fact, since the initial guess is $u_0=0$, from plots C), we observe how our scheme, in the first steps, is giving more weight to the information contained in $A$ to reconstruct the object. Figures \ref{fig: test1_2} D), \ref{fig: test2a_2} D), \ref{fig: test2b_2} D), \ref{fig: test3_2a} D) and \ref{fig: test3_2b} D) enlight how fast the product $\|Ru_k-y^{\delta}\|_2\  \|Au_k-y^{\delta}\|_2$ is going to zero. In fact, since $\lambda_k=7.7\times10^{-5}\|Ru_k-y^{\delta}\|^2_2$, we have that
	\begin{equation*}
	\frac{\|Ru_k-y^{\delta}\|_2}{\lambda_k \|Au_k-y^{\delta}\|_2}=\frac{1}{7.7\times10^{-5} \|Ru_k-y^{\delta}\|_2\ \|Au_k-y^{\delta}\|_2}.
	\end{equation*}
	Finally, we stress that plots A), B) and D) in Figures \ref{fig: test1_2}, \ref{fig: test2a_2}, \ref{fig: test2b_2}, \ref{fig: test3_2a}, \ref{fig: test3_2b} are all in logarithmic scale.}

\newpage
{	\begin{table}[!h]\caption{Test 1. Left part: Parameters used in the test. Right part: some of the results of the test.}\label{tab: table1}
		\begin{center}
			{\renewcommand{\arraystretch}{1.2}
				\begin{tabular}{|l|c|c|c||c|c|c|}
					\cline{1-7} & & & & & & \vspace{-0.2cm} \\
					Method & $\sigma^2$-noise& $\delta$ & $\tau$ & Iterations & Comp. Time (s) & $\frac{\norm{u_{\textrm{true}} - u_{\textrm{rec}}}_2}{\norm{u_{\textrm{true}}}_2}$\\ \hline  
					{DDIRLI} &  &  &  & 25 & 1,0732974 & 0,107250035 \\ \cline{1-1}\cline{5-7}
					{IRLI} & 0.5  & 13,36110296
					& 1.1  & 58 & 1,3639758 & 0,114869088 \\ \cline{1-1}\cline{5-7} 
					LANDWEBER &  & &  & 73 &1,734035
					& 0,12467763 \\ \cline{1-1}\cline{3-4}\cline{5-7} 
					FBP &  & / & / & / & 0,0546157
					& 0,124274036\\ \hline   
				\end{tabular}
			}
		\end{center}   
	\end{table} 
}
\vspace{2cm}
\begin{figure}[!h]
	\centering
	\subfigure{
		\includegraphics[scale=0.35]{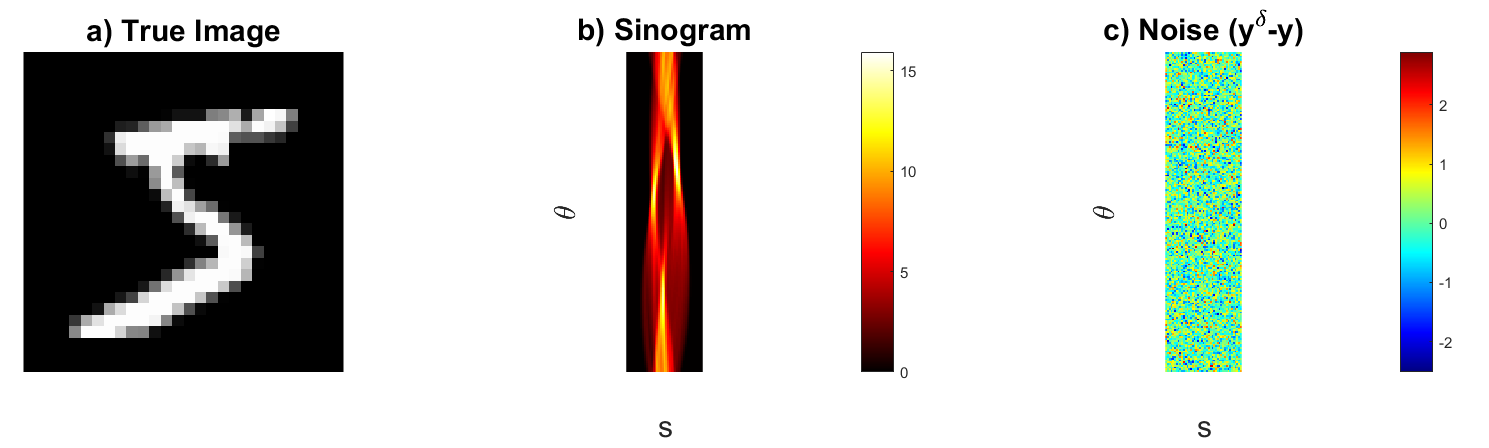}}
	\subfigure{\includegraphics[scale=0.37]{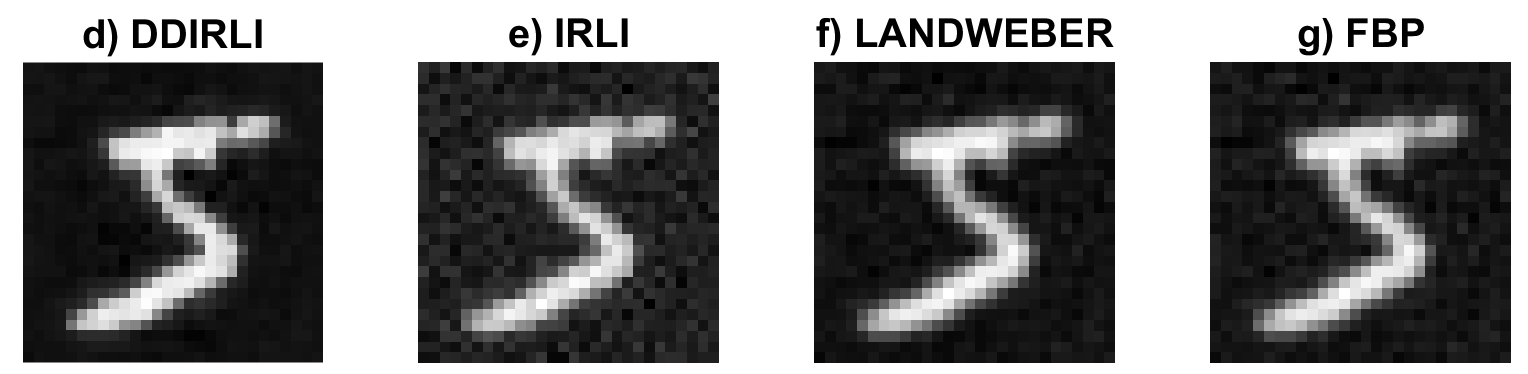}}
	\caption{Test 1. Reconstruction of an image from the training set which has been used to create $A$ in \eqref{eq: def_A_calc}. a) Image to be reconstructed; b) Sinogram $y$ of the true image; c) Plot of the gaussian noise, i.e., $\yd-y$; d)-g) Reconstructions by the different methods. DDIRLI gives a very good reconstruction, with more details compared to IRLI, Landweber and FBP.}\label{fig: test1}
\end{figure}
\newpage
\begin{figure}[!h]
	\centering
	\vspace{1cm}
	\includegraphics[scale=0.4]{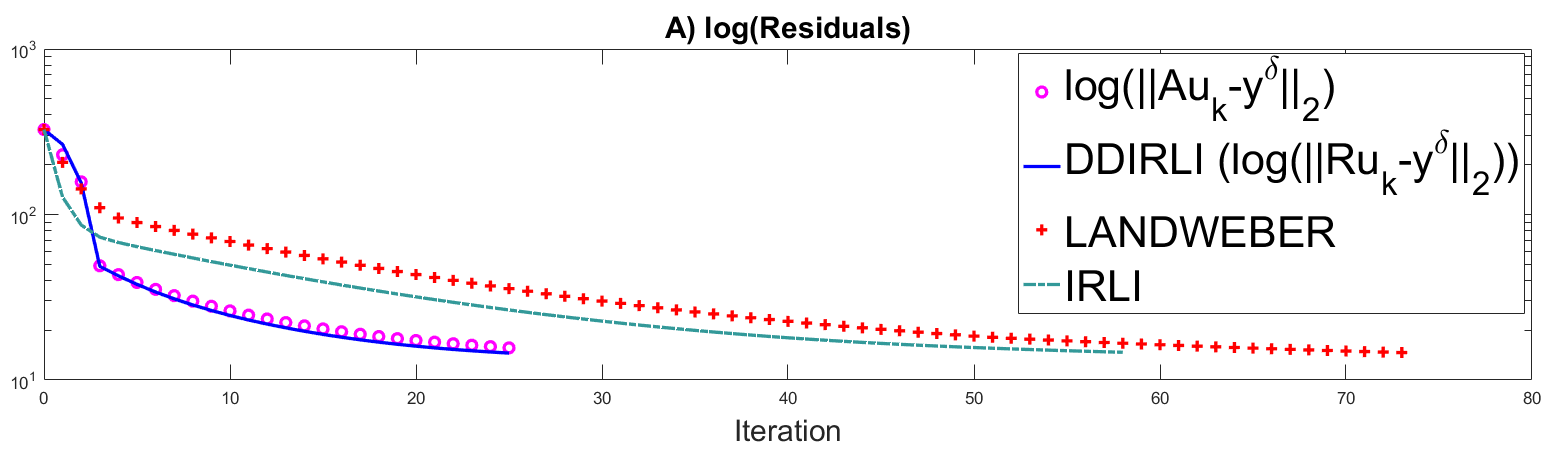}\vspace{1cm}
	\includegraphics[scale=0.4]{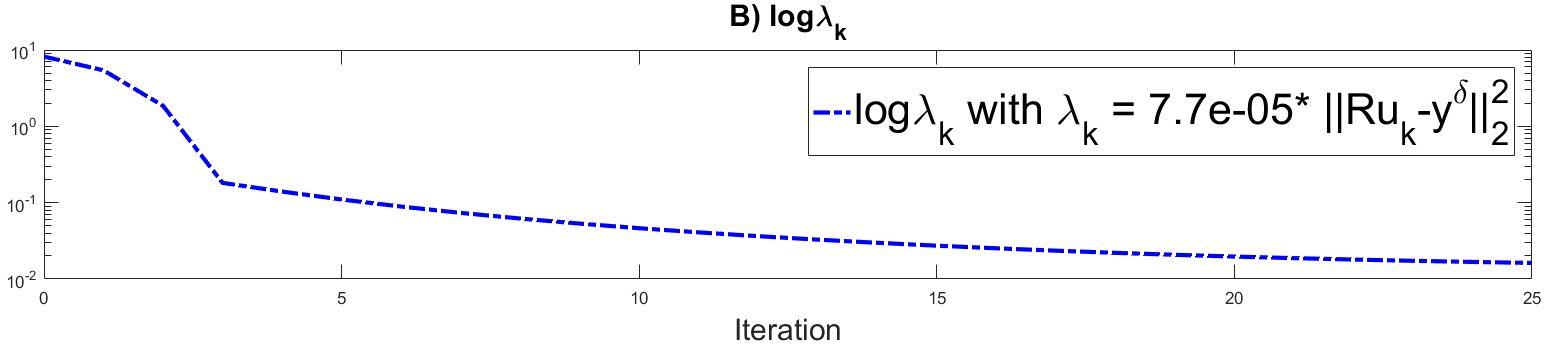}\vspace{1cm}
	\includegraphics[scale=0.4]{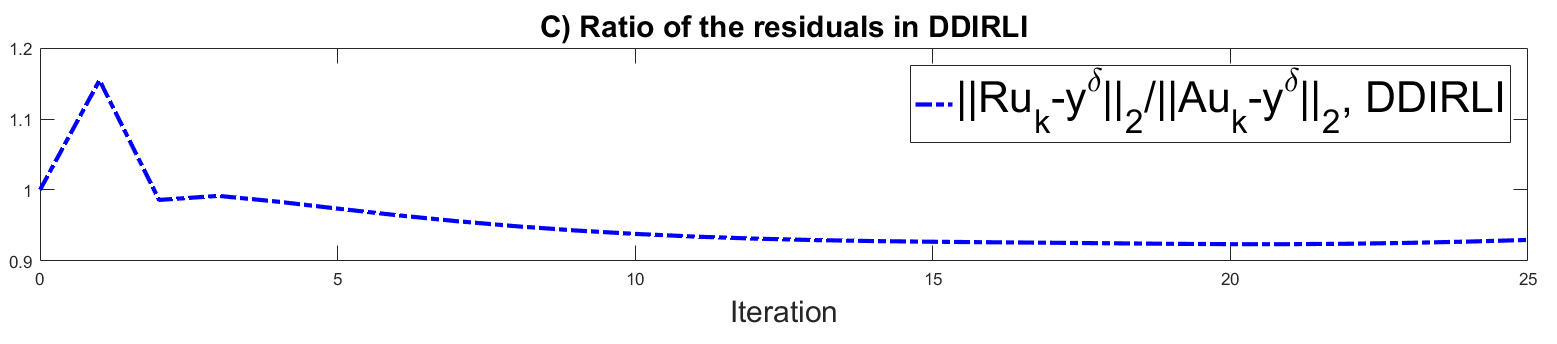}\vspace{1cm}
	\includegraphics[scale=0.4]{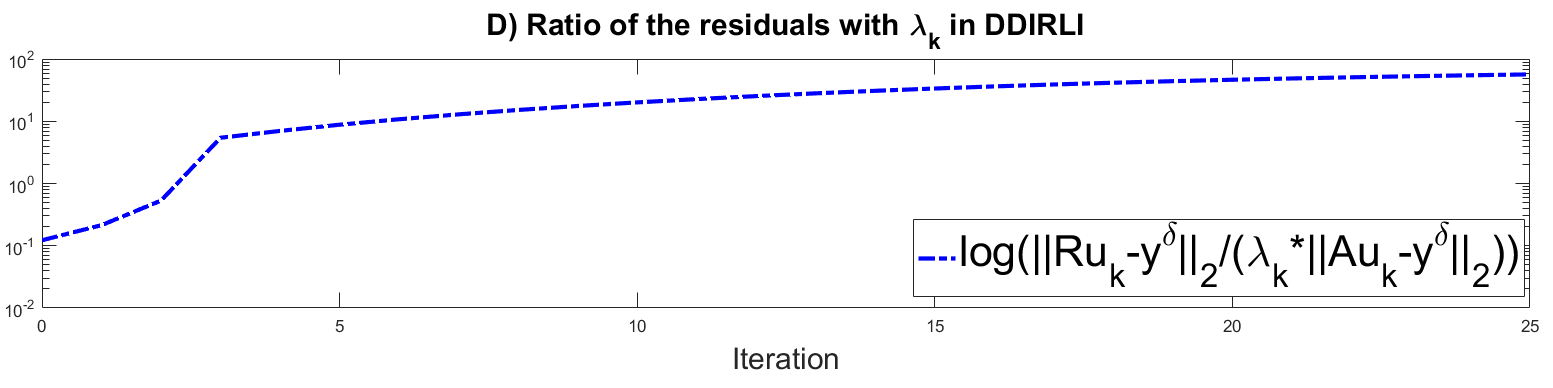}
	\caption{Test 1. A) Residual in logarithmic scale at each step implementing \eqref{eq:ModLand_R}, \eqref{eq:Land_R} and \eqref{eq:IRLI_classic}; B) plot of $\log(\lambda_{k})$; C) plot of the ratio $\norm{Ru_k-\yd}_2/\norm{Au_k-\yd}_2$; D) plot in logarithmic scale of the ratio $\norm{Ru_k-\yd}_2/(\lambda_{k} \norm{Au_k-\yd}_2)$.}\label{fig: test1_2}
\end{figure}

\newpage
{
	\begin{table}[!h]
		\caption{Test 2a. Left part: Parameters used in the test. Right part: some of the results of the test.}\label{tab: table2a}
		\begin{center}
			{\renewcommand{\arraystretch}{1.2}
				\begin{tabular}{|l|c|c|c||c|c|c|}
					\cline{1-7} & & & & & & \vspace{-0.2cm} \\
					Method & $\sigma^2$-noise& $\delta$ & $\tau$ & Iterations & Comp. Time (s) & $\frac{\norm{u_{\textrm{true}} - u_{\textrm{rec}}}_2}{\norm{u_{\textrm{true}}}_2}$\\ \hline  
					{DDIRLI}  &  &  &  & 39 &1,474326   & 0,203051179  \\ \cline{1-1}\cline{5-7}
					{IRLI} & 0.5  & 13,41600448
					& 1.1  & 52 & 1,2946457 & 0,139542084
					\\ \cline{1-1}\cline{5-7} 
					LANDWEBER &  & &  & 67 &1,743846  & 0,14349784
					\\ \cline{1-1}\cline{3-4}\cline{5-7} 
					FBP &  & / & / & / & 0,0273525
					& 0,138023154
					\\ \hline   
				\end{tabular}
			}
		\end{center}   
	\end{table} 
}   
\vspace{2cm}
\begin{figure}[!h]
	\centering
	\includegraphics[scale=0.35]{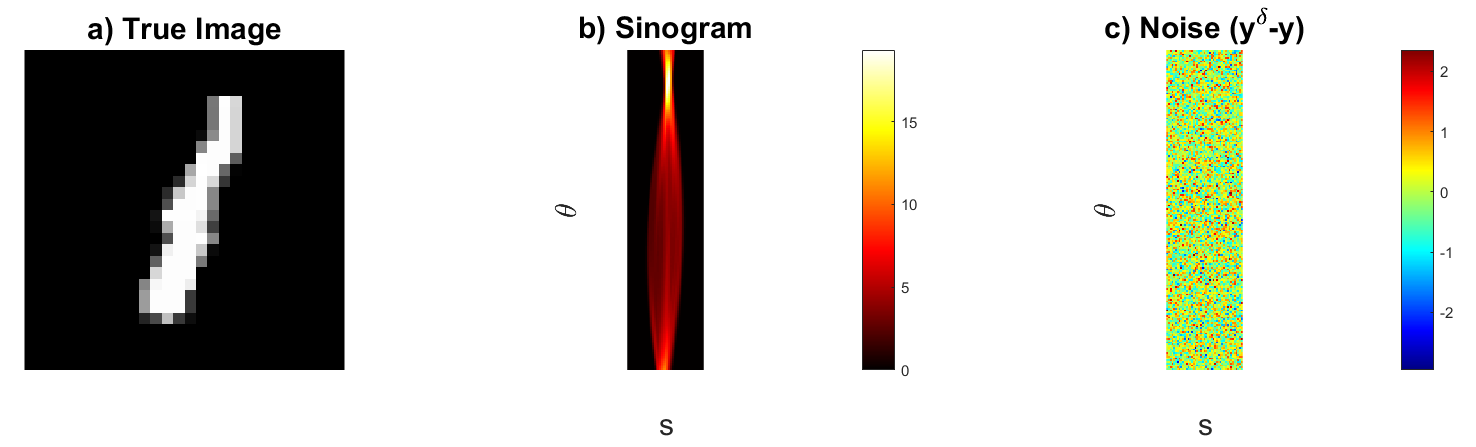}
	\subfigure{\includegraphics[scale=0.37]{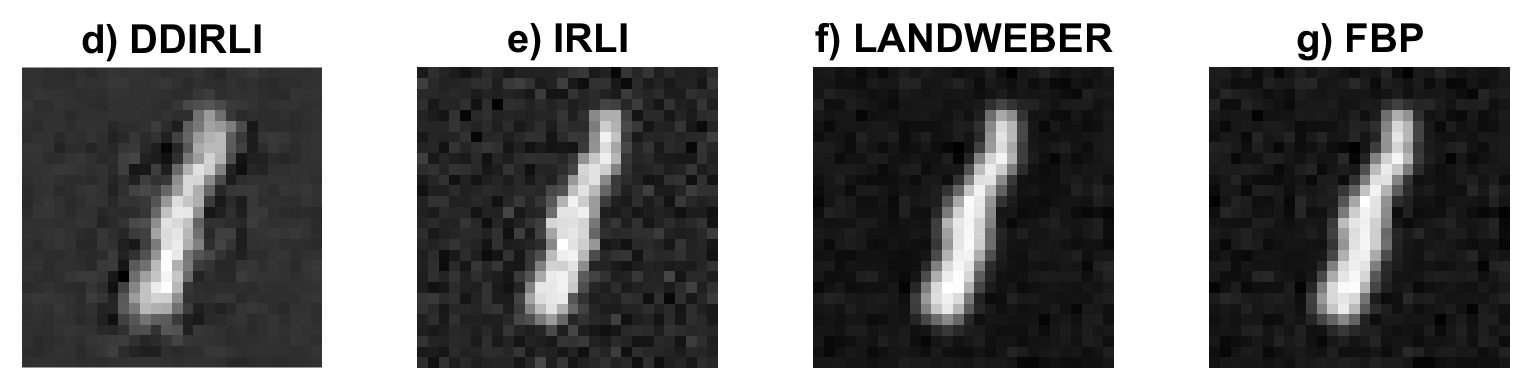}}
	\caption{Test 2a. Reconstruction of an image from the validation set not used to create $A$ in \eqref{eq: def_A_calc}. Only a sample of 50 images from the training set has been used to build $A$. a) Image to be reconstructed; b) Sinogram ($y$) of the true image; c) Plot of the gaussian noise, i.e., $\yd-y$; d)-g) Reconstructions by the different methods.  IRLI, Landweber and FBP give better results than the DDIRLI method. We stress that the discrepancy principle for DDIRLI is satisfied after only 39 iterations.}\label{fig: test2a}
\end{figure}
\newpage
\begin{figure}[!h]
	\centering
	\vspace{1cm}
	\includegraphics[scale=0.4]{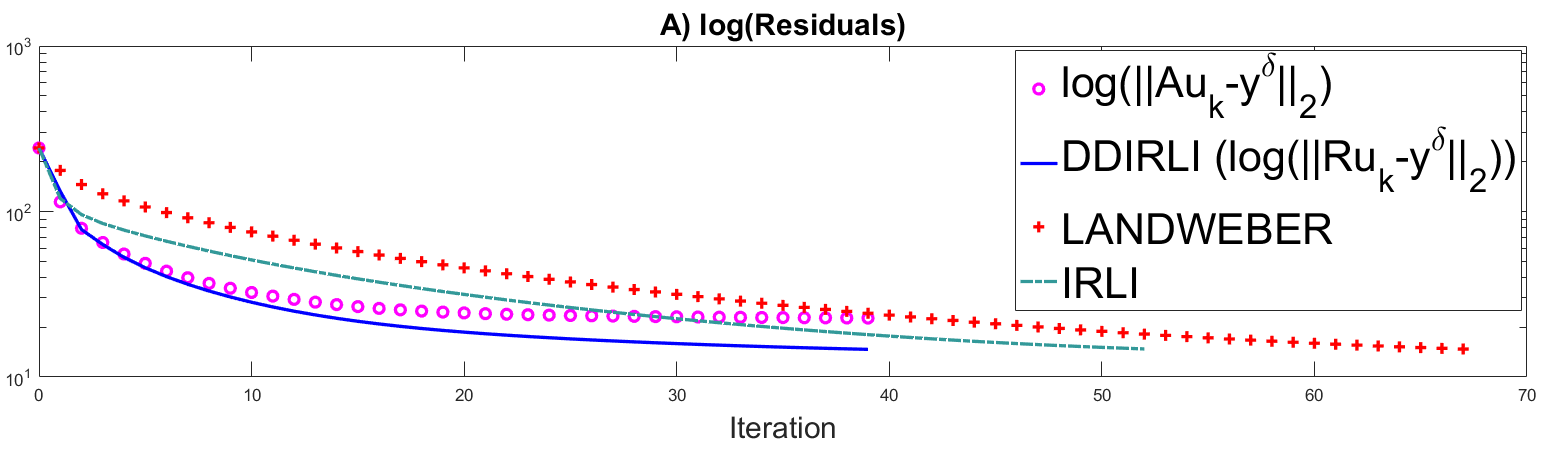}\vspace{1cm}
	\includegraphics[scale=0.4]{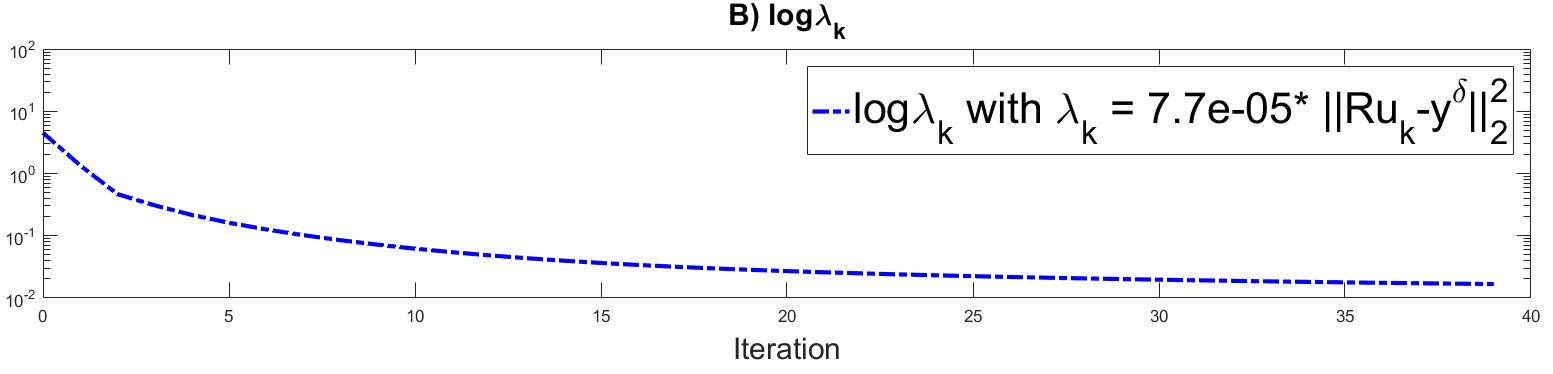}\vspace{1cm}
	\includegraphics[scale=0.4]{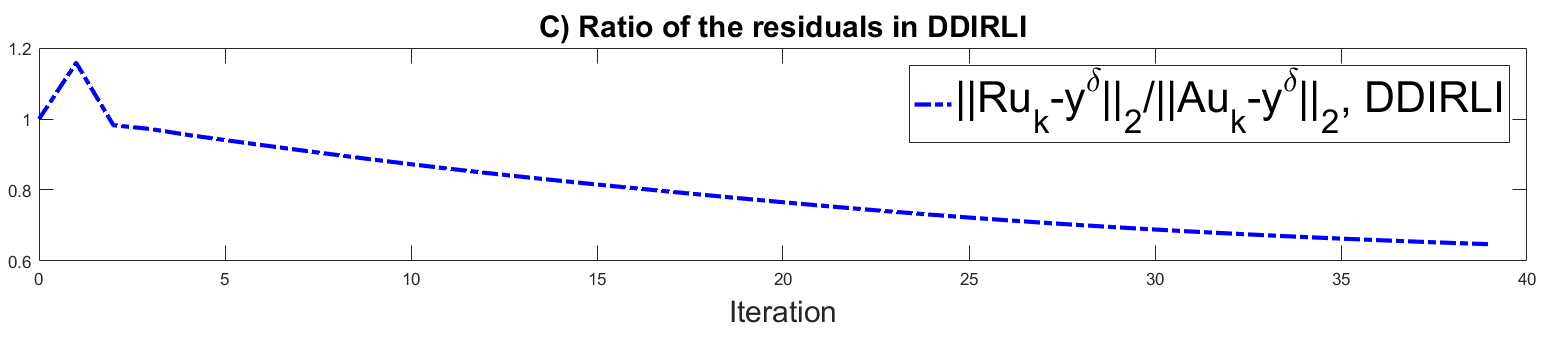}\vspace{1cm}
	\includegraphics[scale=0.4]{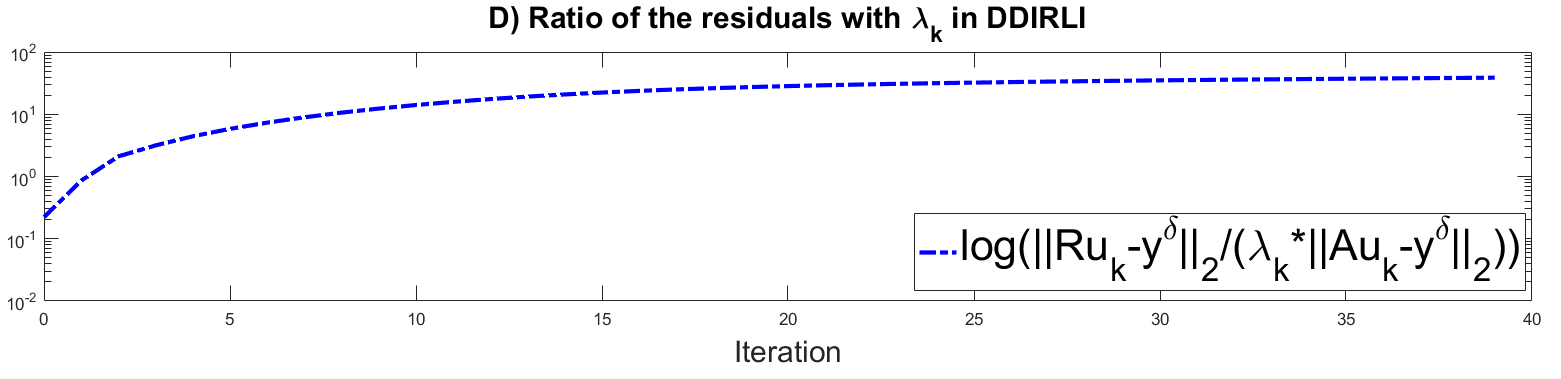}
	\caption{Test 2a. A) Residual in logarithmic scale at each step implementing \eqref{eq:ModLand_R}, \eqref{eq:Land_R} and \eqref{eq:IRLI_classic}; B) plot of $\log(\lambda_{k})$; C) plot of the ratio $\norm{Ru_k-\yd}_2/\norm{Au_k-\yd}_2$; D) plot in logarithmic scale of the ratio $\norm{Ru_k-\yd}_2/(\lambda_{k} \norm{Au_k-\yd}_2)$.}\label{fig: test2a_2}
\end{figure}

\newpage
{
	\begin{table}[!h]
		\caption{Test 2b. Left part: Parameters used in the test. Right part: some of the results of the test.}\label{tab: table2b}
		\begin{center}
			{\renewcommand{\arraystretch}{1.2}
				\begin{tabular}{|l|c|c|c||c|c|c|}
					\cline{1-7} & & & & & & \vspace{-0.2cm} \\
					Method & $\sigma^2$-noise& $\delta$ & $\tau$ & Iterations & Comp. Time (s) & $\frac{\norm{u_{\textrm{true}} - u_{\textrm{rec}}}_2}{\norm{u_{\textrm{true}}}_2}$\\ \hline  
					{DDIRLI} &  &  &  & 24 & 1,0422007 & 0,161019463
					\\ \cline{1-1}\cline{5-7}
					{IRLI} & 0.5  & 13,71978428	& 1.1  & 53 & 1,3471281 & 0,135825795
					\\ \cline{1-1}\cline{5-7} 
					LANDWEBER &  & &  & 66 &1,5579816  & 0,142645348
					\\ \cline{1-1}\cline{3-4}\cline{5-7} 
					FBP &  & / & / & / & 0,0419196
					& 0,135176876
					\\ \hline   
				\end{tabular}
			}
		\end{center}   
	\end{table} 
}
\vspace{2cm}
\begin{figure}[!h]
	\centering
	\includegraphics[scale=0.45]{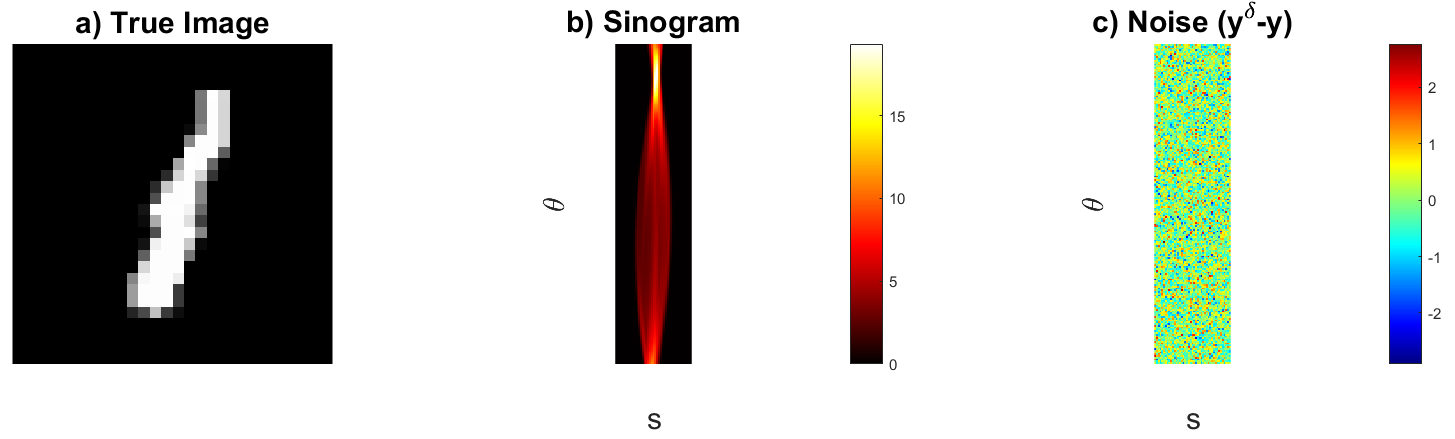}
	\subfigure{\includegraphics[scale=0.47]{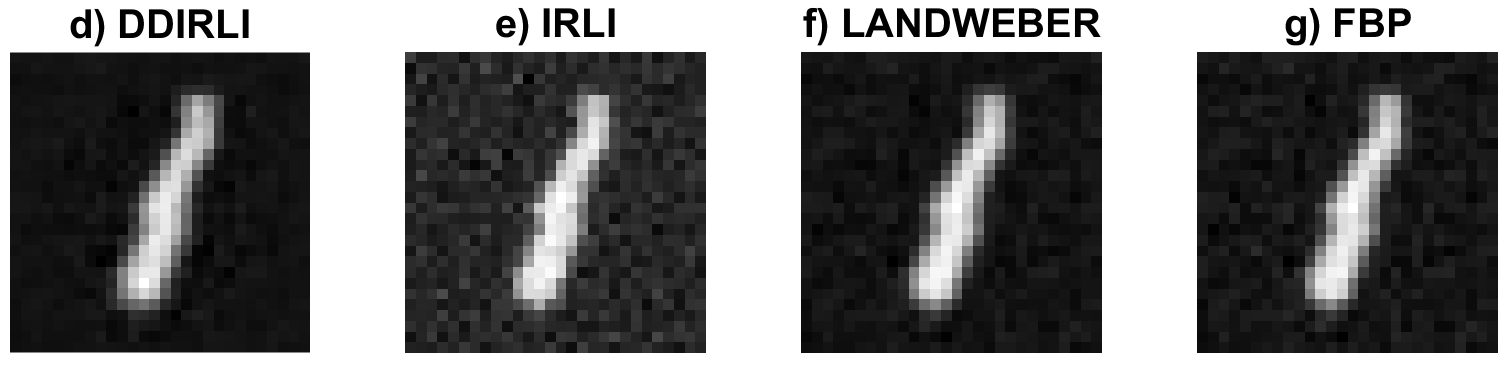}}
	\caption{Test 2b. Reconstruction of an image from the validation set not used to create $A$ in \eqref{eq: def_A_calc}. Only a sample of 150 images from the training set has been used to build $A$. a) Image to be reconstructed; b) Sinogram ($y$) of the true image; c) Plot of the gaussian noise, i.e., $\yd-y$; d)-g) Reconstructions by the different methods.  Reconstructions are all comparable. We stress that the discrepancy principle for DDIRLI is satisfied after only 24 iterations.}\label{fig: test2b}
\end{figure}
\newpage
\begin{figure}[!h]
	\centering
	\vspace{1cm}
	\includegraphics[scale=0.5]{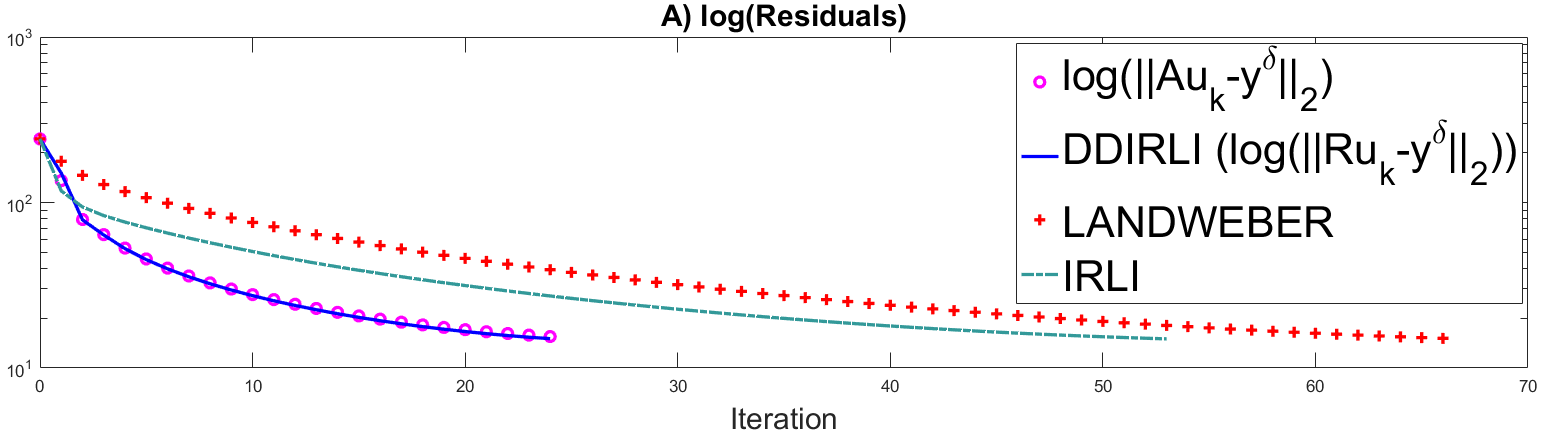}\vspace{1cm}
	\includegraphics[scale=0.5]{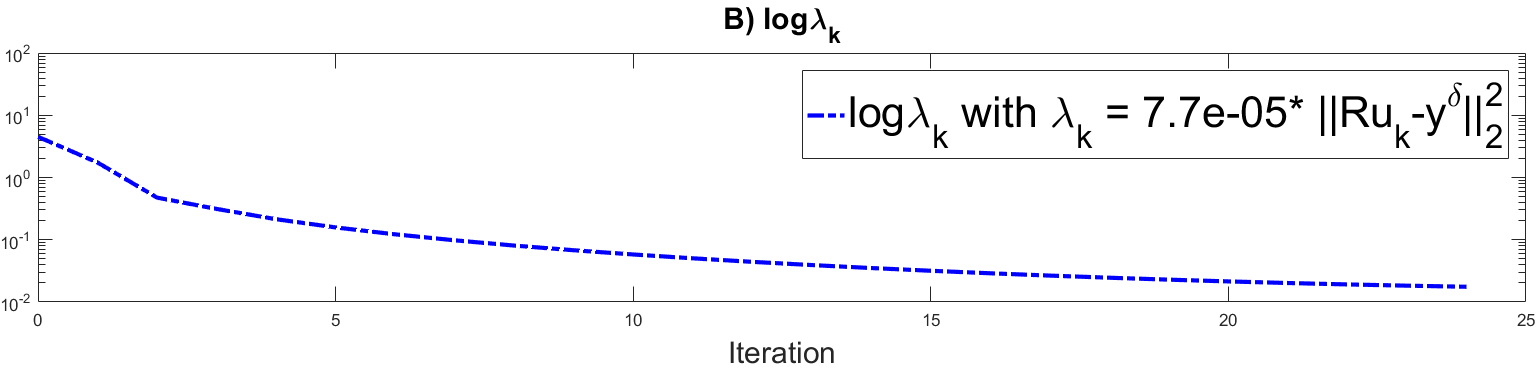}\vspace{1cm}
	\includegraphics[scale=0.5]{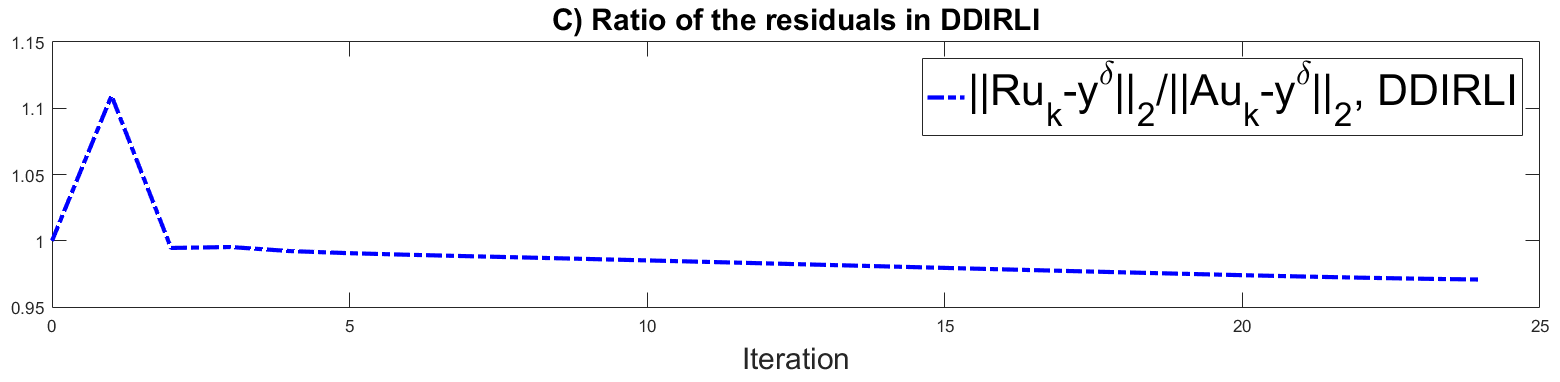}\vspace{1cm}
	\includegraphics[scale=0.5]{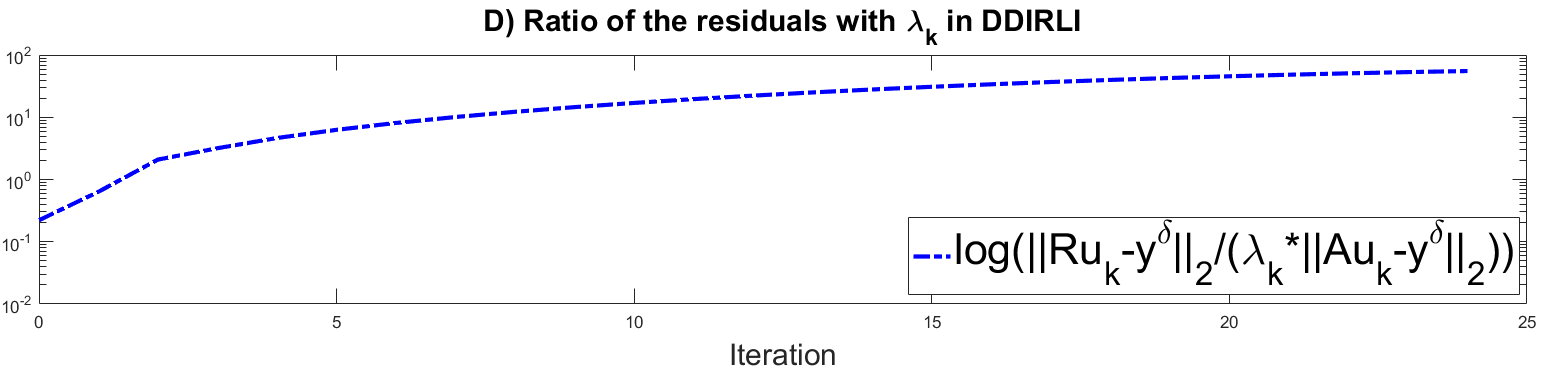}
	\caption{Test 2b. A) Residual in logarithmic scale at each step implementing \eqref{eq:ModLand_R}, \eqref{eq:Land_R} and \eqref{eq:IRLI_classic}; B) plot of $\log(\lambda_{k})$; C) plot of the ratio $\norm{Ru_k-\yd}_2/\norm{Au_k-\yd}_2$; D) plot in logarithmic scale of the ratio $\norm{Ru_k-\yd}_2/(\lambda_{k} \norm{Au_k-\yd}_2)$.}\label{fig: test2b_2}
\end{figure}

\newpage
{
	\begin{table}[!h]
		\caption{Test 3a. Left part: Parameters used in the test. Right part: some of the results of the test.}\label{tab: table3a}
		\begin{center}
			{\renewcommand{\arraystretch}{1.2}
				\begin{tabular}{|l|c|c|c||c|c|c|}
					\cline{1-7} & & & & & & \vspace{-0.2cm} \\
					Method & $\sigma^2$-noise& $\delta$ & $\tau$ & Iterations & Comp. Time (s) & $\frac{\norm{u_{\textrm{true}} - u_{\textrm{rec}}}_2}{\norm{u_{\textrm{true}}}_2}$\\ \hline  
					{DDIRLI} &  &  &  & 28 & 1,0690521
					& 0,259157333
					\\ \cline{1-1}\cline{5-7}
					{IRLI} & 0.03  & 2,863089254 & 5  & 56 & 1,1251772 & 0,236613193
					\\ \cline{1-1}\cline{5-7} 
					LANDWEBER &  & &  & 90 &1,8927882  & 0,312031414
					\\ \cline{1-1}\cline{3-4}\cline{5-7} 
					FBP &  & / & / & / & 0,024378
					& 0,475620179
					\\ \hline   
				\end{tabular}
			}
		\end{center}   
	\end{table} 
}
\vspace{2cm}
\begin{figure}[!h]
	\centering
	\includegraphics[scale=0.35]{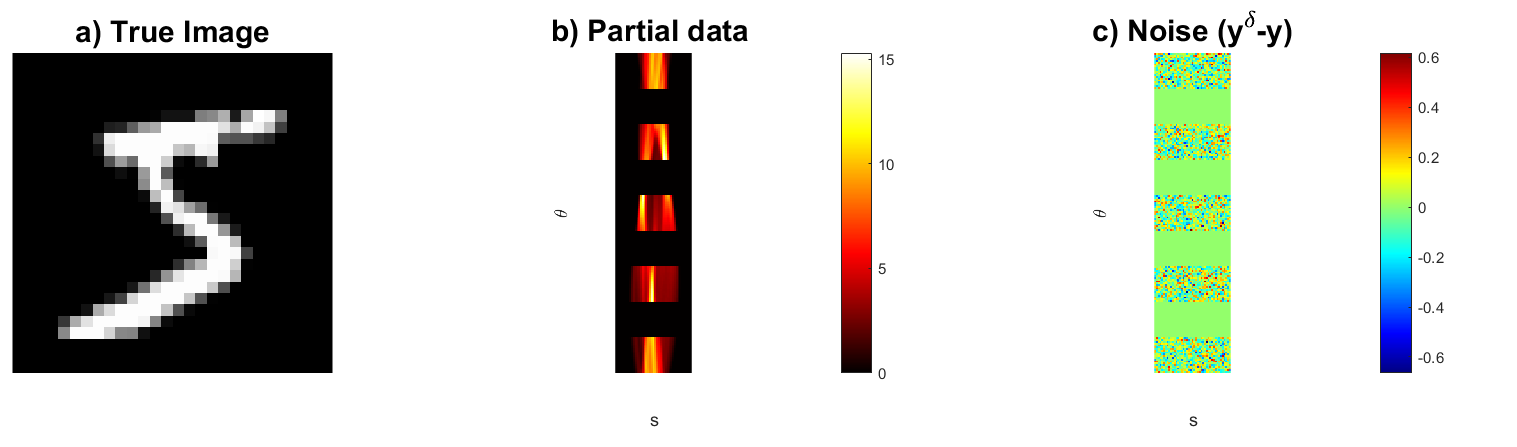}	\subfigure{\includegraphics[scale=0.37]{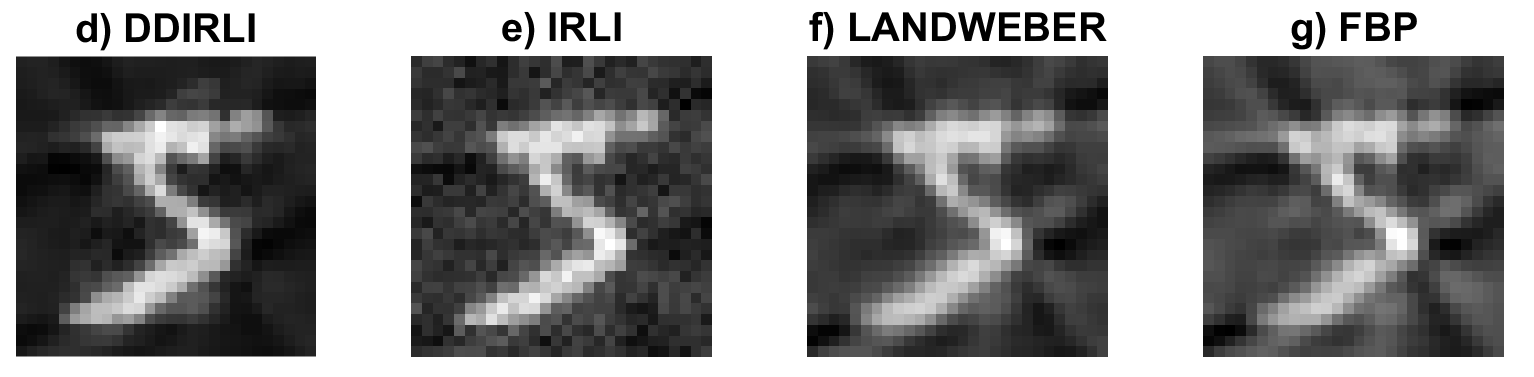}}
	\caption{Test 3a. Reconstruction of an image from the training set, which has been used to create $A$, from partial data. a) Image to be reconstructed; b) Plot of the cropped sinogram of the true image, i.e., $y$; c) Plot of the gaussian noise, i.e., $\yd-y$; d)-g) Reconstructions by the different methods. DDIRLI and IRLI provide better reconstructions than Landweber and FBP.}\label{fig: test3a}
\end{figure}
\newpage
\begin{figure}[!h]
	\centering
	\vspace{1cm}
	\includegraphics[scale=0.4]{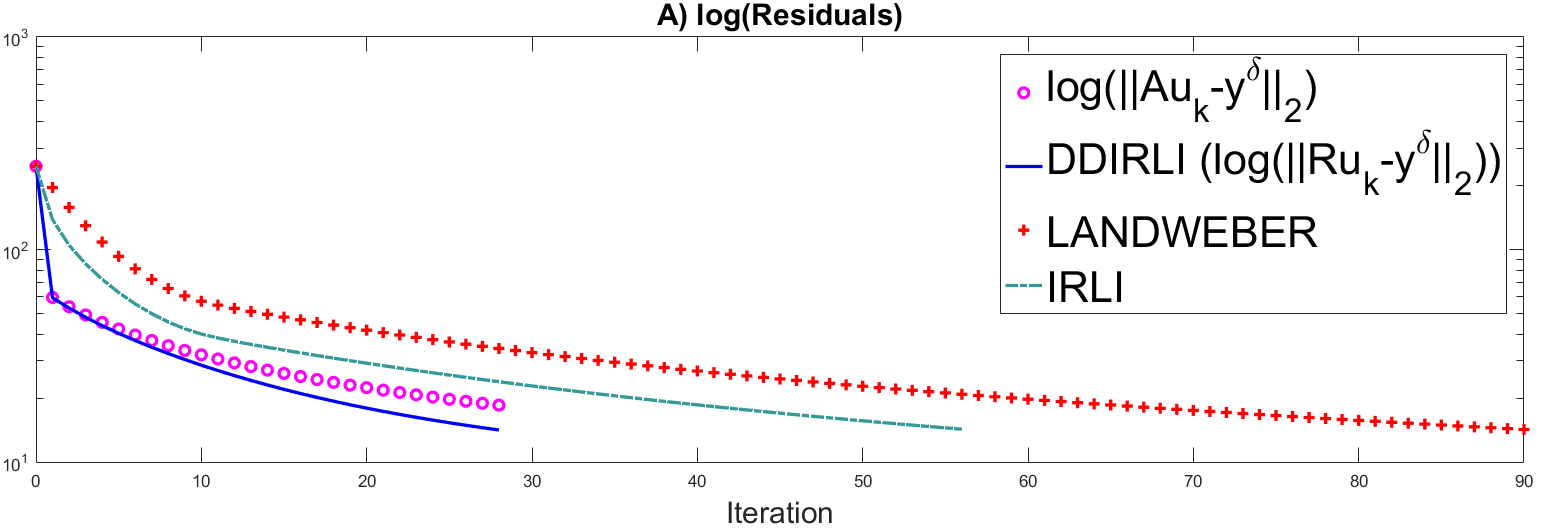}\vspace{1cm}
	\includegraphics[scale=0.4]{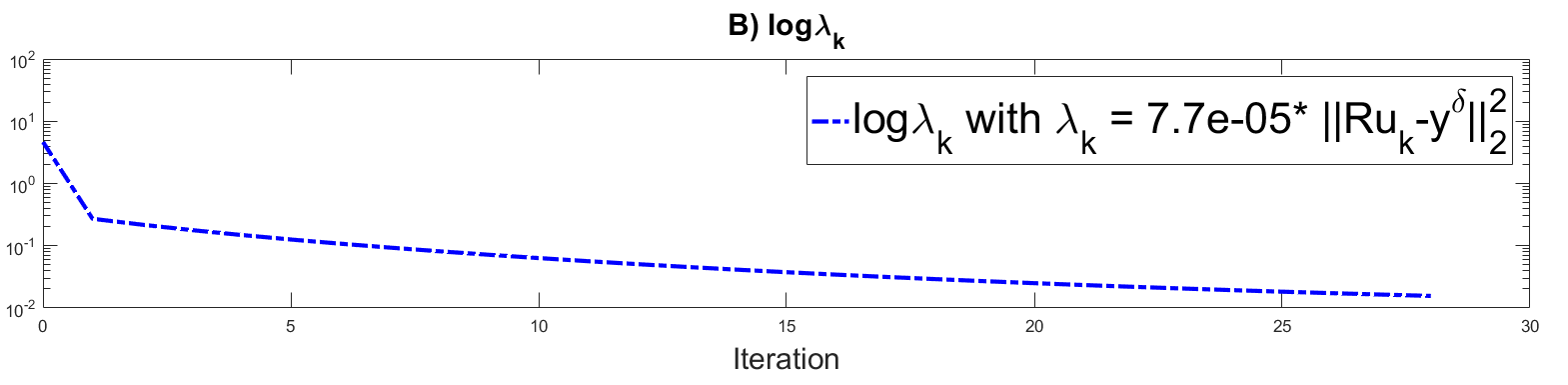}\vspace{1cm}
	\includegraphics[scale=0.4]{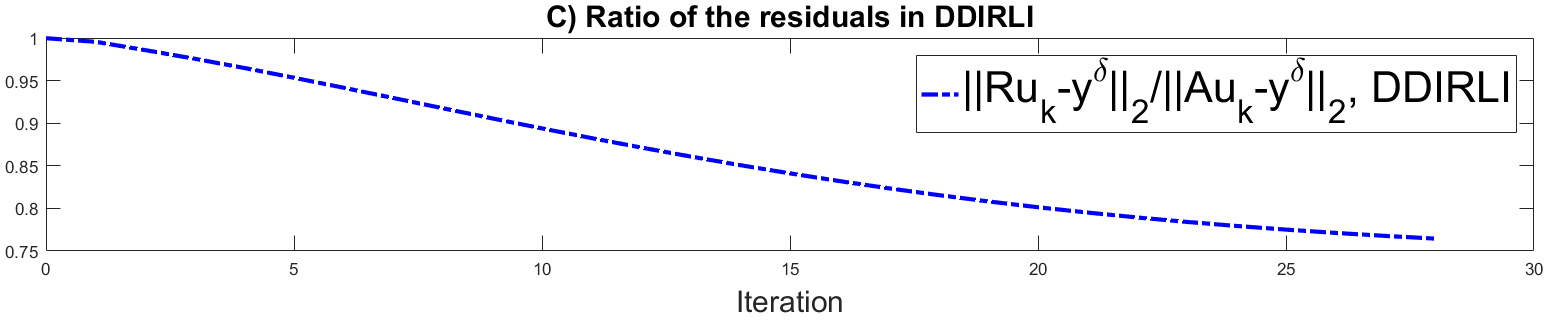}\vspace{1cm}
	\includegraphics[scale=0.4]{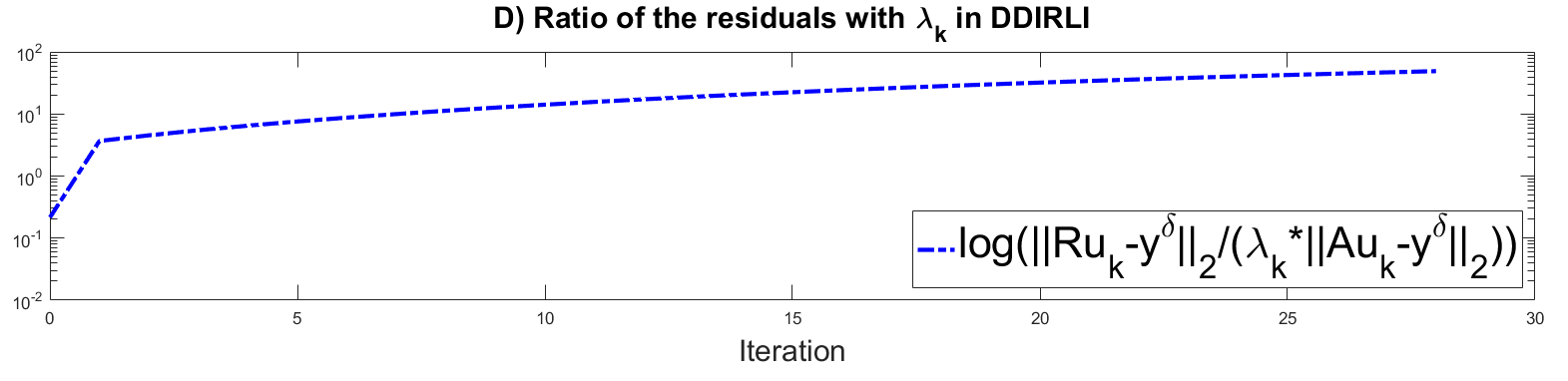}
	\caption{Test 3a. A) Residual in logarithmic scale at each step implementing \eqref{eq:ModLand_R}, \eqref{eq:Land_R} and \eqref{eq:IRLI_classic}; B) plot of $\log(\lambda_{k})$; C) plot of the ratio $\norm{Ru_k-\yd}_2/\norm{Au_k-\yd}_2$; D) plot in logarithmic scale of the ratio $\norm{Ru_k-\yd}_2/(\lambda_{k} \norm{Au_k-\yd}_2)$.}\label{fig: test3_2a}
\end{figure}

\newpage
{
	\begin{table}[!h]
		\caption{Test 3b. Left part: Parameters used in the test. Right part: some of the results of the test.}\label{tab: table3b}
		\begin{center}
			{\renewcommand{\arraystretch}{1.2}
				\begin{tabular}{|l|c|c|c||c|c|c|}
					\cline{1-7} & & & & & & \vspace{-0.2cm} \\
					Method & $\sigma^2$-noise& $\delta$ & $\tau$ & Iterations & Comp. Time (s) & $\frac{\norm{u_{\textrm{true}} - u_{\textrm{rec}}}_2}{\norm{u_{\textrm{true}}}_2}$\\ \hline  
					{DDIRLI} &  &  &  & 29 & 1,1999311
					& 0,270013425
					\\ \cline{1-1}\cline{5-7}
					{IRLI} & 0.03  & 2,760643462
					& 5  & 55 & 1,1808313
					& 0,240562351
					\\ \cline{1-1}\cline{5-7} 
					LANDWEBER &  & &  & 80 &1,654519  & 0,289219401
					\\ \cline{1-1}\cline{3-4}\cline{5-7} 
					FBP &  & / & / & / & 0,0210105
					& 0,427565297
					\\ \hline   
				\end{tabular}
			}
		\end{center}   
	\end{table} 
}
\vspace{2cm}
\begin{figure}[!h]
	\centering
	\includegraphics[scale=0.35]{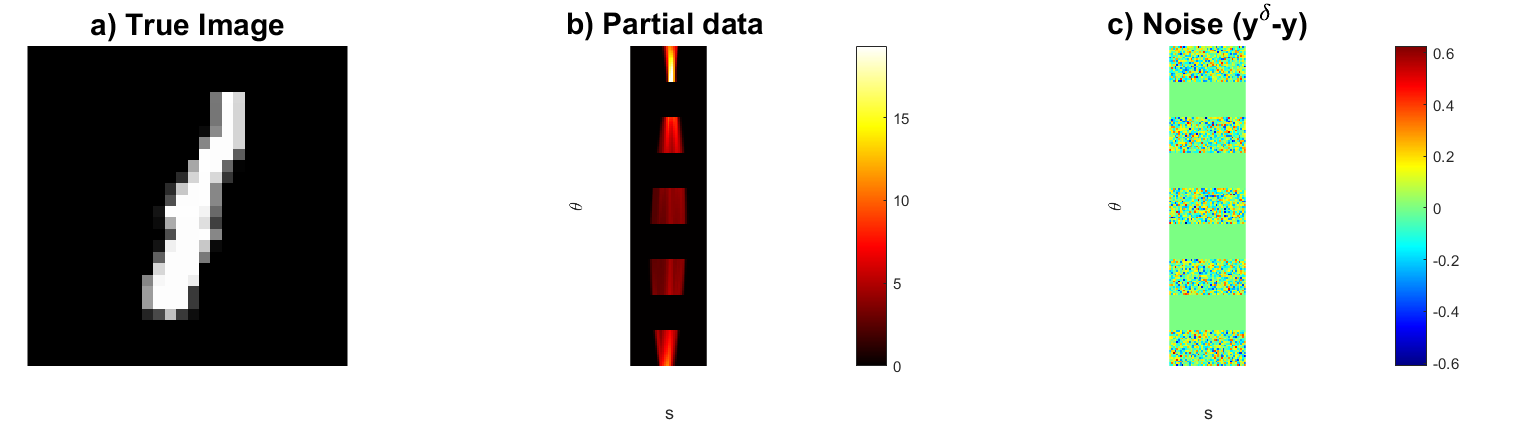}	\subfigure{\includegraphics[scale=0.37]{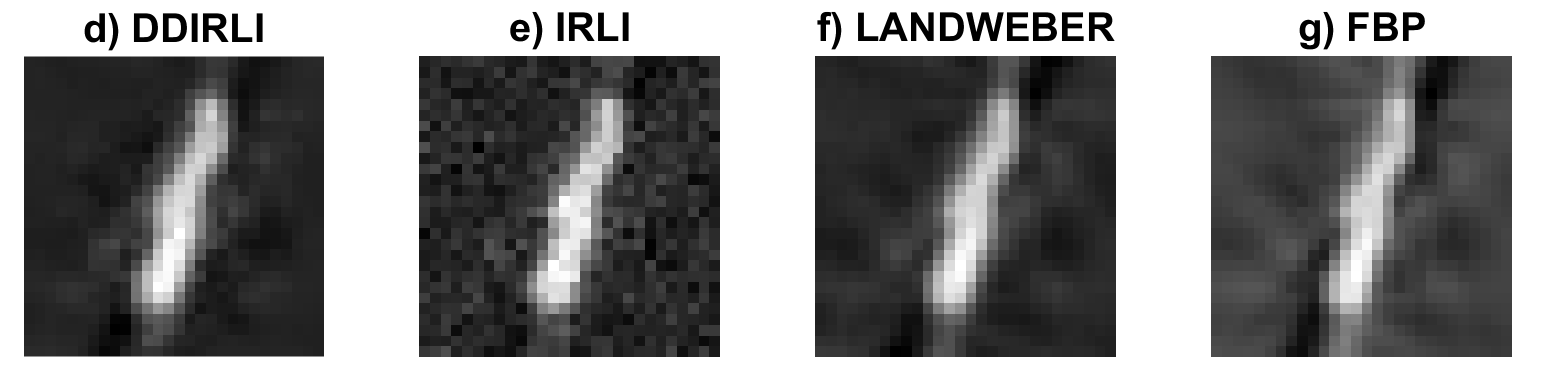}}
	\caption{Test 3b. Reconstruction of an image from the validation set not used to create $A$ in \eqref{eq: def_A_calc}, from partial data. We used the same $150$ images of Test 2b to create $A$. a) Image to be reconstructed; b) Plot of the cropped sinogram of the true image, i.e., $y$; c) Plot of the gaussian noise, i.e., $\yd-y$; d)-g) Reconstructions by the different methods.}\label{fig: test3b}
\end{figure}
\newpage
\begin{figure}[!h]
	\centering
	\vspace{1cm}
	\includegraphics[scale=0.4]{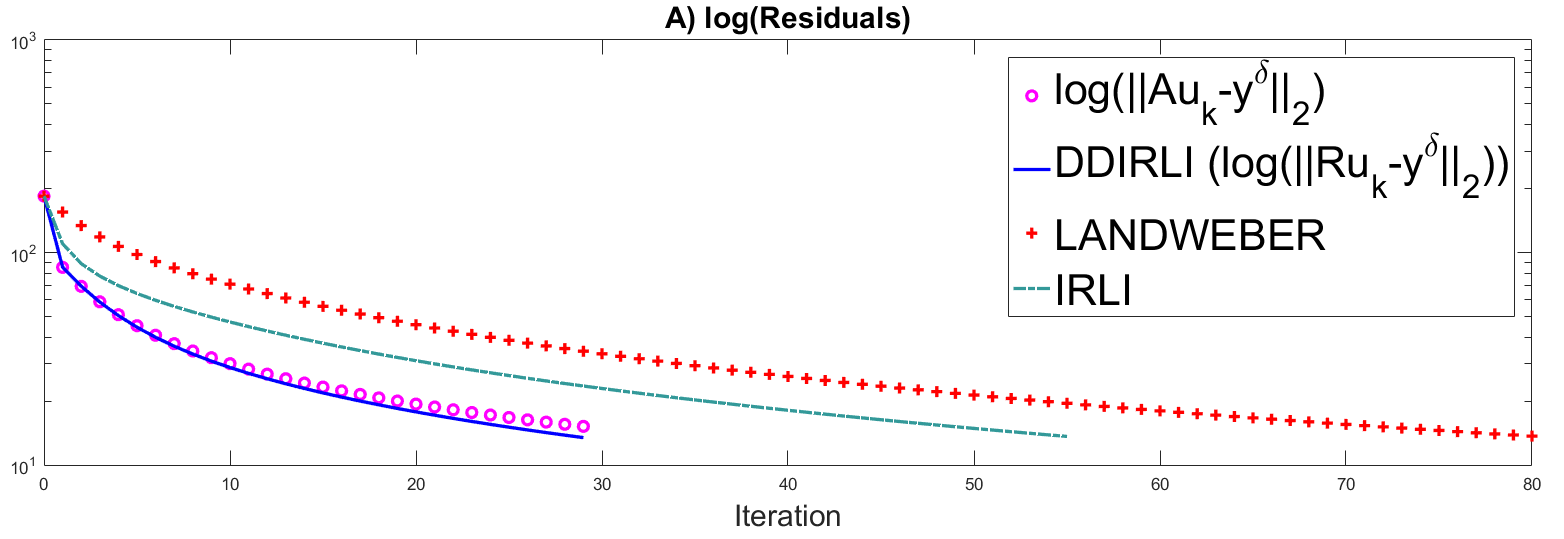}\vspace{1cm}
	\includegraphics[scale=0.4]{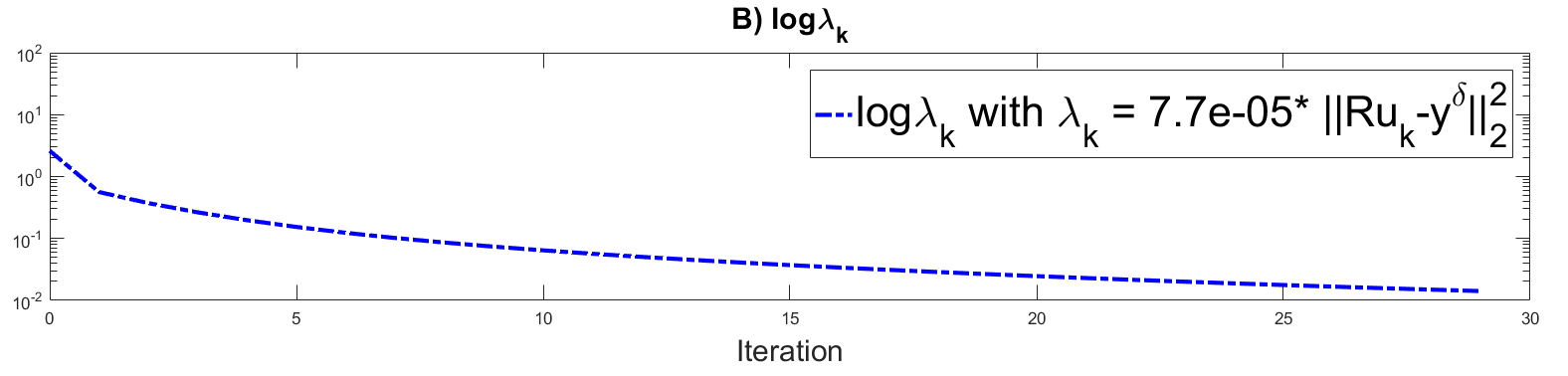}\vspace{1cm}
	\includegraphics[scale=0.4]{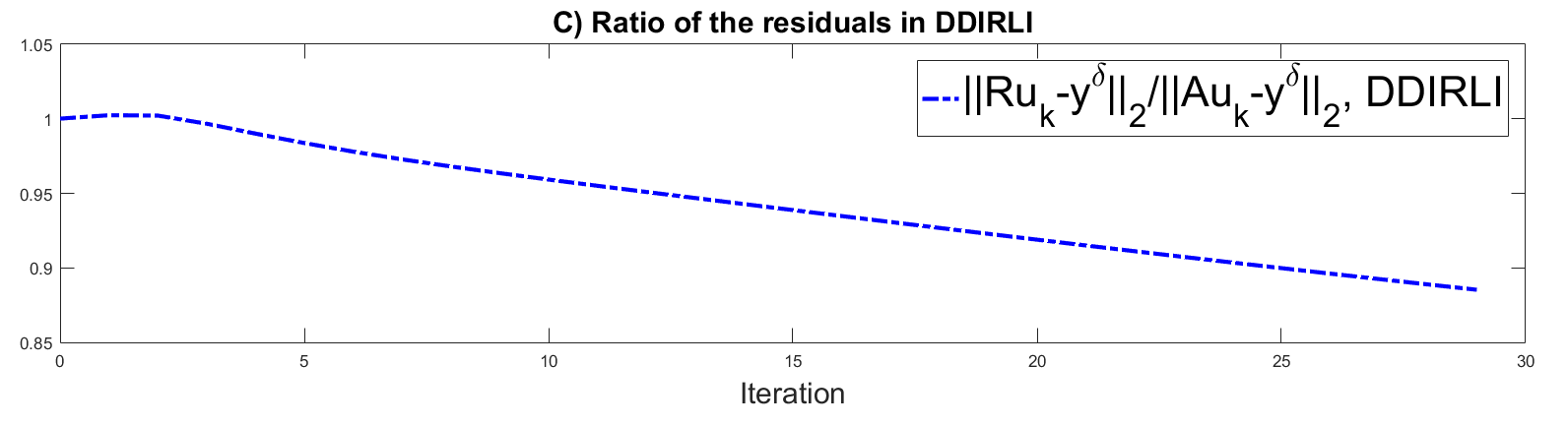}\vspace{1cm}
	\includegraphics[scale=0.4]{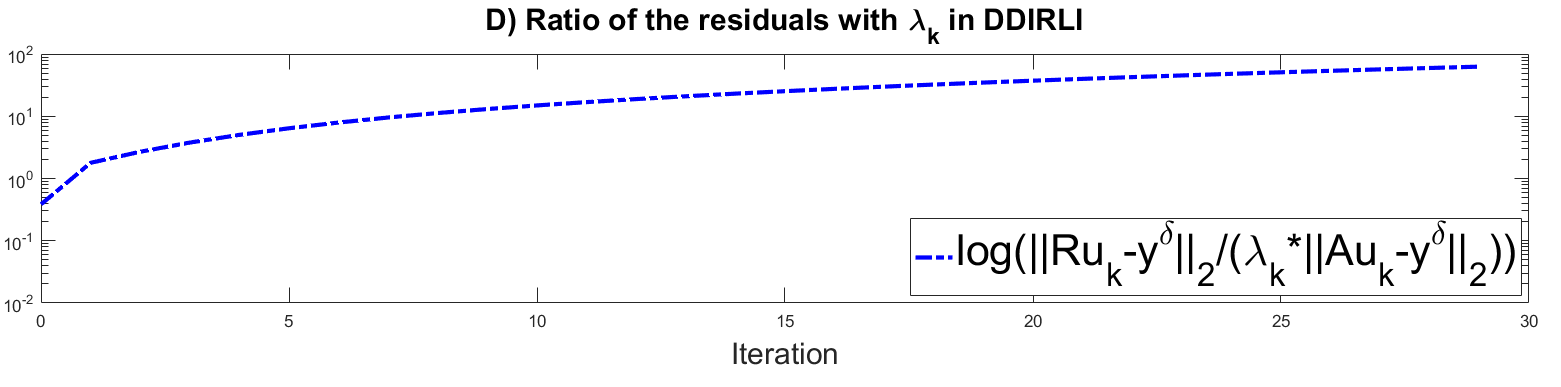}
	\caption{Test 3b. A) Residual in logarithmic scale at each step implementing \eqref{eq:ModLand_R}, \eqref{eq:Land_R} and \eqref{eq:IRLI_classic}; B) plot of $\log(\lambda_{k})$; C) plot of the ratio $\norm{Ru_k-\yd}_2/\norm{Au_k-\yd}_2$; D) plot in logarithmic scale of the ratio $\norm{Ru_k-\yd}_2/(\lambda_{k} \norm{Au_k-\yd}_2)$.}\label{fig: test3_2b}
\end{figure}

\newpage
{
	\begin{table}[!h]
		\caption{Test 4. Left part: Parameters used in the test. Right part: some of the results of the test.}\label{tab: table4}
		\begin{center}
			{\renewcommand{\arraystretch}{1.2}
				\begin{tabular}{|l|c|c|c||c|c|c|}
					\cline{1-7} & & & & & & \vspace{-0.2cm} \\
					Method & $\sigma^2$-noise& $\delta$ & $\tau$ & Iterations & Comp. Time (s) & $\frac{\norm{u_{\textrm{true}} - u_{\textrm{rec}}}_2}{\norm{u_{\textrm{true}}}_2}$\\ \hline  
					{DDIRLI} &  &  &  & 26 & 1,4472282
					& 0,399702146
					\\ \cline{1-1}\cline{5-7}
					{IRLI} & 0.5  & 13,56737965
					& 1.1  & 34 & 0,6318276
					& 0,323027126
					\\ \cline{1-1}\cline{5-7} 
					LANDWEBER &  & &  & 56 &1,0820124  & 0,410787776
					\\ \cline{1-1}\cline{3-4}\cline{5-7} 
					FBP &  & / & / & / & 0,0206258
					& 0,382431698
					\\ \hline   
				\end{tabular}
			}
		\end{center}   
	\end{table} 
}
\vspace{2cm}
\begin{figure}[!h]
	\centering
	\includegraphics[scale=0.35]{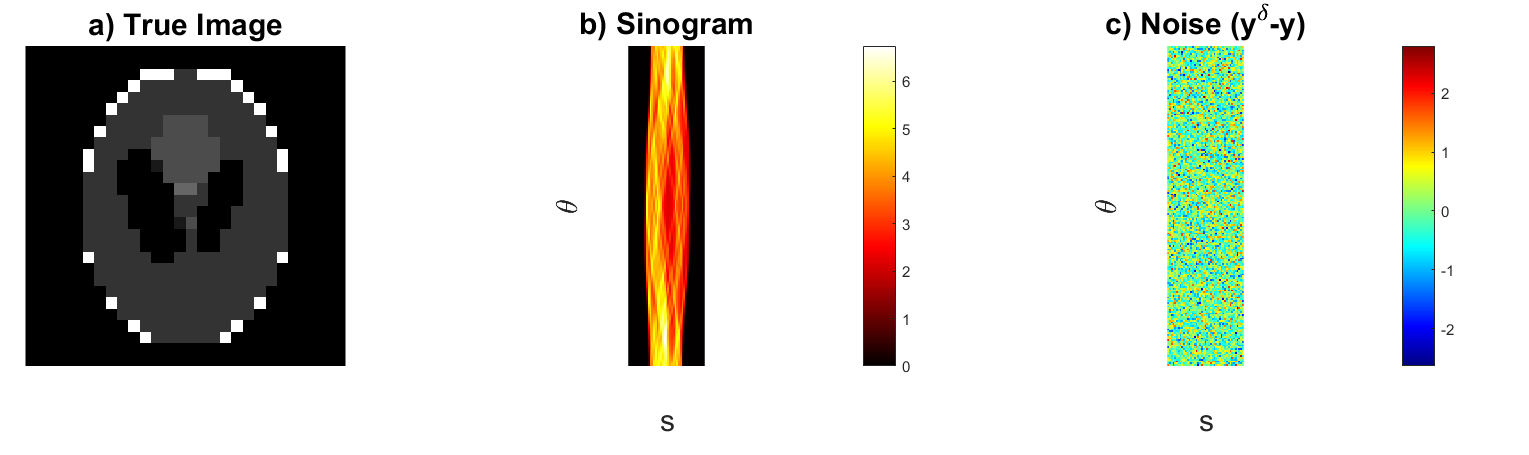}	\subfigure{\includegraphics[scale=0.37]{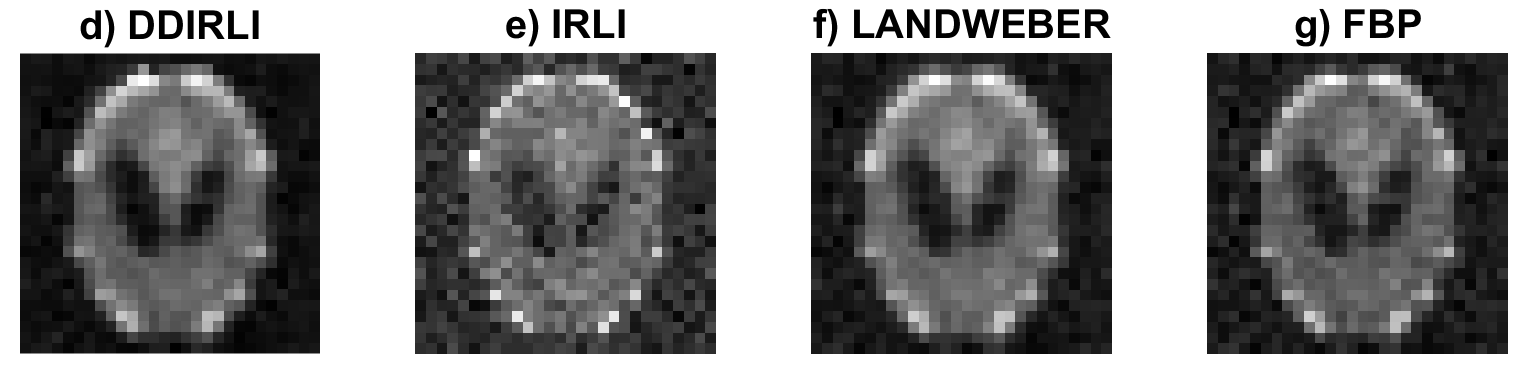}}
	\caption{Test 4. Reconstruction of the Shepp Logan phantom. An amount of $600$ images from the training set have been used to create $A$. a) Image to be reconstructed; b) Plot of the cropped sinogram of the true image, i.e., $y$; c) Plot of the gaussian noise, i.e., $\yd-y$; d)-g) Reconstructions by the different methods.}\label{fig: test4a}
\end{figure}
\newpage
\begin{figure}[!h]
	\centering
	\vspace{1cm}
	\includegraphics[scale=0.4]{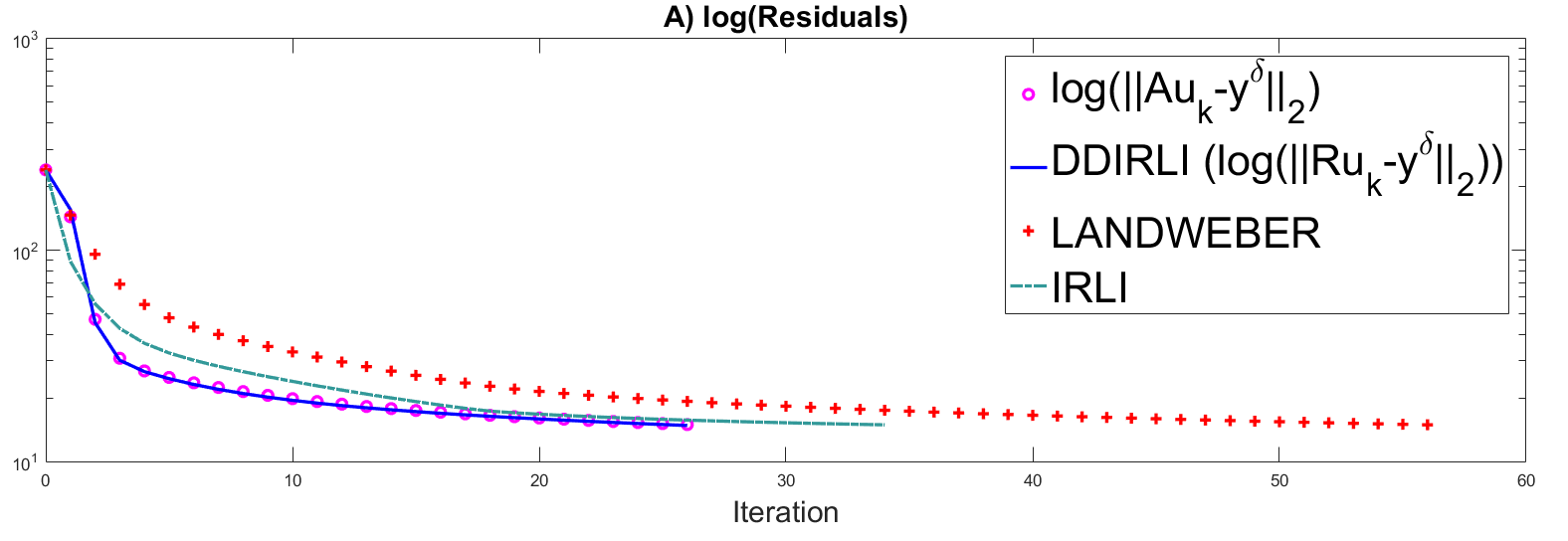}\vspace{1cm}
	\includegraphics[scale=0.4]{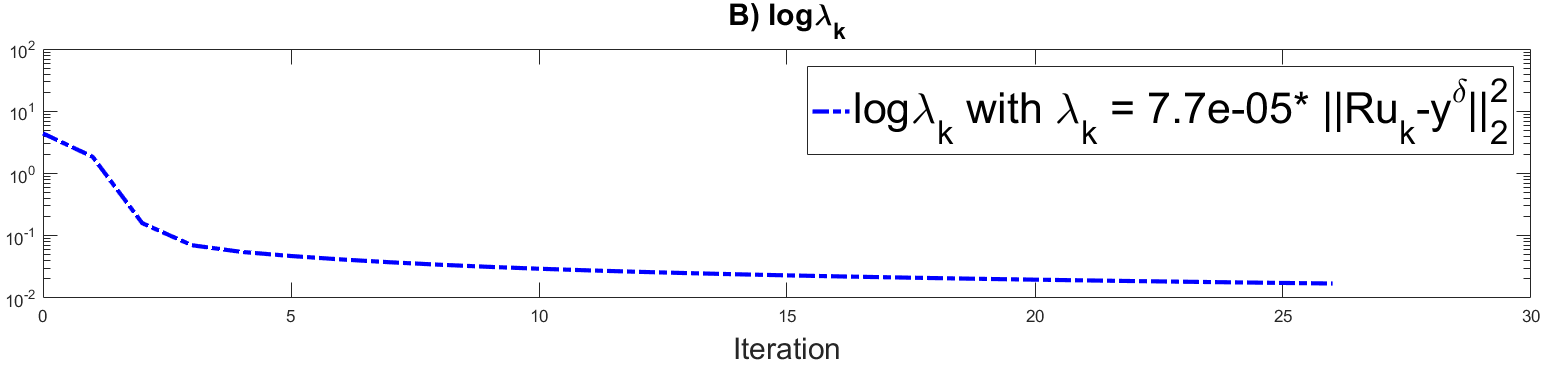}\vspace{1cm}
	\includegraphics[scale=0.4]{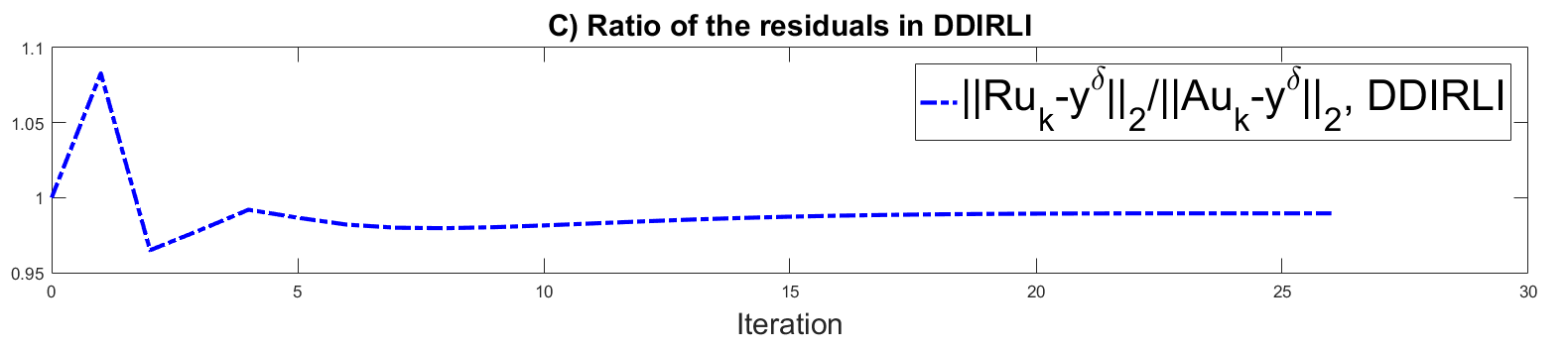}\vspace{1cm}
	\includegraphics[scale=0.4]{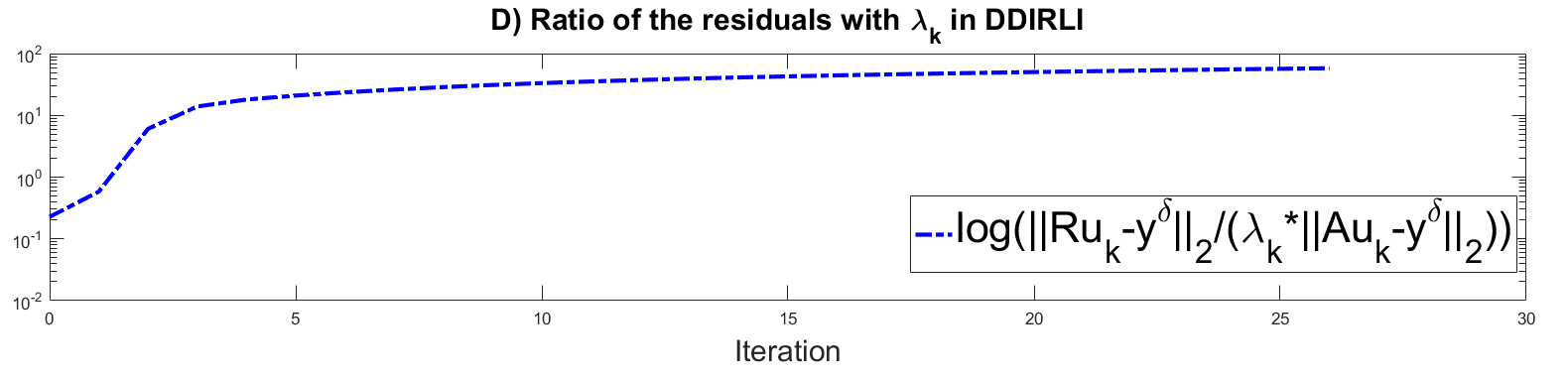}
	\caption{Test 4. A) Residual in logarithmic scale at each step implementing \eqref{eq:ModLand_R}, \eqref{eq:Land_R} and \eqref{eq:IRLI_classic}; B) plot of $\log(\lambda_{k})$; C) plot of the ratio $\norm{Ru_k-\yd}_2/\norm{Au_k-\yd}_2$; D) plot in logarithmic scale of the ratio $\norm{Ru_k-\yd}_2/(\lambda_{k} \norm{Au_k-\yd}_2)$.}\label{fig: test4b}
\end{figure}

\subsection{Nonlinear Operator - Schlieren Model}
We briefly recall the mathematical model behind the Schlieren tomography, which we take as prototype of a nonlinear problem. For the physical model, i.e., the Schlieren optical system and details on data acquisition see, e.g., \cite{ChaPel95,Kow08,PitGreenYuKin94}.

Let $\eta_i\in S^1$, for $i=0,\cdots, l-1$ be a set of recording directions, that is 
\begin{equation*}
\eta_i=\eta(\theta_i)=(\cos\theta_i,\sin\theta_i),
\end{equation*}
where $\theta_i\in[0,\pi)$, for $i=0,\cdots,l-1$, is the angle.	
Taking $B_1(0)\subset \mathbb{R}^2$, we define the Schlieren operator in the direction of $\eta_i$ as the square of the Radon transform, i.e., 
\begin{equation*}
F_i(u):=R^2_i(u), \qquad i=0,\cdots, l-1,
\end{equation*}
where $R_i$ is the Radon transform along the direction $\eta_i$. 
Physically speaking, the function $u$ denotes a pressure.
To reconstruct $u$, we need to solve the system of equations
\begin{equation}\label{eq: oper_F_Schlieren}
F(u):=(F_0(u),\cdots,F_{l-1}(u))=\yd:=(\yd_0,\cdots,\yd_{l-1}).
\end{equation}
It was shown in \cite{HalKowLeitSch07} that each operator $F_i$, for $i=0,\cdots, l-1$ is continuous and Fr\'echet differentiable from $H^1_0(B_1(0))$ to $L^2([-1,1])$ and 
\begin{equation*}
F'_i(u)h=2R_i(u)R_i(h), \qquad \forall h\in H^1_0(B_1(0)).
\end{equation*}  
The adjoint of $F'_i(u)$ is given by 
\begin{equation*}
\begin{aligned}
F'_i(u)^*: L^2([-1,1])&\to H^1_0(B_1(0))\\
g & \to f(g)
\end{aligned}
\end{equation*}
with $f(g)$ the solution to 
\begin{equation}\label{eq: adj_Schlieren}
(I-\Delta)f(g)=2R^*_i(R_i(u)g),
\end{equation}
where $I$ is the identity operator, $\Delta$ is the Laplace operator on $H^1_0(B_1(0))$ and $R^*_i$ denotes the adjoint of $R_i$ as operator from $L^2(B_1(0))$ to $L^2([-1,1])$, i.e,
\begin{equation*}
\begin{aligned}
R^*_i: L^2([-1,1])&\to L^2([B_1(0)])\\
w(x) &\to w(\langle \eta_i, x\rangle). 
\end{aligned}
\end{equation*} 
We stress that all the numerical simulations of this section are performed without the smoothing operator $(I-\Delta)$, contained in \eqref{eq: adj_Schlieren}, since both the iteratively regularized Landweber schemes and the Landweber iteration are regularizing.

Therefore, the {data driven} iteratively regularized Landweber iteration ({DDIRLI}) for the Schlieren operator is given by
\begin{equation}\label{eq:ModLand_Schl}
\ukp := \uk - \alpha \sum_{i=0}^{l-1}\left[R_i^*\bigl(R_i(u_k) (R_i^2(\uk)-\yd_i \bigr))\right]- \lambda_k \tilde{A}^*\bigl(\tilde{A}\uk -\yd \bigr), \qquad k\in\N,
\end{equation} 	  	
where $\alpha$ is a relaxation parameter such that the norm of the operator is less or equal to 1.	
Now, we can create the matrix $A$ for the Equation  \eqref{eq: oper_F_Schlieren} using the procedure described in Section  \ref{sec: num_exp}, see equation \eqref{eq: def_A}. For this purpose, we use the following set of data and parameters:
\begin{enumerate}[(a)]
	\item \label{item: point1S} $180$ different equally distributed angles $\theta$ within the interval $[0,\pi)$; specifically, we consider the set $\Theta=\Big\{0,\frac{\pi}{180},\frac{\pi}{90},\cdots,\frac{179\pi}{180}\Big\}$;
	\item \label{item: point2S} 10 input images of dimensions $110\times 110$ of synthetic pressure functions $u^{(j)}$, $j=1,\cdots,10$,  and the related output data. See Figure \ref{fig: data_sch2} for the dataset used in numerical Test 1S, 2aS and 2bS. See Figure \ref{fig: data_sch1} for the dataset used in numerical Test 3S. Note that the difference between the two datasets is only on the values of the pressure in some pictures.
\end{enumerate}  
We create the matrix $U$ and $Y$, as defined in \eqref{eq: def_A}, utilizing the data in \eqref{item: point2S}. 

{Again, we compare results of \eqref{eq:ModLand_Schl}} with the Landweber scheme, i.e.,  	
\begin{equation}\label{eq:Land_Schl}
\ukp := \uk - \alpha \sum_{i=0}^{l-1}\left[R_i^*\bigl(R_i(u_k) (R_i^2(\uk)-\yd_i \bigr))\right], \qquad k\in\N
\end{equation} 	  	
{and with	the iteratively regularized Landweber iteration (IRLI)
	\begin{equation}\label{eq:IRLI_Land_Schl}
	\ukp := \uk - \alpha \sum_{i=0}^{l-1}\left[R_i^*\bigl(R_i(u_k) (R_i^2(\uk)-\yd_i \bigr))\right]-\beta_k(u_k-u^{(0)}), \qquad k\in\N,
	\end{equation} 	 
	where, according to the parameters rules in \cite[Section 3.2]{KalNeuSch08}, we choose  $\beta_k=\left(\frac{1}{4}\right)^{k+1}$.} 
Here, we specify the hypotheses and the set of parameters which are employed in our numerical experiments for the schemes \eqref{eq:ModLand_Schl}, \eqref{eq:Land_Schl} and \eqref{eq:IRLI_Land_Schl}:
\begin{enumerate}
	\item We choose $\alpha=10^{-8}$ and $\lambda_k=10^{-7}\norm{F(u_k)-\yd}^2_{\infty}$;
	\item As stopping rule we utilize the discrepancy principle \eqref{eq: discrep_princ}, where the choice of $\tau$ will be specified for each test (see tables below), or a maximum number of 400 iterates;
	\item Each vector $\yd_i$, for $i=0,\cdots,179$, of synthetic data is generated by adding a gaussian distributed noise with zero mean and variance $\sigma^2$, equal to 100 or 10, to each vector $y_i$, , for $i=0,\cdots,179$, which represents the exact data;
	\item we define $\delta_S:=\max\left(\norm{\yd_i-y_i}_2\right)$, for $i=0,\cdots,179$;
	\item The initial guess is the constant function $u_0=0.01$;
\end{enumerate} 	
We consider three different numerical experiments: in \textbf{Test 1S} we reconstruct an image which has been utilized to create the matrix $A$. In \textbf{Test 2aS}/\textbf{2bS} we propose the reconstruction of an image containing three {regions} with negative values. In details, in Test 2aS we show the reconstruction adding noise in the data. In Test 2bS we consider the case of partial noisy data. In \textbf{Test 3S}, we consider the case where the pressure to be reconstructed is positive in some regions and negative in others.

{We stress that the Schlieren imaging has an inherent nonuniqueness issue, since the sign of the true solution is undetermined.} {This will be clear in the numerical outcomes of the Landweber iteration.}

As in the previous section, for each test we provide a table containing, {on the left} the values of the main parameters utilized in the generation of the data and, {on the right some results such as the numeber of iterations, the computational time and finally the relative error. See Tables \ref{tab: table1S}, \ref{tab: table2aS}, \ref{tab: table2bS} and \ref{tab: table3aS}.}

\ \\ \ \\
\textbf{Test 1S.} (Figure \ref{fig: Stest1}). \textit{Target: reconstruct one of the image which has been utilized to create the matrix $A$}. See Figure \ref{fig: Stest1} for the ``True Image''. 
All the synthetic data $\yd_i$, for $i=0,\cdots,179$, are obtained by the real data $y_i$, for $i=0,\cdots,179$,  adding a gaussian distributed noise of zero mean and variance $\sigma^2=100$. The noise in the synthetic data is shown in Figure \ref{fig: Stest1} b) in aggregate form. 
For the discrepancy principle \eqref{eq: discrep_princ} we choose $\tau=2$. See Table \ref{tab: table1S}.
We observe, from Figure \ref{fig: Stest1} and Figure \ref{fig: Stest1_2}, how the presence of $A$ (and the fact that the information of the "true image" are included in $A$) ``helps'' the iterates \eqref{eq:ModLand_Schl} in a better reconstruction of the value of the pressure than the Landweber iterates \eqref{eq:Land_Schl}. However, in Figure \ref{fig: Stest1} c) we can observe some evident artifacts which come from the other a-priori information contained in $A$.

\textbf{Test 2aS/2bS.} (Figures \ref{fig: Stest2a} and \ref{fig: Stest2b}). \textit{Target: reconstruct one of the image which has been utilized to create the matrix $A$, with regions with only negative values}. See Figure \ref{fig: Stest2a} for the ``True Image''. 
All the synthetic data $\yd_i$, for $i=0,\cdots,179$, are obtained by the real data $y_i$, for $i=0,\cdots,179$,  adding a gaussian distributed noise of zero mean and variance $\sigma^2=100$ in Test 2aS, and $\sigma^2=10$ in Test 2bS. The noise in the synthetic data for the two tests is shown in Figure \ref{fig: Stest2a} b) and Figure \ref{fig: Stest2b} in aggregate form. 
For the discrepancy principle \eqref{eq: discrep_princ} we choose $\tau=3$ in Test 2aS and $\tau=6$ in Test 2bS. See Table \ref{tab: table2aS} and Table \ref{tab: table2bS}.
Again, in Test 2aS we notice how the presence of $A$ ``helps'' the iterates \eqref{eq:ModLand_Schl} to get a better reconstruction of all the negative values of the pressure than the Landweber iterates \eqref{eq:Land_Schl}, see Figure \ref{fig: Stest2a} and Figure \ref{fig: Stest2a_2}. However, we also observe some evident artifacts which come from of the other a-priori in $A$.
In Test 2bS, we consider the case of limited data, which corresponds in choosing the data only of some directions in $\Theta$ (see (\ref{item: point1S}) for its definition) and set to zero the values of the other vectors related to the angles not chosen, see Figure \ref{fig: test2b} b) for the directions of the available data. Comparing the results of the two schemes, \eqref{eq:ModLand_Schl} and \eqref{eq:Land_Schl}, we can observe a good reconstruction for the iteratively regularized Landweber iteration.

\textbf{Test 3S.} (Figure \ref{fig: Stest3}). \textit{Target: reconstruct one of the image which has been used to create the matrix $A$, with regions with positive or negative values}. See Figure \ref{fig: Stest3} for the true image. 
All the synthetic data $\yd_i$, for $i=0,\cdots,179$, are obtained by the real data $y_i$, for $i=0,\cdots,179$, adding a gaussian distributed noise of zero mean ad variance $\sigma^2=100$. For the discrepancy principle \eqref{eq: discrep_princ} we choose $\tau=3$. See Table \ref{tab: table3aS}.

We stress that Figures \ref{fig: Stest1_2} c), \ref{fig: Stest2a_2} c), \ref{fig: Stest2b_2} c), \ref{fig: Stest3_2} c) show the ratios between the residuals of the data driven model and the model driven approach. A value greater than $1$ of the ratio implies that the data driven approach has a significant influence on the iteration process.

{As in the previous section, Figures \ref{fig: Stest1_2} C), \ref{fig: Stest2a_2} C), \ref{fig: Stest2b_2} C) and \ref{fig: Stest3_2} C) show the ratios between the residuals of the data driven model, $\norm{Au_k-\yd}$, and the model driven approach, $\norm{F(u_k)-\yd}$. {A value greater than $1$ of the ratio implies that the data driven approach has a significant influence on the iteration process. Figures \ref{fig: Stest1_2} D), \ref{fig: Stest2a_2} D), \ref{fig: Stest2b_2} D) and \ref{fig: Stest3_2} D) enlight how fast the product $\|Ru_k-y^{\delta}\|_2\  \|Au_k-y^{\delta}\|_2$ is going to zero.}}

\begin{figure}[!h]
	\centering
		\includegraphics[scale=0.7]{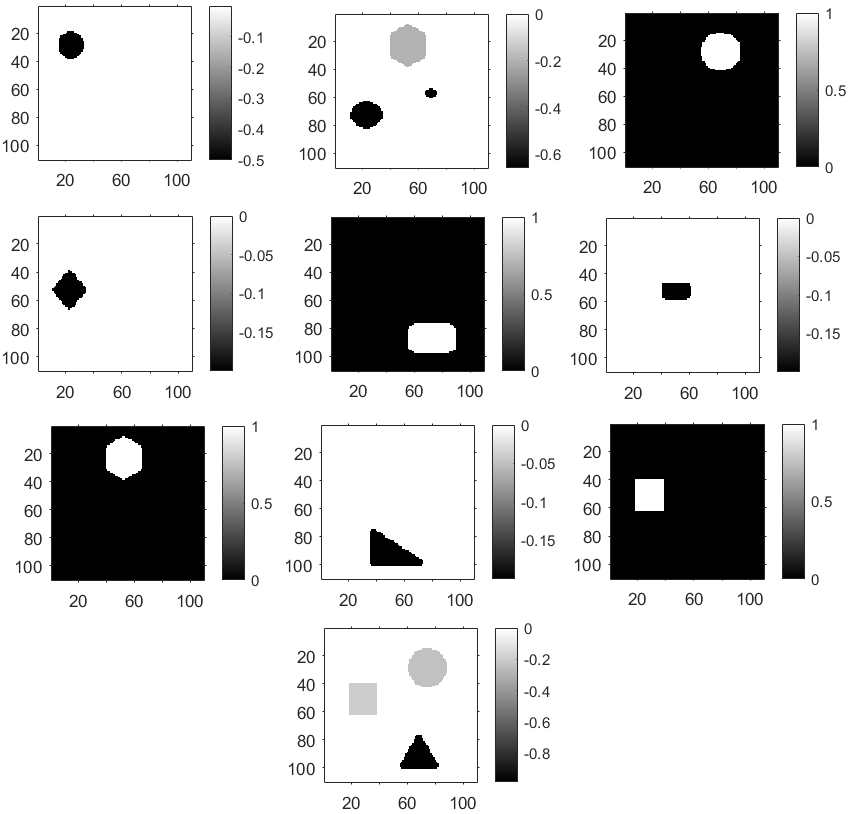}
	\caption{Dataset utilized to create the matrix $A$ in Tests 1S, 2aS and 2bS. }\label{fig: data_sch2}
\end{figure}
\newpage

\begin{figure}[!h]
	\centering
	\includegraphics[scale=0.7]{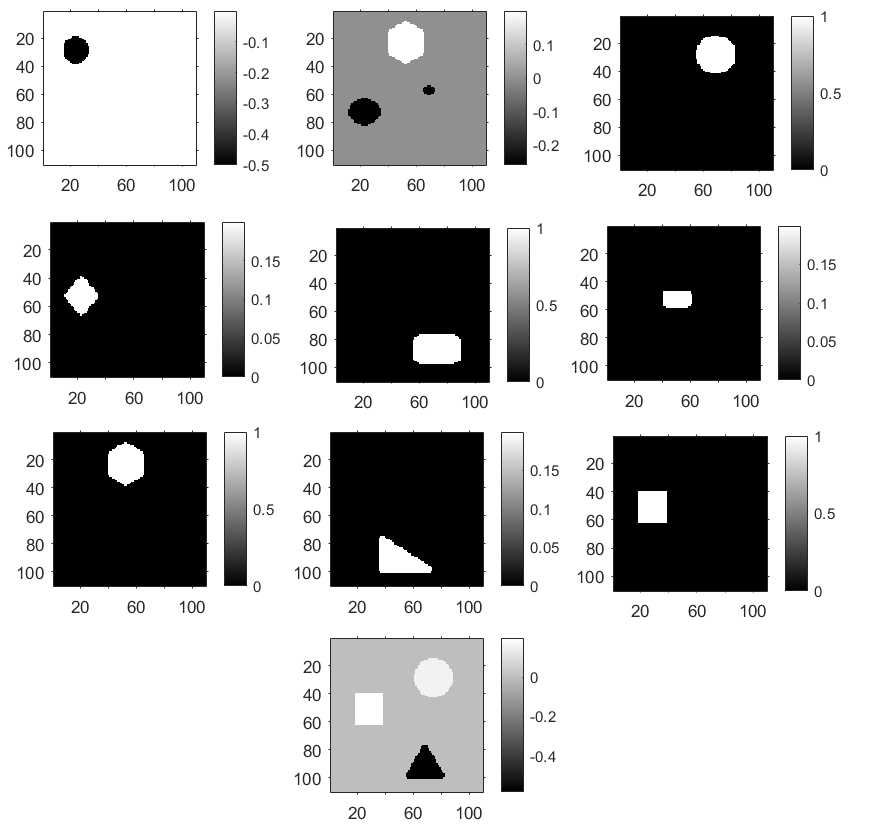}
	\caption{Dataset utilized to create the matrix $A$ in the following Test 3S.}\label{fig: data_sch1}	
\end{figure}
\newpage 

{\begin{table}[!h]\caption{Test 1S. Left part: Parameters used in the test. Right part: some of the results of the test.}\label{tab: table1S}
		\begin{center}
			{\renewcommand{\arraystretch}{1.2}
				\begin{tabular}{|l|c|c|c||c|c|c|}
					\cline{1-7} & & & & & & \vspace{-0.2cm} \\
					Method & $\sigma^2$-noise& $\delta$ & $\tau$ & Iterations & Comp. Time (s) & $\frac{\norm{u_{\textrm{true}} - u_{\textrm{rec}}}_2}{\norm{u_{\textrm{true}}}_2}$\\ \hline  
					{DDIRLI} &  &  &  & 73 & 34,9193152
					& 0,224935913
					\\ \cline{1-1}\cline{5-7}
					{IRLI} & 100  & 141,519952
					& 2  & 128 & 7,5506274
					& 0,233332376
					\\ \cline{1-1}\cline{5-7} 
					LANDWEBER &  & &  & 190 &11,3537075
					& 0,247679701
					\\ \hline
				\end{tabular}
			}
		\end{center}   
	\end{table} 
}
\vspace{2cm}
\begin{figure}[!h]
	\centering
	\subfigure{
		\includegraphics[scale=0.35]{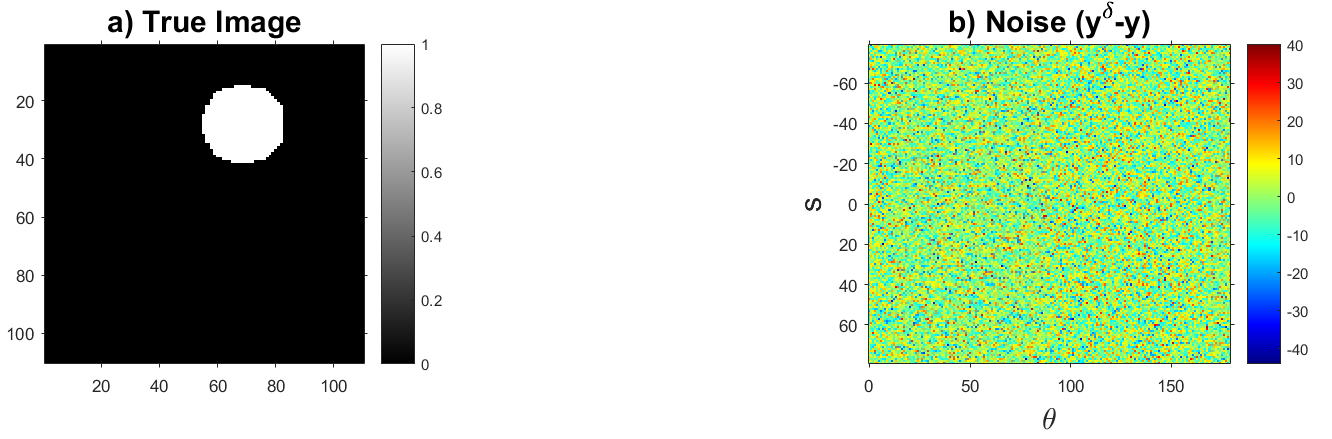}}
	\subfigure{\includegraphics[scale=0.37]{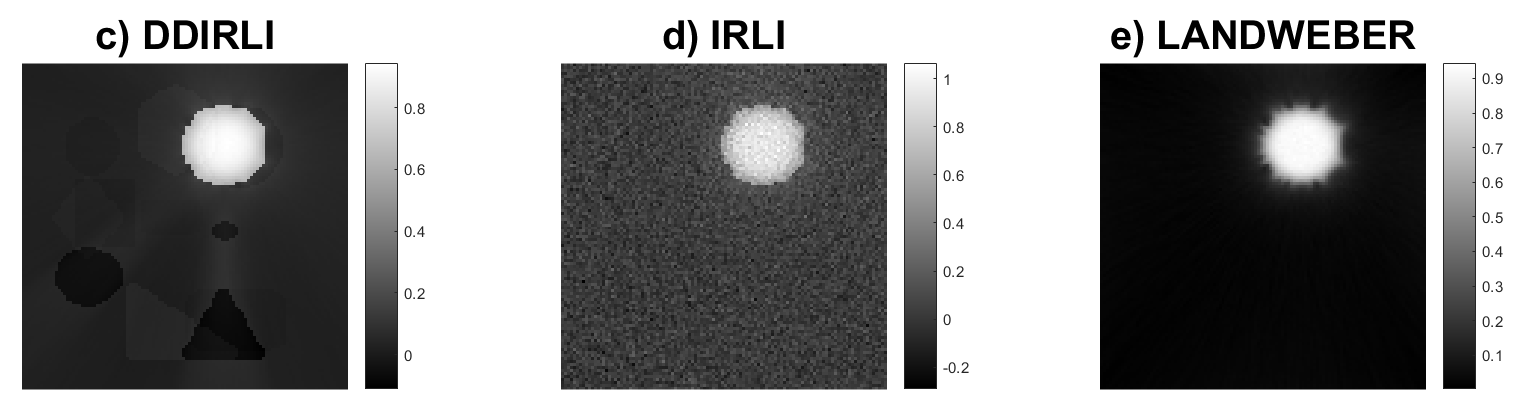}}
	\caption{Test 1S. Reconstruction of an image which has been used to build the matrix $A$, only with a region with a positive value. a) Image to be reconstructed; b) Plot of the noise added in the synthetic data, which are aggregated in a matrix form; c)-e) Reconstructions by the different methods.}\label{fig: Stest1}
\end{figure}
\newpage
\begin{figure}[!h]
	\centering
	\vspace{1cm}
	\includegraphics[scale=0.38]{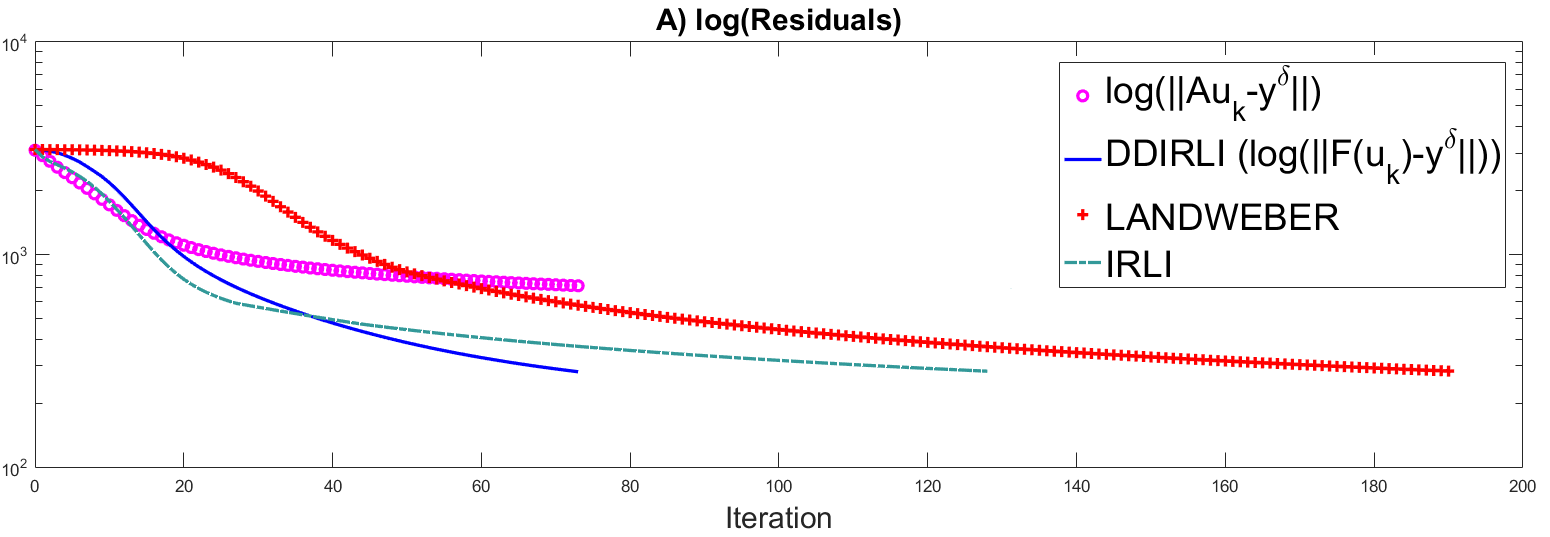}\vspace{1cm}
	\includegraphics[scale=0.38]{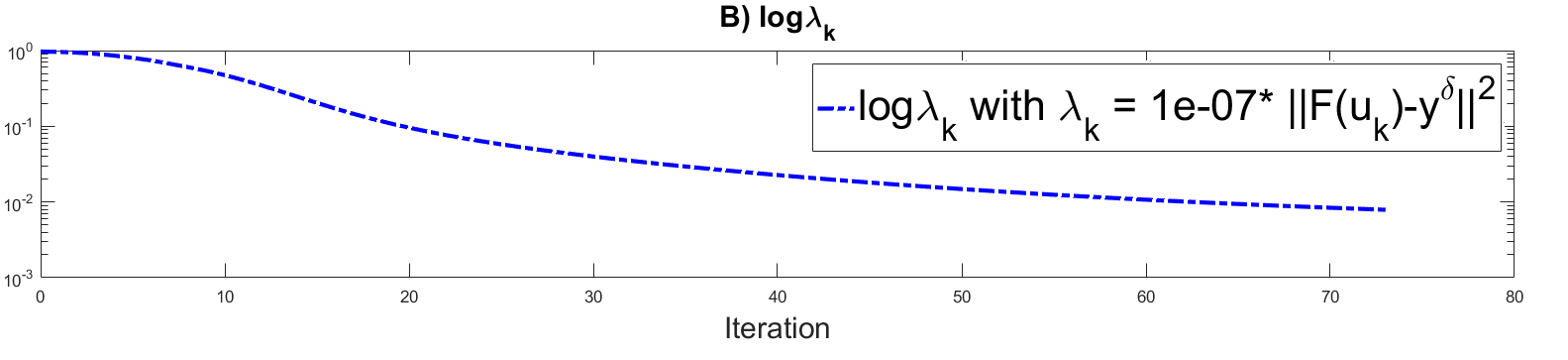}\vspace{1cm}
	\includegraphics[scale=0.38]{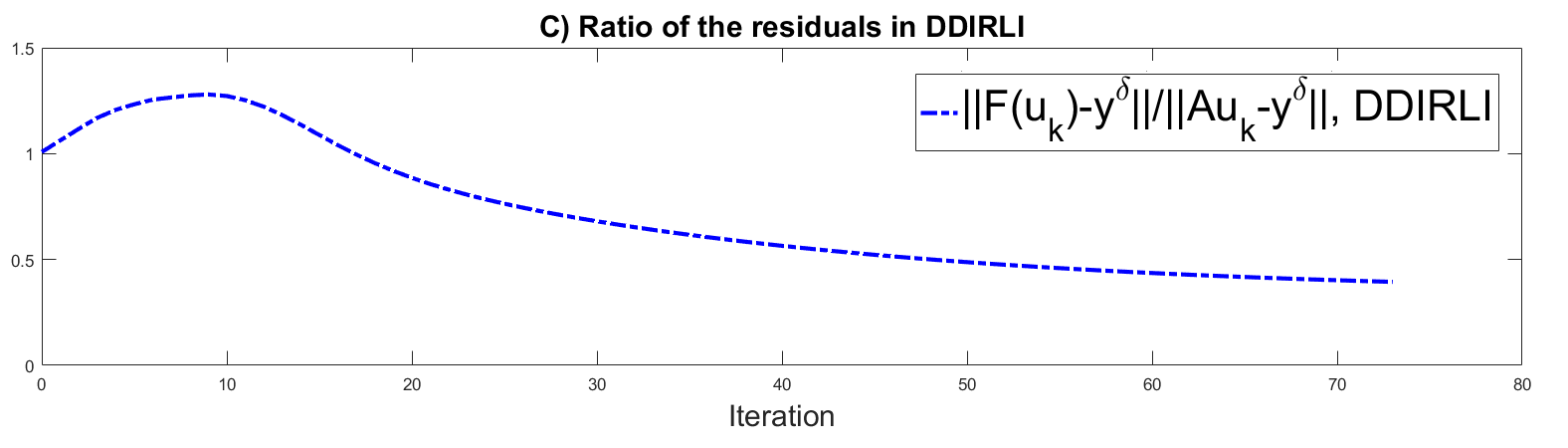}\vspace{1cm}
	\includegraphics[scale=0.38]{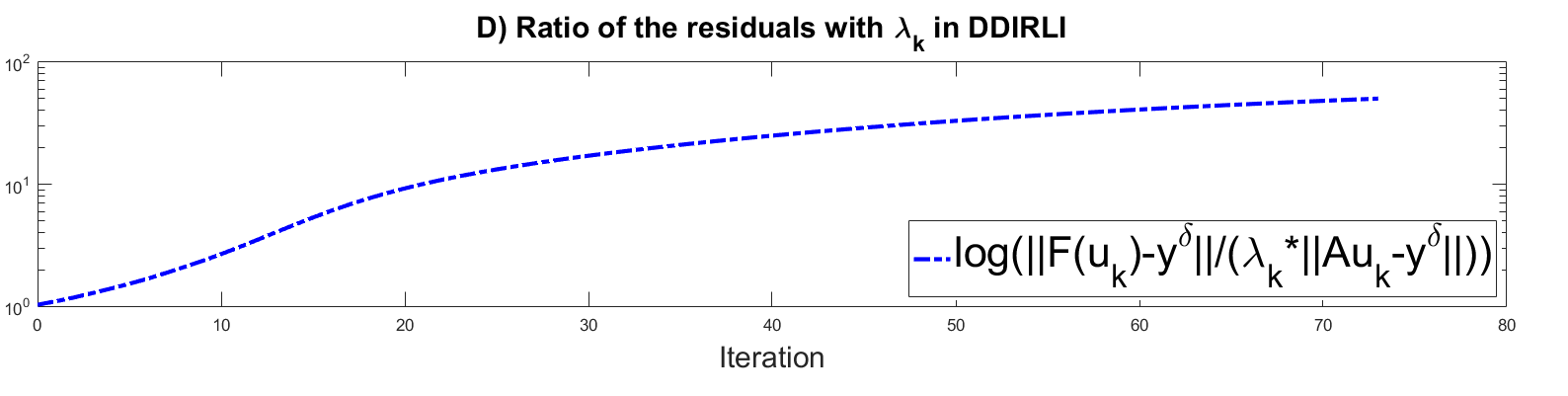}
	\caption{Test 1S. A) Residual at each step of schemes \eqref{eq:ModLand_Schl}, \eqref{eq:Land_Schl} and \eqref{eq:IRLI_Land_Schl}; B) values of $\lambda_{k}$ in \eqref{eq:ModLand_Schl} scheme at each iteration; C) ratio of the residuals $\norm{F(u_k)-\yd}_{\infty}/\norm{Au_k-\yd}_{\infty}$ at each iteration; D) plot of the ratio $\norm{F(u_k)-\yd}_{\infty}/(\lambda_{k} \norm{Au_k-\yd}_{\infty})$.}\label{fig: Stest1_2}
\end{figure}

\newpage
{\begin{table}[!h]\caption{Test 2aS. Left part: Parameters used in the test. Right part: some of the results of the test.}\label{tab: table2aS}
		\begin{center}
			{\renewcommand{\arraystretch}{1.2}
				\begin{tabular}{|l|c|c|c||c|c|c|}
					\cline{1-7} & & & & & & \vspace{-0.2cm} \\
					Method & $\sigma^2$-noise& $\delta$ & $\tau$ & Iterations & Comp. Time (s) & $\frac{\norm{u_{\textrm{true}} - u_{\textrm{rec}}}_2}{\norm{u_{\textrm{true}}}_2}$\\ \hline  
					{DDIRLI} &  &  &  & 42 & 20,0855993 & 0,18951741
					\\ \cline{1-1}\cline{5-7}
					{IRLI} & 100  & 144,6648313
					& 3  & 66 & 4,3175441
					& 0,30307348
					\\ \cline{1-1}\cline{5-7} 
					LANDWEBER &  & &  & 156 &9,0376626
					& 1,802786044
					\\ \hline
				\end{tabular}
			}
		\end{center}   
	\end{table} 
}
\vspace{2cm}
\begin{figure}[!h]
	\centering
	\subfigure{
		\includegraphics[scale=0.35]{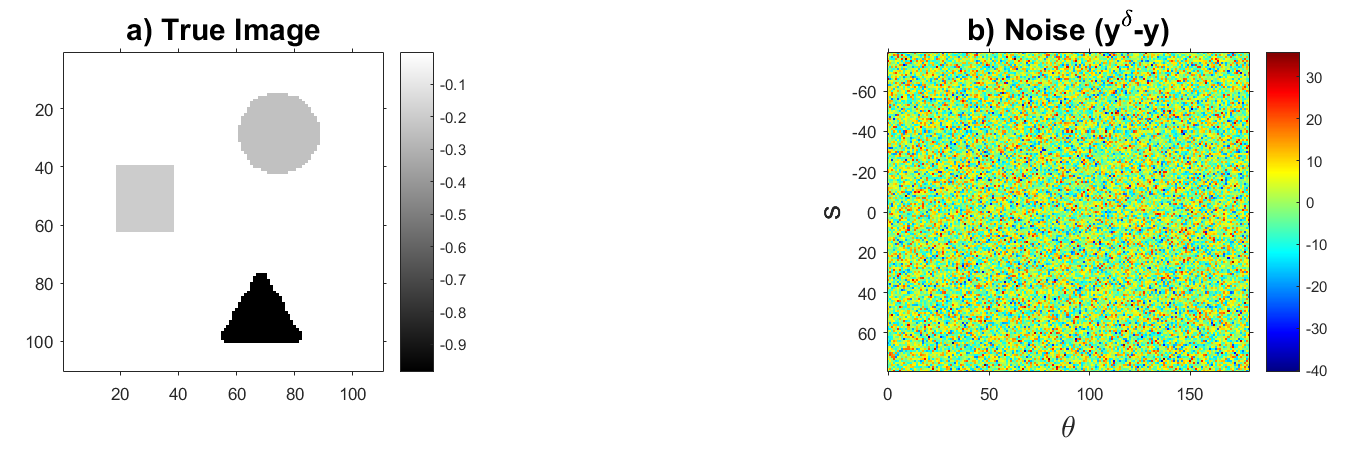}}
	\subfigure{\includegraphics[scale=0.37]{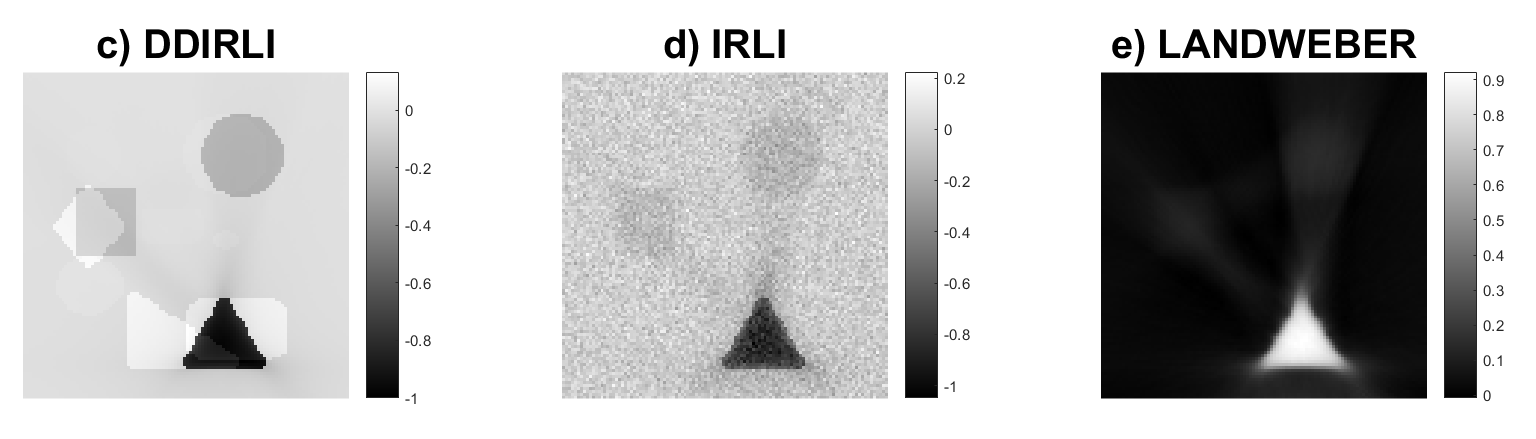}}
	\caption{Test 2aS. Reconstruction of an image which has been used to build the matrix $A$, only with regions with negative values. a) Image to be reconstructed; b) Plot of the noise added in the synthetic data, which are aggregated in a matrix form; c)-e) Reconstructions by the different methods.}\label{fig: Stest2a}
\end{figure}
\newpage
\begin{figure}[!h]
	\centering
	\vspace{1cm}
	\includegraphics[scale=0.38]{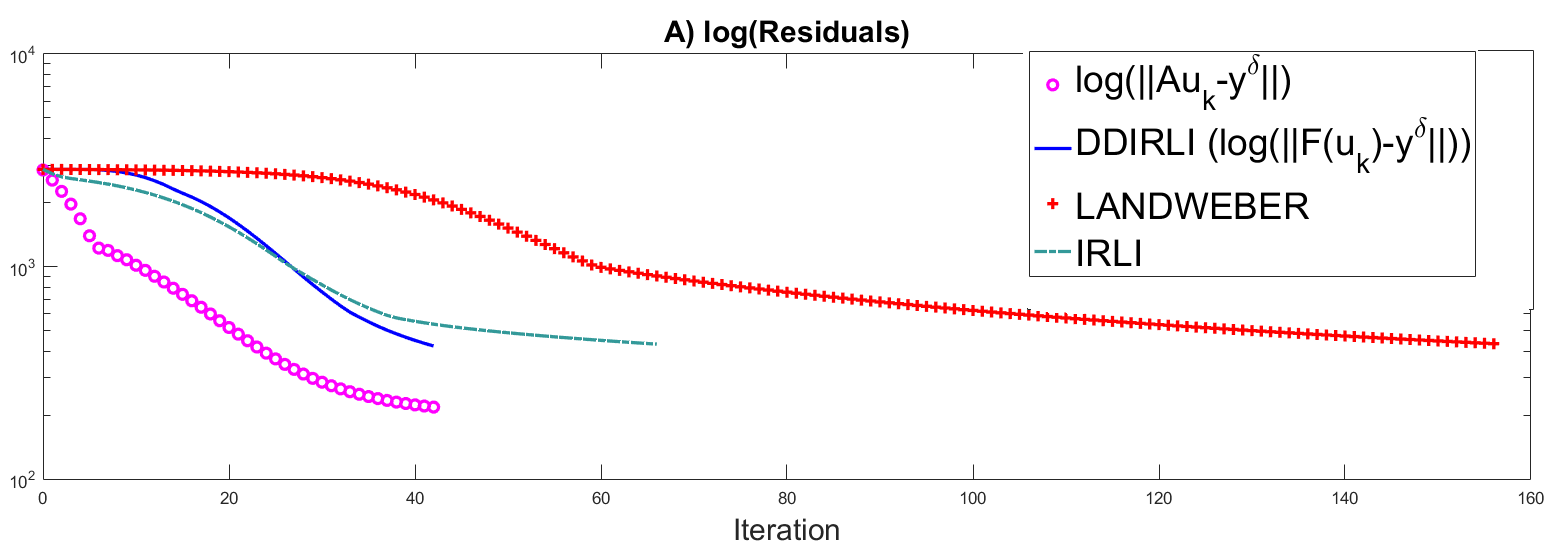}\vspace{1cm}
	\includegraphics[scale=0.38]{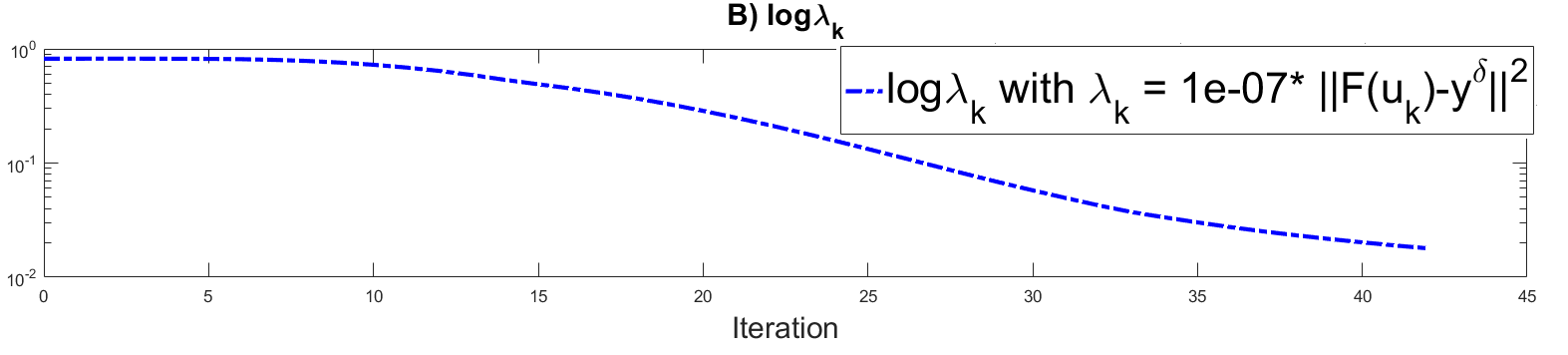}\vspace{1cm}
	\includegraphics[scale=0.38]{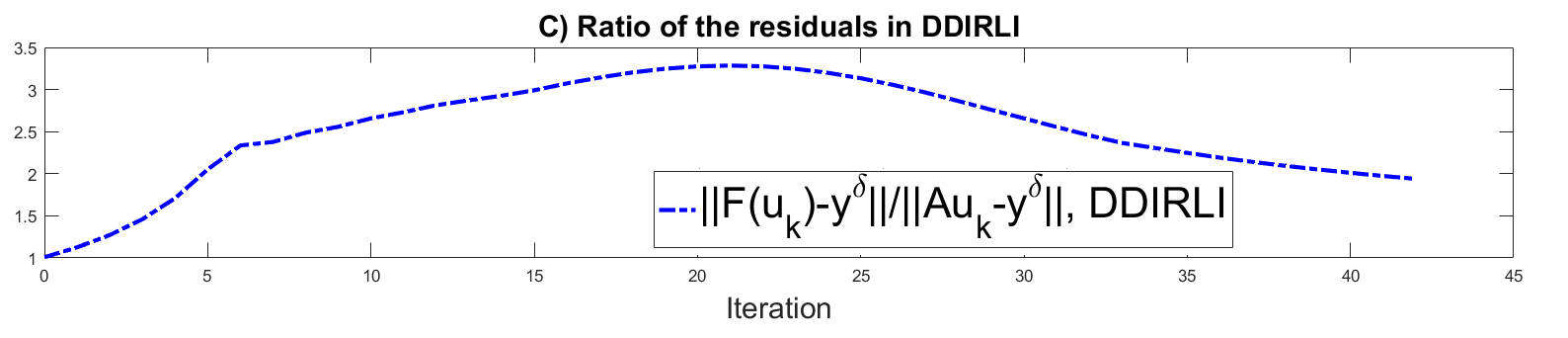}\vspace{1cm}
	\includegraphics[scale=0.38]{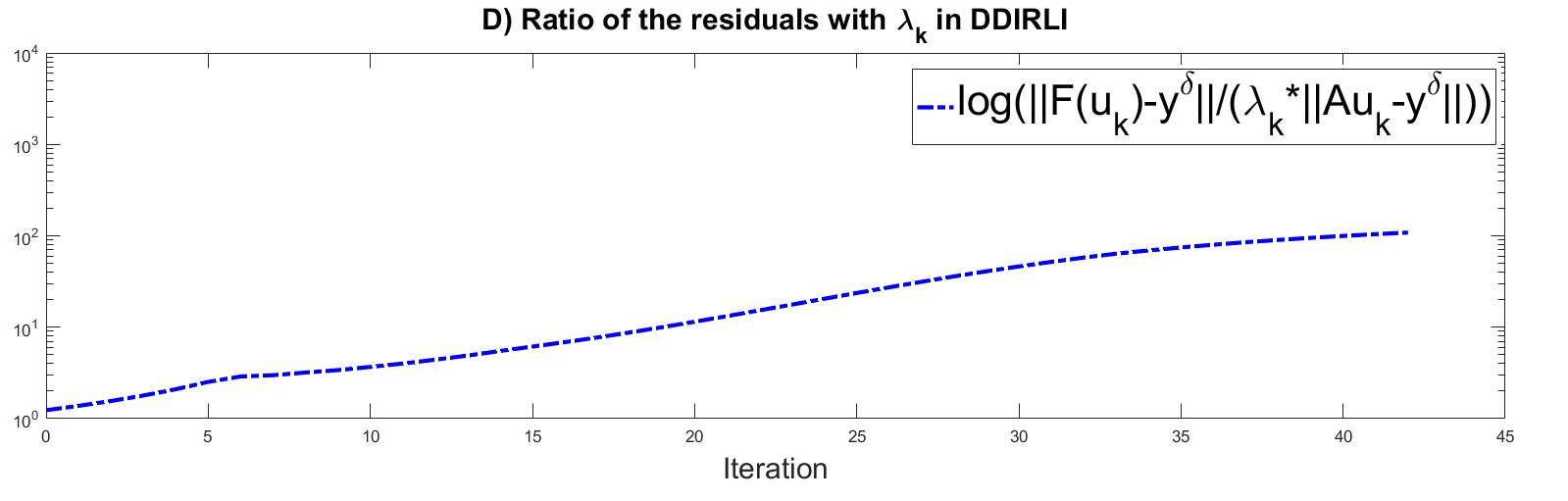}
	\caption{Test 2aS. A) Residual at each step of schemes \eqref{eq:ModLand_Schl}, \eqref{eq:Land_Schl} and \eqref{eq:IRLI_Land_Schl}; B) values of $\lambda_{k}$ in \eqref{eq:ModLand_Schl} scheme at each iteration; C) ratio of the residuals $\norm{F(u_k)-\yd}_{\infty}/\norm{Au_k-\yd}_{\infty}$ at each iteration; D) plot of the ratio $\norm{F(u_k)-\yd}_{\infty}/(\lambda_{k} \norm{Au_k-\yd}_{\infty})$.}\label{fig: Stest2a_2}
\end{figure}

\newpage
{\begin{table}[!h]\caption{Test 2bS. Left part: Parameters used in the test. Right part: some of the results of the test.}\label{tab: table2bS}
		\begin{center}
			{\renewcommand{\arraystretch}{1.2}
				\begin{tabular}{|l|c|c|c||c|c|c|}
					\cline{1-7} & & & & & & \vspace{-0.2cm} \\
					Method & $\sigma^2$-noise& $\delta$ & $\tau$ & Iterations & Comp. Time (s) & $\frac{\norm{u_{\textrm{true}} - u_{\textrm{rec}}}_2}{\norm{u_{\textrm{true}}}_2}$\\ \hline  
					{DDIRLI} &  &  &  & 84 & 37,8148317 & 0,144113319
					\\ \cline{1-1}\cline{5-7}
					{IRLI} & 10  & 44,46384766
					& 6  & 399 & 24,3457831
					& 0,262173351
					\\ \cline{1-1}\cline{5-7} 
					LANDWEBER &  & &  & 399 &24,076839
					& 1,77556498
					\\ \hline
				\end{tabular}
			}
		\end{center}   
	\end{table} 
}
\vspace{2cm}
\begin{figure}[!h]
	\centering
	\subfigure{
		\includegraphics[scale=0.35]{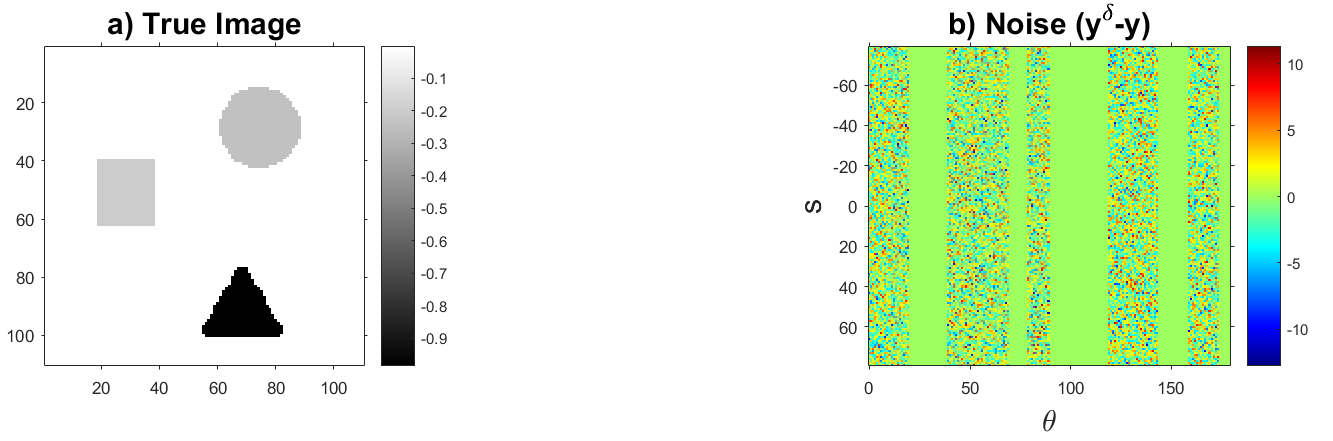}}
	\subfigure{\includegraphics[scale=0.37]{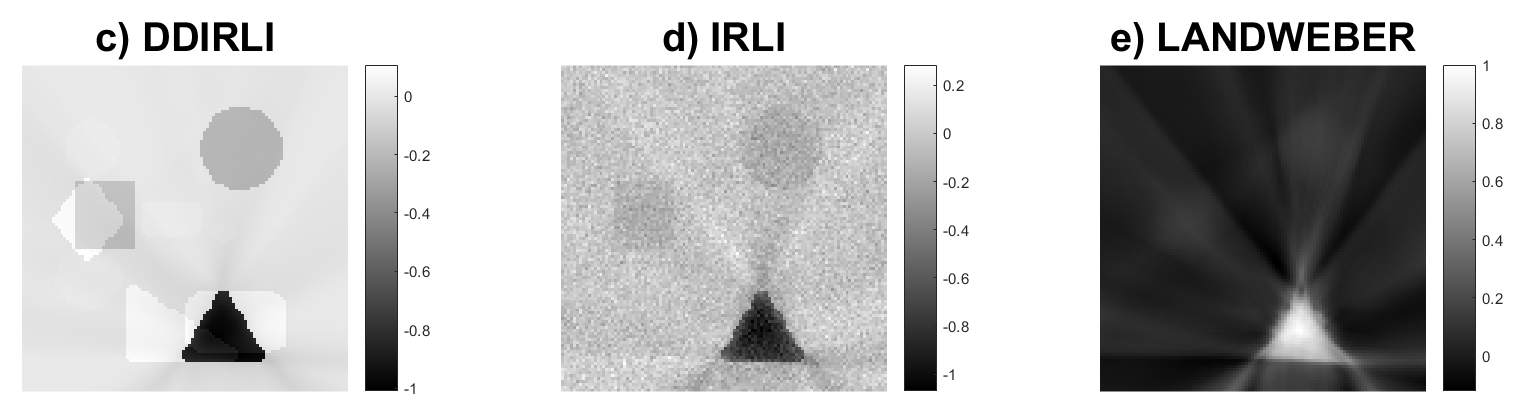}}
	\caption{Test 2bS. Reconstruction of an image which has been used to build the matrix $A$, from partial data and only with regions with negative values. a) Image to be reconstructed; b) Plot of the noise added in the synthetic data, which are aggregated in a matrix form; c)-e) Reconstructions by the different methods.}\label{fig: Stest2b}
\end{figure}
\newpage
\begin{figure}[!h]
	\centering
	\vspace{1cm}
	\includegraphics[scale=0.38]{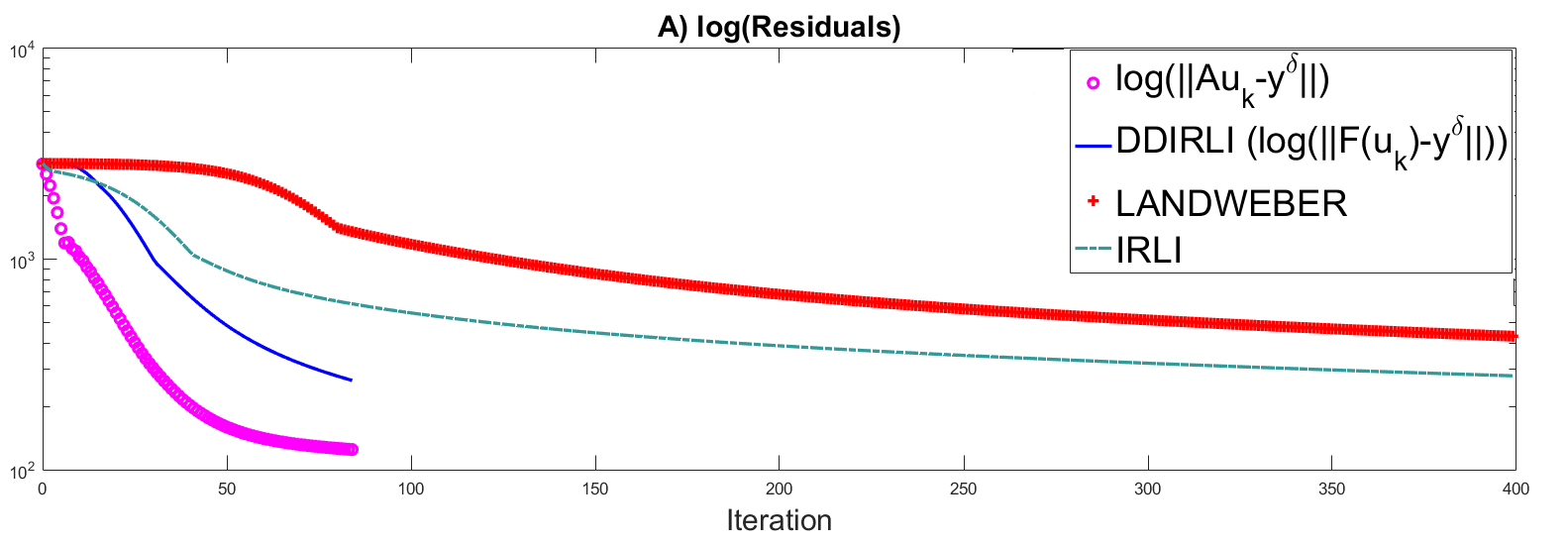}\vspace{1cm}
	\includegraphics[scale=0.38]{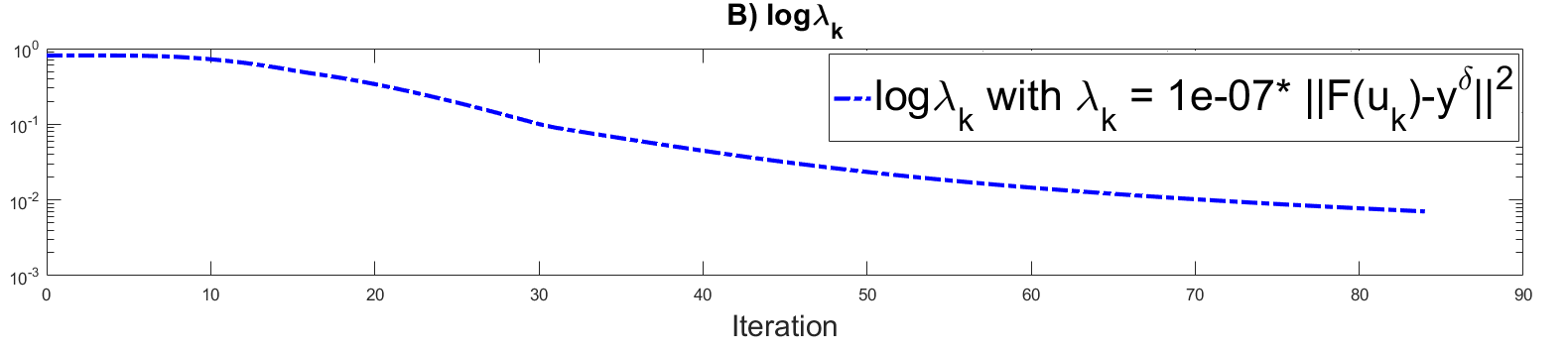}\vspace{1cm}
	\includegraphics[scale=0.38]{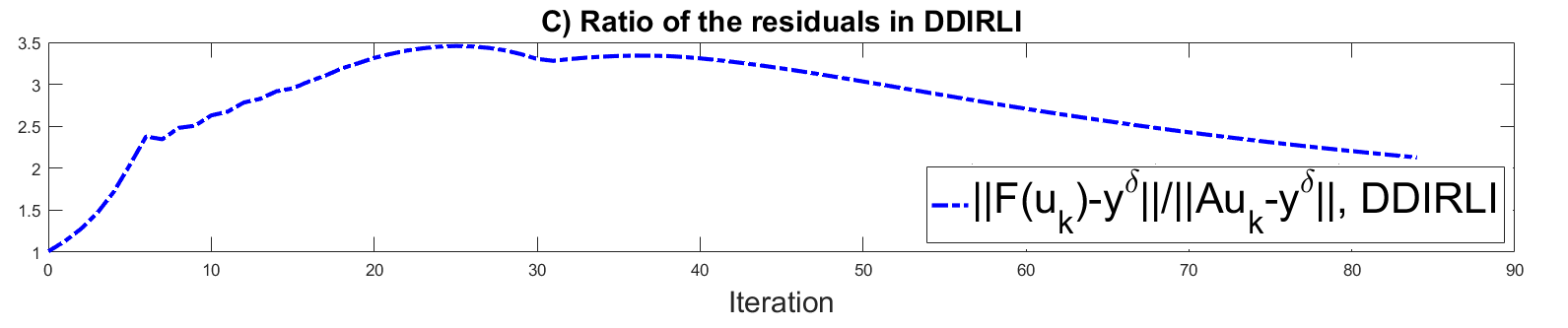}\vspace{1cm}
	\includegraphics[scale=0.38]{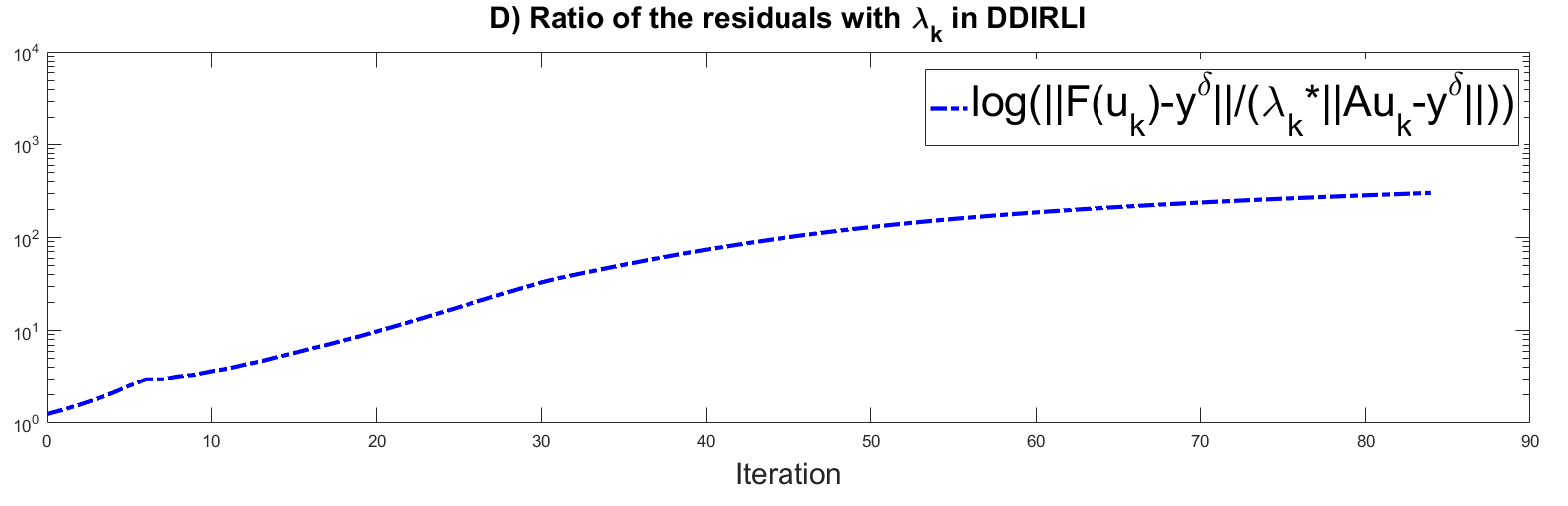}
	\caption{Test 2bS. A) Residual at each step of schemes \eqref{eq:ModLand_Schl}, \eqref{eq:Land_Schl} and \eqref{eq:IRLI_Land_Schl}; B) values of $\lambda_{k}$ in \eqref{eq:ModLand_Schl} scheme at each iteration; C) ratio of the residuals $\norm{F(u_k)-\yd}_{\infty}/\norm{Au_k-\yd}_{\infty}$ at each iteration; D) plot of the ratio $\norm{F(u_k)-\yd}_{\infty}/(\lambda_{k} \norm{Au_k-\yd}_{\infty})$.}\label{fig: Stest2b_2}
\end{figure}

\newpage
{\begin{table}[!h]\caption{Test 3S. Left part: Parameters used in the test. Right part: some of the results of the test.}\label{tab: table3aS}
		\begin{center}
			{\renewcommand{\arraystretch}{1.2}
				\begin{tabular}{|l|c|c|c||c|c|c|}
					\cline{1-7} & & & & & & \vspace{-0.2cm} \\
					Method & $\sigma^2$-noise& $\delta$ & $\tau$ & Iterations & Comp. Time (s) & $\frac{\norm{u_{\textrm{true}} - u_{\textrm{rec}}}_2}{\norm{u_{\textrm{true}}}_2}$\\ \hline  
					{DDIRLI} &  &  &  & 213 & 94,3754666 & 0,254685973
					\\ \cline{1-1}\cline{5-7}
					{IRLI} & 100  & 145,8908582
					& 3  & 83 & 4,7161903
					& 0,475122924
					\\ \cline{1-1}\cline{5-7} 
					LANDWEBER &  & &  & 151 &8,5576423
					& 1,372143804
					\\ \hline
				\end{tabular}
			}
		\end{center}   
	\end{table} 
}
\vspace{2cm}
\begin{figure}[!h]
	\centering
	\subfigure{
		\includegraphics[scale=0.35]{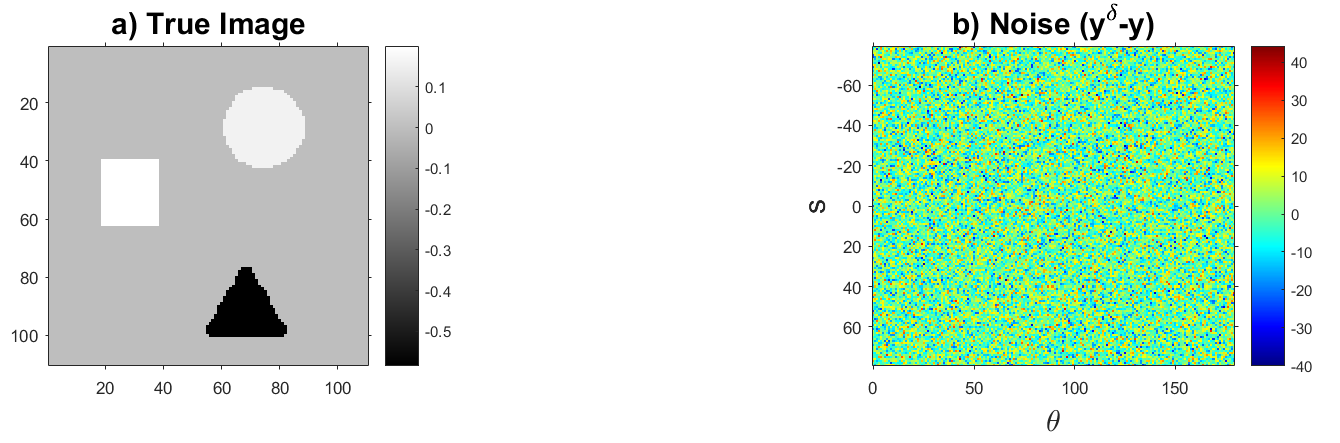}}
	\subfigure{\includegraphics[scale=0.37]{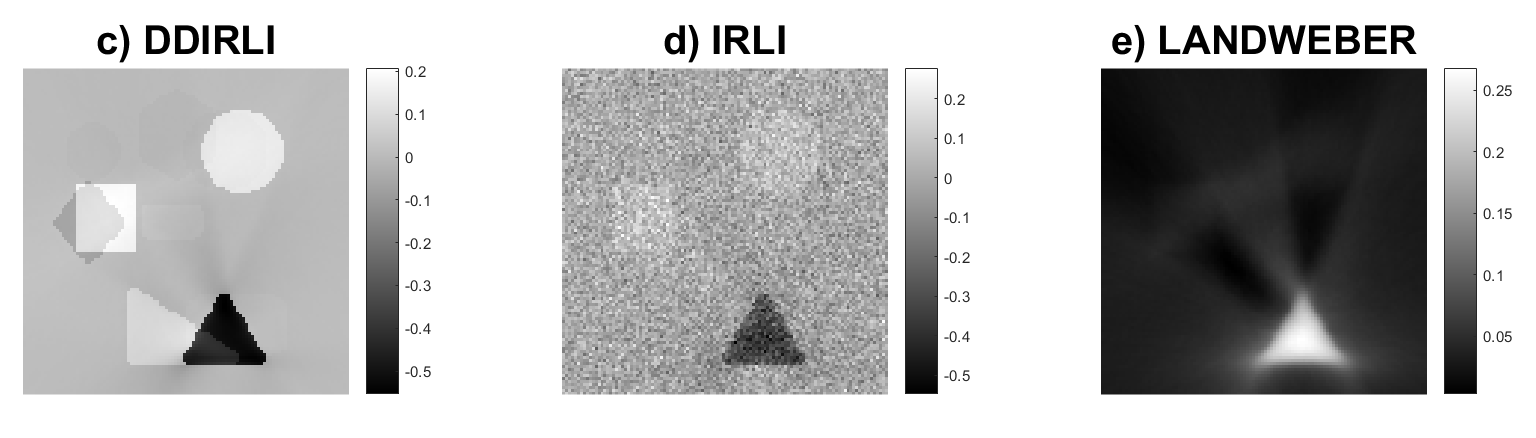}}
	\caption{Test 3S. Reconstruction of an image which has been used to build the matrix $A$, with regions with positive and negative values. a) Image to be reconstructed; b) Plot of the noise added in the synthetic data, which are aggregated in a matrix form; c)-e) Reconstructions by the different methods.}\label{fig: Stest3}
\end{figure}
\newpage
\begin{figure}[!h]
	\centering
	\vspace{1cm}
	\includegraphics[scale=0.38]{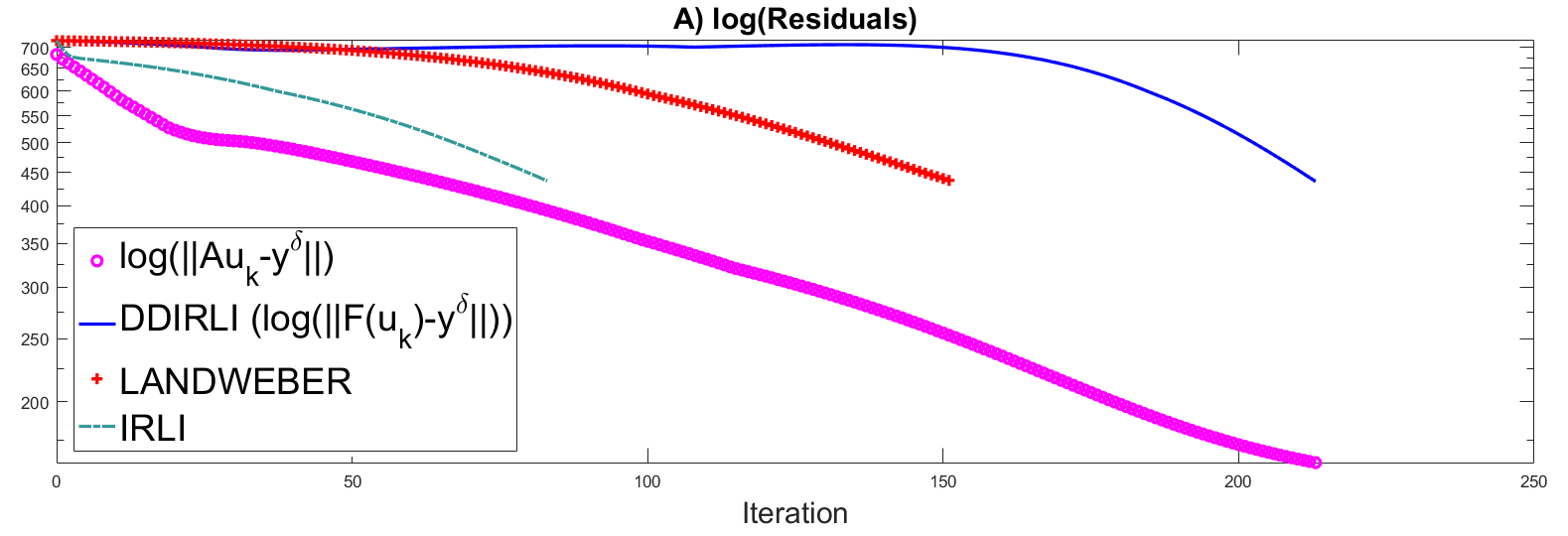}\vspace{1cm}
	\includegraphics[scale=0.38]{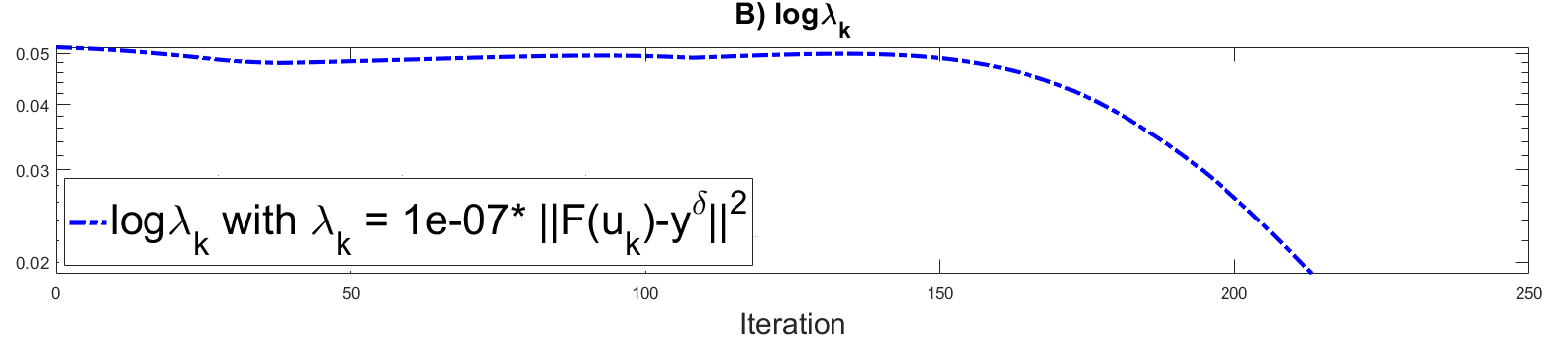}\vspace{1cm}
	\includegraphics[scale=0.38]{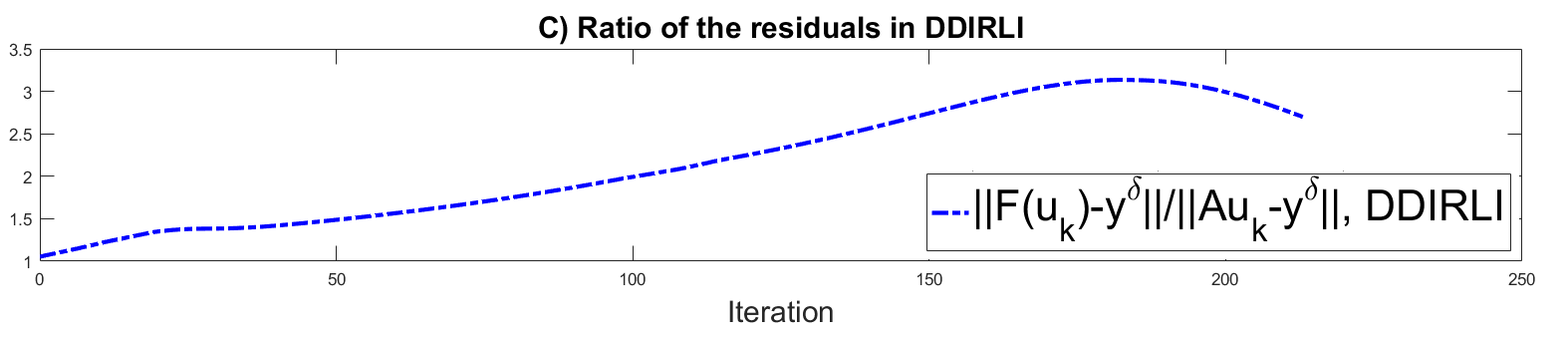}\vspace{1cm}
	\includegraphics[scale=0.38]{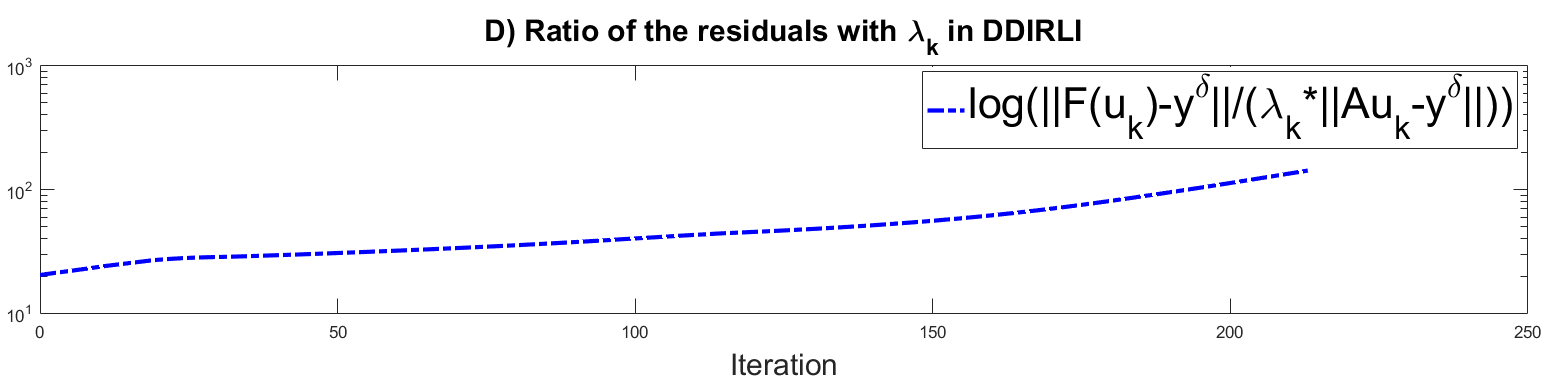}
	\caption{Test 3S. A) Residual at each step of schemes \eqref{eq:ModLand_Schl}, \eqref{eq:Land_Schl} and \eqref{eq:IRLI_Land_Schl}; B) values of $\lambda_{k}$ in \eqref{eq:ModLand_Schl} scheme at each iteration; C) ratio of the residuals $\norm{F(u_k)-\yd}_{\infty}/\norm{Au_k-\yd}_{\infty}$ at each iteration; D) plot of the ratio $\norm{F(u_k)-\yd}_{\infty}/(\lambda_{k} \norm{Au_k-\yd}_{\infty})$.}\label{fig: Stest3_2}
\end{figure}

\section{Conclusion} \label{sec:conclusion}
In this paper we introduced a new variant of the iteratively regularized Landweber iteration, where additional regularization 
is enforced from training data. The data driven regularization term is determined by finding a linear system from the 
input-output relation of some training data. 

This strategy is uneconomical in terms of the amount of usable training data, stability and matrix storage capacity, 
and this asks for more advanced methods of learning (see for example \cite{AdlOek17,GoodBenCour16}) instead of a linear black-box strategy. For some literature on deep learning in inverse problems,  see for example \cite{AdlOek17,AntHalSch18,BubKutLas19,DrPerHal18} and references therein. {We believe that the issue, of the use of a large amount of storage
	capacity for the creation of the matrix A, could be overcome by using other
	methods of numerical linear algebra which are equally valid and
	precise compared to the singular value decomposition. We leave this issue for future works.}
Looking closer to our proposed algorithm it averages data driven regularization and physical model terms in such a way 
that in the beginning of the iteration the data driven term is dominant, while during the iteration the physical model 
takes over the leading role. By this sense we make sure that for data with little noise we are always close to the physical 
solution. Our numerical experiments show that the iteration always reconstructs an image where the residuum is close to the 
measured data. Moreover, if the desired solution is well approximated by the training data, the 
reconstruction with the proposed algorithm always works better than the standard Landweber iteration and {the iteratively regularized Landweber iteration}. 
This, in particular, applies to severely ill--posed problems of limited angle tomography. A significant challenge are 
nonlinear inverse problems, such as the problem of Schlieren tomography. The inherent non-uniqueness of the solution 
of the Schlieren problem gives more emphasis on selecting the right a--priori choice, which our flexible 
regularization term can handle more efficiently than existing theory. {In the nonlinear case, one can observe various artifacts
	in the reconstructions due to the linearity of $A$, which are some of the images contained in $A$. Therefore, the use of a nonlinear operator $A$ should be investigated. We leave this issue for future works.}

\section*{Acknowledgments}
A. Aspri thanks Y. Korolev for suggesting the dataset MNIST, used in the first part of the paper. O. Scherzer is supported by the FWF via the projects I3661-N27 (Novel Error Measures and Source Conditions of Regularization Methods for Inverse Problems) and via SFB F68, project F6807-N36 (Tomography with Uncertainties).

\appendix



\end{document}